\documentclass[a4paper,12pt,twoside]{amsart}
\usepackage{amsmath}
\usepackage{amsthm}
\usepackage{amssymb}
\theoremstyle{plain}
\newtheorem{thm}{\indent\bf Theorem}[section]
\newtheorem{lem}[thm]{\indent\bf Lemma}
\newtheorem{prop}[thm]{\indent\bf Proposition}
\newtheorem{cor}[thm]{\indent\bf Corollary}
\theoremstyle{definition}
\newtheorem{rem}{\indent\it Remark}[section]
\newtheorem{exa}{\indent\it Example}[section]
\numberwithin{equation}{section}
\def \re {\mathrm{Re\,}}
\def \im {\mathrm{Im\,}}
\def \diag {\mathrm{diag\,}}

\begin{document}
\title[Degenerate third Painlev\'e transcendents]{Boutroux ansatz for 
the degenerate third Painlev\'e transcendents} 
\author[Shun Shimomura]{Shun Shimomura} 
%%%%%%%%%%%%%%%%%%%%%%%%%%%%%%%%%%%%%%%%%%%%%
%%%%%%%%%%%%%%%%%%%%%%%%%%%%%%%%%%%%%%%%%%%%%%
\address{Department of Mathematics, 
Keio University, 
\endgraf
3-14-1, Hiyoshi, Kohoku-ku,
Yokohama 223-8522 
Japan
\endgraf
{\tt shimomur@math.keio.ac.jp}
}
\date{}
%%%%%%%%%%%%%%%%%%%%%%%%%%%%%%%%%%%%%%%
% \subjclass[2010]{34M55, 34M56, 34M35}
% \keywords{Fifth Painlev\'{e} equation; critical behaviour;
% isomonodromy deformation; monodromy data; Schlesinger equation}
%%%%%%%%%%%%%%%%%%%%%%%%%%%%%%%%%%%%%%%
\begin{abstract}
For a general solution of the degenerate third Painlev\'e equation 
we show the Boutroux ansatz near the point at infinity. 
It admits an asymptotic
representation in terms of the Weierstrass pe-function in cheese-like
strips along generic directions. The expression is obtained by using
isomonodromy deformation of a linear system governed by the degenerate
third Painlev\'e equation. 
\vskip0.2cm
\par
2010 {\it Mathematics Subject Classification.} {34M55, 34M56, 34M40, 34M60,
33E05.}
\par
{\it Key words and phrases.} {degenerate third Painlev\'{e} equation; 
Boutroux ansatz; Boutroux equations; isomonodromy deformation; WKB analysis;
pe-function; theta-function.} 
\end{abstract}
\maketitle
%\ams{34M55}
\allowdisplaybreaks
%%%%%%%%%%%%%%%%%%%%%%%%%%%%%%%%%%%%%
%%%%%% section 1 %%%%%%%
\section{Introduction}\label{sc1}
%%%%%%%%%%%%%%%%%%%%%%%%%%%%%%%
%%%%%%%%%%%%%%%%%%%%%%%%%%%%%%%%%%
In the geometrical study of the spaces of initial values
for Painlev\'e equations, Sakai \cite{Sa} classified the third
Painlev\'e equations into three types $P_{\mathrm{III}}(D_6)$,
$P_{\mathrm{III}}(D_7)$ and $P_{\mathrm{III}}(D_8)$. 
For the types $P_{\mathrm{III}}(D_7)$ and $P_{\mathrm{III}}(D_8)$ 
Ohyama et al.~\cite{Oh} examined basic properties including $\tau$-functions,
irreducibility, the spaces of initial values. Equation $P_{\mathrm{III}}
(D_8)$ is changed into a special case of $P_{\mathrm{III}}(D_6)$.
Equation $P_{\mathrm{III}}(D_7)$ is called the degenerate third Painlev\'e
equation or degenerate $P_{\mathrm{III}}$, which may be normalised in the 
form  
\begin{equation*}
v_{\xi\xi}=\frac{v_{\xi}^2}{v} -\frac{v_{\xi}}{\xi} -\frac{2v^2}{\xi^2}
 +\frac a{\xi} + \frac 1v
\end{equation*}
$(v_{\xi}=dv/d\xi)$ with $a\in\mathbb{C}.$
The change of variables
$$
  2\xi= \epsilon b \tau^2, \quad  v=\epsilon\tau u
$$
takes this equation to the equivalent equation discussed in \cite{KV1},
\cite{KV2}
%%%%%% (1.1) %%%%%%%%%%%%%%%%%
\begin{equation}\label{1.1}
u_{\tau\tau}=\frac{u_{\tau}^2}{u} -\frac{u_{\tau}}{\tau} +\frac 1{\tau}
(-8\epsilon u^2 +2ab) +\frac{b^2}{u}
\end{equation}
with $\epsilon=\pm 1,$ $a\in\mathbb{C},$ $b\in \mathbb{R}\setminus\{0\}$,
which governs isomonodromy deformation of linear system \eqref{3.1}.
Using isomonodromy system \eqref{3.1}, Kitaev and Vartanian 
\cite{KV1}, \cite{KV2} obtained asymptotic solutions of \eqref{1.1} 
as $\tau \to \pm\infty$, $\pm i \infty$ and $\tau \to \pm 0$, $\pm i 0,$ 
with connection formulas among them. Furthermore, for \eqref{1.1}, a special
meromorphic solution is studied by \cite{Kitaev-4}, \cite{KV3}, and
truncated solutions are given by \cite{Vartanian}. 
\par
As mentioned in \cite{KV1}, \cite{Vartanian}, in physical and geometrical 
applications, degenerate $P_{\mathrm{III}}$ appears in contexts 
independent of $P_{\mathrm{III}}(D_6)$, i.e. complete $P_{\mathrm{III}}$, 
and its significant analytic properties are important.
Indeed behaviours of solutions of \eqref{1.1} along real and imaginary axes
\cite{KV1}, \cite{KV2} are quite different from those for complete
$P_{\mathrm{III}}$ \cite{K-Sb}. 
For complete $P_{\mathrm{III}}$ of the Sine-Gordon type, Novokshenov 
\cite{Novokshenov-3}, \cite{Novokshenov-4}, \cite[Chap.~16]{FIKN} 
provided an asymptotic 
representation of solutions in terms of the $\mathrm{sn}$-function 
along generic directions near the point at infinity. It is meaningful 
to establish the counterpart of this expression for degenerate 
$P_{\mathrm{III}}$. 
\par
In this paper we show the Boutroux ansatz \cite{Boutroux} for degenerate 
$P_{\mathrm{III}}$, that is, present an elliptic asymptotic 
representation for a general solution along generic directions 
near the point at infinity. 
The main results are described in Section \ref{sc2}.
As in Theorems \ref{thm2.1} and \ref{thm2.2}, degenerate $P_{\mathrm{III}}$ 
admits a general solution written in terms of the Weierstrass $\wp$-function,
and so does $P_{\mathrm{I}}$ (\cite{Joshi-Kruskal-1}, \cite{Joshi-Kruskal-2}, 
\cite{Kitaev-2}, \cite{Kitaev-3}). 
On the other hand for $P_{\mathrm{II}},$ $P_{\mathrm{IV}}$, $P_{\mathrm{III}}
(D_6)$ (of Sine-Gordon type) and $P_{\mathrm{V}}$, elliptic asymptotic solutions
are given by the $\mathrm{sn}$-function (\cite{FIKN}, \cite{Kapaev-1}, 
\cite{Kapaev-2},
\cite{Kapaev-3}, \cite{Kitaev-3}, \cite{Novokshenov-1}, \cite{Novokshenov-2},
\cite{Novokshenov-3}, \cite{Novokshenov-4},
\cite{Shimom}, \cite{Vere}). This fact reflects
the position of degenerate $P_{\mathrm{III}}$, i.e. $P_{\mathrm{III}}(D_7)$
in the degeneration scheme of the Painlev\'e equations \cite{Oh}, \cite{Sa}.
\par
For our purpose it is appropriate to treat an equation of the form
%%%%%% (1.2) %%%%%%%%%%%%%%%%%%%
\begin{equation}\label{1.2}
y''=\frac {(y')^2}y -\frac{y'}x -2y^2 + \frac{3a}x +\frac 1y
\end{equation}
$(y'=dy/dx)$, which comes from \eqref{1.1} via the substitution
%%%%%%%%%%%%%%%%%%%%%%%%%%%%%%%
\begin{equation}\label{1.*}
\epsilon\tau u= (x/3)^2y, \quad  \epsilon b \tau^2 = 2 (x/3)^3.
\end{equation}
%%%%%%%%%%%%%%%%%%%
Equation \eqref{1.2} with $x=e^{i\phi}t$ governs isomonodromy deformation of 
the linear system   
%%%%%% (1.3) %%%%%%%%%%%%%%%%%%%%%
\begin{align}\label{1.3}
\frac{d\Psi}{d\lambda} = & \frac t 3 \mathcal{B}(\lambda,t) \Psi,
\\
\notag
\mathcal{B}  (\lambda,t) =& -i e^{i\phi} \lambda \sigma_3 
+ \begin{pmatrix}  0  & -2 ie^{i\phi}y \\ 
\Gamma_0(t,y,y^t)/y  & 0    \end{pmatrix}
\\
\notag
& -(\Gamma_0(t,y,y^t) + 3 (1/2+ ia)t^{-1})\lambda^{-1} \sigma_3 + 2 e^{i\phi}
\begin{pmatrix} 0 & i \\ i & 0 \end{pmatrix} \lambda^{-2},
\end{align}
where $y$ and $y^t$ are arbitrary complex parameters, and
$$
 \Gamma_0(t,y,y^t) =\frac {y^t}{y} -\frac{i e^{i\phi}}y - (1+3ia) t^{-1},\quad
\sigma_3 = \begin{pmatrix} 1 & 0 \\ 0 & -1 \end{pmatrix}.
$$
%%%%%%%%%%%%%%%%%%%%%%%%%%%%%%%%%%%%%%%%%%%%%%%
As shown in Section \ref{sc3} system \eqref{1.3} is a result of transformation
of system \eqref{3.1} treated in \cite{KV1}, \cite{KV2}. 
The isomonodromy deformation of \eqref{3.1} is governed by equation
\eqref{1.1}, and solutions of \eqref{1.1} are labelled by coordinates on the
monodromy manifold for \eqref{3.1} defined by Stokes matrices and a 
connection matrix $G=(g_{ij})\in SL_2(\mathbb{C})$ for matrix solutions 
around $\mu=0$ and $\mu=\infty.$
System \eqref{1.3} admits the same monodromy 
manifold as of \eqref{3.1}, which is described by the same Stokes matrices
and $G$ for suitably chosen matrix solutions (cf. Proposition \ref{prop3.3}), 
so that solutions of \eqref{1.1} and \eqref{1.2} are 
labelled by the same monodromy data. 
\par
Applying WKB analysis we solve the direct monodromy problem for linear system
\eqref{1.3} in Section \ref{sc5}, and obtain key relations in 
Corollary \ref{cor5.2} containing the monodromy data $G$ and certain integrals, 
which lead to a solution of an inverse problem. 
Basic necessary materials for this calculation are summarised in 
Section \ref{sc4}. Asymptotic properties of these integrals are examined
in Section \ref{sc6} by the use of the $\vartheta$-function, and from these 
formulas asymptotic forms in the main theorems are derived in Section \ref{sc7}.
Then the justification as a solution of \eqref{1.2} is made along the line
of Kitaev \cite{Kitaev-1}, \cite{Kitaev-3}. The final section is devoted to
the Boutroux equations, which determine the modulus contained in the elliptic
representation of solutions.
\par
Throughout this paper we use the following symbols:
\par
(1) $\sigma_1,$ $\sigma_2,$ $\sigma_3$ are the Pauli matrices
$$
\sigma_1=\begin{pmatrix} 0 & 1 \\ 1 & 0 \end{pmatrix}, \quad
\sigma_2=\begin{pmatrix} 0 & -i \\ i & 0 \end{pmatrix}, \quad
\sigma_3=\begin{pmatrix} 1 & 0 \\ 0 & -1 \end{pmatrix}; 
$$
\par
(2) for complex-valued functions $f$ and $g$, we write $f \ll g$ or $g  \gg f$
if $f = O(|g|)$, and write $f\asymp g$ if $g \ll f \ll g.$
%%%%%%%%%%%%%%%%%%%%%%%%%%%%%%%%
%%%% Section 2 %%%%%%%%%%%%%%
%%%%%%%%%%%%%%%%%%%%%%%%%%%%%%%%
\section{Main results}\label{sc2}
%%%%%%%%%%%%%%%%%%%%%%%%%%%%%%%%%%%%
To state our main results we give some explanations on necessary facts.
%%%%%%%%%%%%%%%%%%%%%%%%%%
%%%%%% 2.1 %%%%%%%%%
\subsection{Monodromy data}\label{ssc2.1}
Isomonodromy system \eqref{3.1} admits the matrix solutions
$$
Y^{\infty}_k(\mu) =(I+O(\mu^{-1})) \mu^{-(1/2+ia)\sigma_3}
\exp(-i\tau \mu^2 \sigma_3 )
$$
as $\mu \to \infty$ through the sector  
$|\arg \mu +\arg \tau^{1/2} -\pi k/2| <\pi/2,$ and
$$
X^0_k(\mu)= ({i}/{\sqrt 2})\Theta_0^{\sigma_3} 
(\sigma_1+\sigma_3 +O(\mu))
\exp(-i \sqrt{\tau\epsilon b} \mu^{-1}\sigma_3 )
$$
as $\mu \to 0$ through the sector $|\arg \mu -\arg(\tau\epsilon
b)^{1/2} -\pi k|<\pi$, where $k\in \mathbb{Z}$ (see Section \ref{ssc3.2}). Let 
the invariant Stokes matrices and a connection matrix be such that
$Y^{\infty}_{j+1}(\mu) = Y^{\infty}_j(\mu)S^{\infty}_j$, 
$X^{0}_{j+1}(\mu) = X^{0}_j(\mu)S^{0}_j$ with $j\in \mathbb{Z}$ and that
$Y^{\infty}_0(\mu)=X^0_0(\mu) G.$ These are 
$$
S^{\infty}_0 =\begin{pmatrix} 1 & 0 \\ s^{\infty}_0 & 1 \end{pmatrix}, \quad
S^{\infty}_1 =\begin{pmatrix} 1 & s^{\infty}_1 \\ 0 & 1 \end{pmatrix}, \quad
S^{0}_0 =\begin{pmatrix} 1 & s^0_0 \\ 0 & 1 \end{pmatrix} 
$$
with \eqref{3.6}, and $G=(g_{ij})$ with $g_{11}g_{22}-g_{12}g_{21}=1$.
The monodromy manifold is given by $GS^{\infty}_0S^{\infty}_1 \sigma_3
e^{\pi(i/2-a)\sigma_3} =S^0_0 \sigma_1 G$, whose generic points admit 
the coordinates expressed by $G$ \cite{KV1}.
Solutions $u(\tau)$ of \eqref{1.1} and $y(x)$ of \eqref{1.2}
related via \eqref{1.*} are labelled by the same monodromy data.
%%%%%%%%%%%%%%%%%%%%%%%%%%%%%%%%%%
%%%%%% 2.2.%%%%%%%%%%%%%%
\subsection{Elliptic curve and Boutroux equations}\label{ssc2.2}
For $A\in \mathbb{C}$ around $A=3\cdot 2^{2/3}$ the polynomial $4z^3-Az^2+1$
has roots $z_0,$ $z_1$ close to $2^{-1/3}$ and $z_2$ close to $-4^{-2/3},$
and especially, $z_0=z_1=2^{-1/3},$ $z_2=-4^{-2/3}$ when $A=3\cdot 2^{2/3}.$
Let $\Pi_+$ and $\Pi_-$ be the copies of $P^1(\mathbb{C}) \setminus ([\infty,
z_2] \cup [z_0,z_1])$ and set $\Pi_A = \Pi_+ \cup \Pi_-$ glued  along the cuts
$[\infty, z_2]$ and $[z_0,z_1]$, where $\re z \to -\infty$ along 
$[\infty, z_2]$.
Then $\Pi_A$ is the elliptic curve given by 
$$
w(A,z)^2 =4z^3 -Az^2+1,
$$ 
where the
branch of $\sqrt{4z^3 -Az^2 +1}:= 2\sqrt {z-z_0} \sqrt{ z-z_1} \sqrt{z-z_2}$
is chosen in such a way that $\re \sqrt{z-z_j} \to +\infty$ as $z\to \infty$
along the positive real axis on the upper plane $\Pi_+.$
The elliptic curve $\Pi_A$ does not degenerate as long as $A\not= 3\cdot
2^{2/3} e^{2\pi i m/3}$ $(m=0, \pm 1),$ that is, $4z^3-Az^2+1$ has no double
roots, and then we may define $\Pi_A$ continuously. 
\par
As will be shown in Section \ref{sc8}, for any $\phi \in \mathbb{R}$, there 
exists $A_{\phi} \in \mathbb{C}$ with $\Pi_{A_{\phi}}$ such that, for 
every cycle $\mathbf{c}$ on $\Pi_{A_{\phi}}$ 
$$
\im e^{i\phi} \int_{\mathbf{c}} \frac{w(A_{\phi}, z)}{z^2} dz=0,
$$
and that $A_{\phi}$ has the properties (Proposition \ref{prop8.15}):
\par
(1) for every $\phi$, $A_{\phi}$ is uniquely determined;
\par
(2) $A_{\phi}$ is continuous in $\phi \in \mathbb{R},$ and is smooth in
$\phi \in \mathbb{R}\setminus \{k\pi/3 \,|\, k\in \mathbb{Z} \}$;
\par
(3) $A_{\phi \pm 2\pi/3}= e^{\pm 2\pi i/3} A_{\phi},$ $A_{\phi+\pi }
=A_{\phi},$ $A_{-\phi}=\overline{A_{\phi}};$
\par
(4) $\Pi_{A_{\phi}}$ degenerates if and only if $\phi=k\pi/3$ with $k\in
\mathbb{Z},$ and then $A_0= 3\cdot 2^{2/3},$ $A_{\pm \pi/3}=e^{\mp 2\pi i/3}
A_0,$ $A_{\pm 2\pi/3} = e^{\pm 2\pi i/3}A_0,$ $A_{\pm \pi}= A_0.$ 
\par
In particular, for $0<|\phi|<\pi/3$ let us consider $A_{\phi}$ for specified
cycles. For $A_{\phi}$ close to $A_0=3\cdot 2^{2/3},$ by Proposition 
\ref{prop8.16} number the roots of 
$w(A_{\phi},z)^2$ close to $2^{-1/3}$ in such a way that $\im z_0 \le
\im z_1$ if $\phi >0,$ (respectively, $\im z_1 \le \im z_0$ if $\phi <0$)
and let the numbering be retained as long as coalescence does not occur. 
Then for $0<|\phi|<\pi/3$ we have basic cycles $\mathbf{a}$ and $\mathbf{b}$  
on $\Pi_{A_{\phi}}$, which are drawn on $\Pi_+$ as in Figure \ref{cycles1}.
%%%%%%%%%%%%%%%%%%%%%%%%%%%%%%%%%%%%%%%%%%%%%%%%%%%%%%
%%%%%%%%%%%%%%%%%%%%%%  Figure 2.1 %%%%%%%%%%%%%%%%%%%%%%%%%%%%%%%%
{\small
\begin{figure}[htb]
\begin{center}
\unitlength=0.78mm
%%%%%%%%%%%%%%%%%%%%%%%%%%%%%%%%%%
%%%%%%%%%%%%%%%%%%%%%%%%%%%%%%%%%%
%%%% \phi <0 %%%%%%%%%
%%%%%%%%%%%%%%%%%%%%%%%
%%%%%%%%%%%%%%%%%%
\begin{picture}(80,50)(-10,-5)
 \put(-8,22){\makebox{$\infty$}}
 \put(20,12){\makebox{$z_2$}}
 \put(65,32){\makebox{$z_0$}}
 \put(53,2){\makebox{$z_1$}}
  \put(20,-5){\makebox{(a) $\phi \to -0$}}
\thinlines
 \put(25,20.5){\line(-1,0){25}}
 \put(25,19.6){\line(-1,0){25}}
\put(50.5,10){\line(1,2){10}}
\put(49.5,10){\line(1,2){10}}
  \qbezier(66,20) (64,14) (61.5,11.5)
    \qbezier(27,26.3) (30,26.8) (37,25.0)
  \put(66,20){\vector(0,1){0}}
 \put(37,25){\vector(3,-2){0}}
  \put(65,10){\makebox{$\mathbf{a}$}}
  \put(27,29){\makebox{$\mathbf{b}$}}
\put(-5,0){\makebox{$\Pi_{+}$}}
\thicklines
 \put(25,20){\circle*{1}}
\put(50,10){\circle*{1}}
\put(60,30){\circle*{1}}
 \qbezier[20](22,19.5) (23,11.5) (44,7.5)
 \qbezier[7](48.4,7.5) (54.2,8) (52.2,13.1)
  \qbezier(37,22) (42,20.5) (44.5,19)
  \qbezier(47.6,17.5) (48.3,17.8)  (51.9,14.1)
 \qbezier(22,20.4) (23,25.5) (37,22)
\qbezier(46,8) (50,2.5) (57.7,10.5)
\qbezier(64,32) (66,26) (63.2,20)
\qbezier(46,8) (44,14) (46.8,20)
\qbezier (64,32)(60,37.5)(52.3,29.5)
%%%%%%%
\qbezier (63.2,20)(60.8,14)(57.7,10.5)
\qbezier (46.8,20)(49.2,26)(52.3,29.5)
\end{picture} \qquad\qquad
%%%%%%%%%%%%%%%%%%%%%%%%
%%%%% \phi >0 %%%%%%
%%%%%%%%%%%%%%%%%%
\begin{picture}(80,50)(-10,-5)
 \put(-8,22){\makebox{$\infty$}}
 \put(20,12){\makebox{$z_2$}}
 \put(65,3){\makebox{$z_0$}}
 \put(43,37){\makebox{$z_1$}}
  \put(20,-5){\makebox{(b) $\phi \to +0$}}
\thinlines
 \put(25,20.5){\line(-1,0){25}}
 \put(25,19.6){\line(-1,0){25}}
\put(50.5,30){\line(1,-2){10}}
\put(49.5,30){\line(1,-2){10}}
  \qbezier(66,20) (64,26) (60.9,29.5)
    \qbezier(27,29.7) (30,32.1) (37,33.5)
  \put(60.7,29.5){\vector(-2,1){0}}
 \put(37,33.5){\vector(1,0){0}}
  \put(68,21){\makebox{$\mathbf{a}$}}
  \put(22,32){\makebox{$\mathbf{b}$}}
\put(-5,0){\makebox{$\Pi_{+}$}}
\thicklines
 \put(25,20){\circle*{1}}
\put(50,30){\circle*{1}}
\put(60,10){\circle*{1}}
 \qbezier(22,20.2) (23,28.5) (44,32.5)
 \qbezier(48.4,32.5) (54.2,32) (52.2,26.2)
  \qbezier[10](37,18) (48,21.5) (51.2,25.5)
  \qbezier[15](22,19) (23,14.5) (37,18)
\qbezier(46,32) (50,37.5) (57.7,29.5)
\qbezier(64,8) (66,14) (63.2,20)
\qbezier(46,32) (44,26) (46.8,20)
\qbezier (64,8)(60,2.5)(52.3,10.5)
%%%%%%%
\qbezier (63.2,20)(60.8,26)(57.7,29.5)
\qbezier (46.8,20)(49.2,14)(52.3,10.5)
\end{picture}
\end{center}
\caption{Cycles $\mathbf{a}$ and $\mathbf{b}$}
\label{cycles1}
\end{figure}
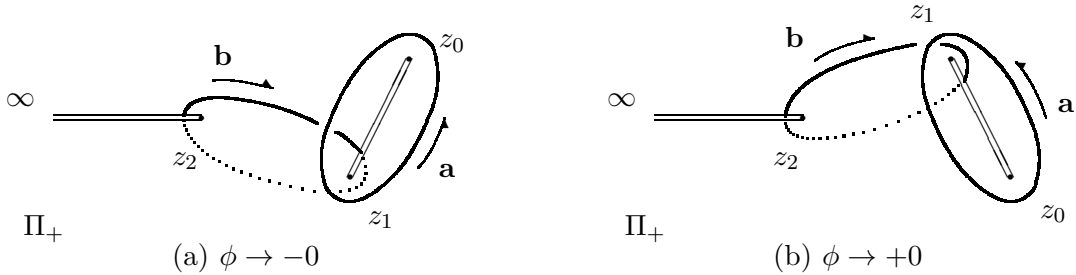
}
%%%%%%%%%%%%%%%%%%%%%%%%%%%%%%%%%%%%%%%%%%%%%%%%%
For $|\phi|<\pi/3$ the cycles $\mathbf{a}$ and $\mathbf{b}$ may be defined
continuously on $\Pi_{A_{\phi}}$, and the Boutroux equations are given by
%%%%%% (2.1) %%%%%%%%%%%%%%
\begin{equation}\label{2.1}
\im e^{i\phi} \int_{\mathbf{a}} \frac{w(A_{\phi},z)}{z^2} dz =0, \quad
\im e^{i\phi} \int_{\mathbf{b}} \frac{w(A_{\phi},z)}{z^2} dz =0
\end{equation}
admitting a unique solution $A_{\phi}$.
For $|\phi|<\pi/3$ the periods of $\Pi_{A_{\phi}}$ along $\mathbf{a}$ and
$\mathbf{b}$ are defined by
$$
\Omega^{\phi}_{\mathbf{a}}=\Omega_{\mathbf{a}}
=\int_{\mathbf{a}} \frac{dz}{w(A_{\phi},z)}, \quad
 \Omega^{\phi}_{\mathbf{b}}=\Omega_{\mathbf{b}}
=\int_{\mathbf{b}} \frac{dz}{w(A_{\phi},z)}, \quad
$$
which satisfy $\im \Omega_{\mathbf{b}}/\Omega_{\mathbf{a}} >0.$
%%%%%%%%%%%%%%%%%%
%%%%%% 2.3 %%%%%%%%%%
\subsection{Main theorems}\label{ssc2.3}
%%%%%%%%%%%%%%%%%%%%%%%%%%%%%%%%%%
Let $y(x)=y(G,x)$ be a solution of \eqref{1.2} labelled by the monodromy
data $G=(g_{ij})\in SL_2(\mathbb{C})$. 
Then we have the following, in which $\wp(u; g_2,g_3)$
is the Weierstrass $\wp$-function satisfying $\wp_u^2=4\wp^3-g_2\wp -g_3$
(\cite{HC}, \cite{WW}).
%%%%%%% Theorem 2.1 %%%%%%%
\begin{thm}\label{thm2.1}
Suppose that $0<\phi<\pi/3$ and that $G=(g_{ij})$ satisfies $g_{11}g_{12}
g_{22}\not=0.$ Then 
$$
y(x)= \wp\bigl(i(x-x_0^+) + O(x^{-\delta}); g_2(A_{\phi}), g_3(A_{\phi})\bigr)
 + \frac{A_{\phi}}{12}
$$
as $x=te^{i\phi} \to \infty$ through the cheese-like strip
$$
S(\phi,t_{\infty}, \kappa_0,\delta_0) =\{x=te^{i\phi}\,|\, \re t>t_{\infty},\,\,
|\im t|<\kappa_0\} \setminus \bigcup_{\sigma\in \mathcal{P}} \{ |x-\sigma|<
\delta_0 \}
$$
with
$$
\mathcal{P}=\{\sigma \,|\, \wp(i(\sigma-x_0^+);g_2(A_{\phi}), g_3(A_{\phi}))
=\infty\}= \{x_0^+ -i \Omega_{\mathbf{a}}\mathbb{Z}-i\Omega_{\mathbf{b}}
\mathbb{Z} \}.
$$
Here $\delta$ is some positive number, $\kappa_0$ a given positive number,
$\delta_0$ a given small positive number, $t_{\infty}=t_{\infty}(\kappa_0,
\delta_0)$ a sufficiently large number depending on $(\kappa_0,\delta_0);$
and 
\begin{align*}
&g_2(A_{\phi})= \frac {A_{\phi}^2}{12}, \quad
g_3(A_{\phi})= \frac {A_{\phi}^3}{216}-1, \quad
\\
&-i x^+_0 \equiv \frac i{2\pi}\Bigl(\Omega_{\mathbf{a}} \log\frac{g_{12}}
{g_{22}} -\Omega_{\mathbf{b}} (\log(g_{11}g_{22})-\pi i)\Bigr)
-ia \Omega_0 \quad \mod \Omega_{\mathbf{a}}\mathbb{Z}+
\Omega_{\mathbf{b}}\mathbb{Z}
\end{align*}
with
$$
\Omega_0=\int^{0^+}_{\infty}\frac{dz}{w(A_{\phi},z)},
$$
in which $0^+$ denotes $0\in \Pi_+$ and the contour $[\infty,0^+]\subset
\Pi_+$ contains
the line from $-\infty$ to $z_2$ along the upper shore of the cut 
$[\infty,z_2].$ 
\end{thm}
%%%%%%%%%%%%%%%%% Theorem 2.2 %%%%%%%%%%%%%%%%
\begin{thm}\label{thm2.2}
Suppose that $-\pi/3<\phi <0$ and that $G=(g_{ij})$ satisfies $g_{11}g_{21}
g_{22}\not=0.$ Then $y(x)$ admits an asymptotic representation of the same
form as in Theorem $\ref{thm2.1}$ with the phase shift  
$$
-i x^-_0 \equiv \frac {-i}{2\pi}\Bigl(\Omega_{\mathbf{a}} \log\frac{g_{21}}
{g_{11}} +\Omega_{\mathbf{b}} (\log(g_{11}g_{22})-\pi i)\Bigr)
-ia \Omega_0 \,\,\, \mod \Omega_{\mathbf{a}}\mathbb{Z}+
\Omega_{\mathbf{b}}\mathbb{Z}.
$$
\end{thm}
%%%%%%%%%%%%%%%%%%%%%%%%%%%%%%%%%%%%%
%%%%%%%% Remark 2.1 %%%%%%
\begin{rem}\label{rem2.1}
From a relation in the proof of Theorem $\ref{thm2.1}$ we have an expression of
$y'(x)$ for $0<\phi<\pi/3$ and $-\pi/3<\phi <0$ of the form
$$
\frac {iy'(x)+1}{2y(x)^2}=\wp\bigl(i(x-\hat{x}^{\pm}_0)+O(x^{-\delta});
\, g_2(A_{\phi}), g_3(A_{\phi}) ) +\frac{A_{\phi}}{12},
$$
respectively, where $i\hat{x}^{\pm}_0 =i x^{\pm}_0+\Omega_0.$ 
\end{rem}
%%%%%%%%%%%%%%%%%%%%%%%%%%%%%%%%
The expressions of $y(x)$ in Theorems \ref{thm2.1} and \ref{thm2.2} are
determined by $A_{\phi}$ and $x_0=x_0^+$ for $0<\phi<\pi/3,$ $=x_0^-$ for
$-\pi/3<\phi<0$. Since $\Omega_{\mathbf{a},\, \mathbf{b}}$ and $\Omega_0$
depend on $A_{\phi}$, these may be denoted by
$\Omega_{\mathbf{a},\, \mathbf{b}}^{\phi}$ and $\Omega_0^{\phi},$ respectively.
To emphasise this fact, write
$$
y(x)=P(A_{\phi}, x_0(G, \Omega_{\mathbf{a}}^{\phi}, \Omega_{\mathbf{b}}^{\phi},
\Omega_0^{\phi}); x)
$$
for $0<|\phi|<\pi/3.$  
\par
For $\phi$ such that $|\phi- 2m \pi/3|< \pi/3$ $(m\in \mathbb{Z})$,
set $\Omega^{\phi}_{\mathbf{a}, \, \mathbf{b}} =e^{2m\pi i/3} \Omega
_{\mathbf{a},\, \mathbf{b}}^{\phi-2m\pi/3}.$ The period, say, 
$\Omega^{\phi}_{\mathbf{a}}$ may be expressed by the integral on $\Pi_+$ 
\begin{align*}
\Omega^{\phi}_{\mathbf{a}} 
=& \int_{ e^{2m\pi i/3}\mathbf{a}} \frac{dz}{w(A_{\phi},z)}
= \int_{ e^{2m\pi i/3}\mathbf{a}} \frac{dz}{w(e^{2m\pi i/3}A_{\phi-2m\pi/3},z)}
\\
=& e^{2m\pi i/3}
\int_{\mathbf{a}} \frac{d\zeta}{w(A_{\phi-2m\pi/3},\zeta)}
= e^{2m\pi i/3} \Omega_{\mathbf{a}}^{\phi-2m\pi/3} \quad (z=e^{2m\pi i/3}\zeta).
\end{align*}
Furthermore, for $|\phi-2m\pi/3|<\pi/3$ set 
$\Omega_0^{\phi}=e^{2m\pi i/3}\Omega_0^{\phi-2m\pi/3}.$
The following provides an analytic continuation of $y(x)$ beyond the sector
$|\phi|<\pi/3.$
%%%%%%% Theorem 2.3 %%%%%%%%%%
\begin{thm}\label{thm2.3}
Suppose that $0<\phi-2m\pi/3<\pi/3$ $($respectively, $-\pi/3< \phi-2m\pi/3 <0)$
for $m\in \mathbb{Z}\setminus \{0\}.$
Then $y(x)$ admits the expression
$$
y(x)=y(G,x)=P(A_{\phi}, x_0(G^{(m)},\Omega_{\mathbf{a}}^{\phi},
\Omega_{\mathbf{b}}^{\phi},\Omega_0^{\phi});x)
$$
as $x=te^{i\phi} \to \infty$ through the cheese-like strip $S(\phi,t_{\infty},
\kappa_0,\delta_0),$ if $g^{(m)}_{11} g^{(m)}_{12} g^{(m)}_{22}
\not=0$ $($respectively, $g^{(m)}_{11} g^{(m)}_{21} g^{(m)}_{22} \not=0)$,
where
$$
G^{(m)}= \begin{cases}
 (S^0_0 \sigma_1)^m G\sigma_3^m e^{(m\pi/3)(a-i/2)\sigma_3} &\quad 
\text{if $m \ge 1;$} 
\\[0.1cm]
 (\sigma_1S^0_0 )^{n} G\sigma_3^n e^{(n\pi/3)(i/2-a)\sigma_3} &\quad 
\text{if $m =-n \le -1.$}
\end{cases} 
$$
\end{thm}
%%%%%%%%%%%%%%%%%%%%%%%%%%%%
%%%% Remark 2.1 %%%%%%
\begin{rem}\label{rem2.2}
The matrix $G^{(m)}$ has another expression of the form
$$
G^{(m)}= \begin{cases}
 G(S^{\infty}_0 S^{\infty}_1\sigma_3 e^{\pi(i/2-a) \sigma_3})^m 
\sigma_3^m e^{(m\pi/3)(a-i/2)\sigma_3} &\quad 
\text{if $m \ge 1;$} 
\\[0.1cm]
 G(\sigma_3 e^{\pi(a-i/2)\sigma_3}S^{\infty}_1S^{\infty}_0 )^{n} 
\sigma_3^n e^{(n\pi/3)(i/2-a)\sigma_3} &\quad 
\text{if $m =-n \le -1.$}
\end{cases} 
$$
\end{rem}
%%%%%%%%%%%%%%%%%%%%%%%%%%%%%
%%%% 2.4 %%%%%%%%%%%%%%%%
\subsection{Examples}\label{ssc2.4}
For simplicity suppose that $\epsilon=1$ and $b=2$ in equation \eqref{1.1}.
Let $G=({g}_{ij})$ with 
${g}_{11}{g}_{22} -{g}_{12}{g}_{21}=1$ be the monodromy data 
in Kitaev-Vartanian \cite{KV1}, \cite{KV2}, which coincide with ours above.
Suppose that ${g}_{11}{g}_{12}{g}_{21}{g}_{22}\not=0.$ 
Then \cite[Theorem 3.1]{KV1}, \cite[Theorems 2.1 and 2.3]{KV2} with 
$\varepsilon_1=\varepsilon_2=0$ provide general solutions of \eqref{1.1}
as in the following examples, in which we write
$l({g}_{11}{g}_{22})=i(2\pi)^{-1} \log({g}_{11}{g}_{22})$.
%%%%%%%% Example 2.1 %%%%%%
\begin{exa}\label{exa2.1}
If $|\re l({g}_{11}{g}_{22})| <1/6,$ equation \eqref{1.1} admits 
a solution of the form
\begin{align*}
u(\tau)&=2^{-1/3} \tau^{1/3}+2^{1/2}3^{-1/4} e^{3\pi i/4} l({g}_{11}
{g}_{22})^{1/2} \cosh(\chi(\tau)),
\\
\chi(\tau)&= i2^{1/3}3^{3/2} \tau^{2/3} +l({g}_{11}{g}_{22}) 
\log(2^{1/3}3^{3/2} \tau^{2/3}) +\gamma({g}_{11}{g}_{22}, {g}_{12}
/{g}_{22}) +o(\tau^{-\tilde{\delta}}) 
\end{align*}
as $\tau \to +\infty$, where $\gamma({g}_{11}{g}_{22}, {g}_{12}
/{g}_{22})$ is a constant expressed by $({g}_{11}{g}_{22},
 {g}_{12}/{g}_{22})$, and $\tilde{\delta}$ is some positive number. 
\end{exa}
%%%%%%%%%%%%%%%%%%%%%%%
%%%%%%%% Example 2.2 %%%%%%
\begin{exa}\label{exa2.2}
For $\re l({g}_{11}{g}_{22})\in (0,1),$ equation \eqref{1.1} admits 
a solution of the form
\begin{equation*}
u(\tau)=2^{-1/3} \tau^{1/3} \Bigl(1-\frac{3}{2\sin^2(\tilde{\chi}(\tau)/2)}
\Bigr)
= 2^{-1/3} \tau^{1/3} \frac{\sin(\tilde{\chi}(\tau)/2 -\chi_0)
\sin(\tilde{\chi}(\tau)/2 +\chi_0)}{\sin^2 (\tilde{\chi}(\tau)/2)}
\end{equation*}
with
\begin{align*}
\chi_0 =& -\pi/2+(i/2)\log(2+\sqrt{3}), 
\\
\tilde{\chi}(\tau)=& 2^{1/3}3^{3/2}\tau^{2/3} + l_*({g}_{11}{g}_{22})
\log(2^{1/3}3^{3/2} \tau^{2/3}) +\gamma_*({g}_{ij})
+o(\tau^{-\tilde{\delta}}) 
\end{align*}
as $\tau \to +\infty$ in a strip $|\im \tau^{2/3}|\ll 1.$ 
Here  
$l_*({g}_{11}{g}_{22}) =(2\pi)^{-1}\log(-{g}_{11}{g}_{22})$
 $(\in \mathbb{R})$ if $\re l({g}_{11}{g}_{22})=1/2,$ and $=
-i(l({g}_{11}{g}_{22})-1/2)$ otherwise; 
and $\gamma_*({g}_{ij})$ is a constant expressed by  
$(l_*({g}_{11}{g}_{22}), {g}_{11}{g}_{12},
{g}_{21}{g}_{22})$ if $\re l({g}_{11}{g}_{22})=1/2,$ and 
by $(l({g}_{11}{g}_{22}), {g}_{11}{g}_{12})$ otherwise.
\end{exa}
%%%%%%%%%%%%%%%%%%%%%%
By the change of variables $\tau^2=(x/3)^3,$ $\tau u=(x/3)^2 y$, these
solutions are taken to solutions of \eqref{1.2} on the positive real axis.
Proposition \ref{prop3.3}
guarantees the transfer between solutions of \eqref{1.1} and \eqref{1.2} with
labels. Observing
${g}_{11}{g}_{12}={g}_{11}{g}_{22} \cdot
{g}_{12}/{g}_{22}$ and
${g}_{21}{g}_{22}={g}_{11}{g}_{22} \cdot
{g}_{21}/{g}_{11}$, and applying 
Theorems \ref{thm2.1} and \ref{thm2.2}, we have elliptic representations
of these solutions for $-\pi/3<\phi<0$ and $0<\phi<\pi/3.$
\par
In the case where ${g}_{12}=0$ or ${g}_{21}=0,$ \cite[Theorems 3.2
and 3.3]{KV1} with $\varepsilon_1=\varepsilon_2=0$ give one-parameter 
solutions as follows:
%%%%%%%%%%%%%%%%%%%%%%%%
%%%%%%%% Example 2.3 %%%%%%
\begin{exa}\label{exa2.3}
Suppose that ${g}_{21}$ or ${g}_{12}=0$ and that ${g}_{11}{g}_{22}=1.$ 
Then \eqref{1.1} admits
\begin{equation*}
u(\tau)=2^{-1/3} \tau^{1/3}+ \frac{(s^0_0- ie^{-\pi a})c_*^{ia}}
{2\cdot 3^{1/4}\pi^{1/2}} \exp\bigl(\epsilon_* i(2^{1/3}3^{3/2}\tau^{2/3}
+k_*\pi/4) \bigr)(1+o(\tau^{-\tilde{\delta}})),
\end{equation*}
as $\tau \to +\infty.$ Here
$s^0_0-i e^{-\pi a}={g}_{12}/{g}_{22},$ $c_*=2-\sqrt{3},$ $\epsilon_*
=-1,$ $k_*=-1$ if ${g}_{21}=0;$ and
$s^0_0-i e^{-\pi a}=-{g}_{21}/{g}_{11},$ $c_*=2+\sqrt{3},$ $\epsilon_*
=1,$ $k_*=3$ if ${g}_{12}=0.$ 
\end{exa}
%%%%%%%%%%%%%%%%%%%%%%%%%%%%%%%%%%%%%
If ${g}_{11}{g}_{22} {g}_{12}\not=0$, ${g}_{21}=0$ 
(respectively, ${g}_{11}{g}_{22} {g}_{21}\not=0$, ${g}_{12}=0$),
Theorem \ref{thm2.1} for $0<\phi<\pi/3$ (respectively, \ref{thm2.2} for $-\pi/3
<\phi<0$) applies to the corresponding solution of \eqref{1.2}. In the case, 
say ${g}_{21}=0$, this solution is represented by the $\wp$-function for
$0<\phi<\pi/3$, and is truncated for $-\pi <\phi <0.$ 
%%%%%%%%%%%%%%%%%%%%%%%
%%%%%% Section 3 %%%%%%%%%%%%%%%%%
%%%%%%%%%%%%%%%%%%%%%%%%%%%%
\section{Isomonodromy deformation and monodromy data}\label{sc3}
%%%%%%%%%%%%%%%%%%%%%%%%%%%%%%%%%%%%%%%%
%%%% 3.1 %%%%%%%%%%
\subsection{Isomonodromy deformation}\label{ssc3.1}
%%%%%%%%%%%%%%%%%%%%%%%%%%%
Equation \eqref{1.1} governs isomonodromy deformation of the linear system
\begin{align}\label{3.1}
 &\frac{dU}{d\mu}= \mathcal{U}(\mu,\tau) U,
\\
\notag
 &\mathcal{U} (\mu,\tau)= -2i \tau \mu\sigma_3 +2\tau 
\begin{pmatrix}  0  &  2i \epsilon e^{i\varphi} \\
     -(\epsilon/4)e^{-i \varphi}(u^{\tau}/u -1/\tau -i\varphi_{\tau}) & 0
\end{pmatrix}
\\
\notag
& \phantom{---} -\frac 1{\mu} \Bigl(ia + \frac{\tau}2 (u^{\tau}/u -i\varphi
_{\tau}) \Bigr)\sigma_3 +\frac 1{\mu^2} \begin{pmatrix} 
0  &  2\epsilon e^{i\varphi}(ia -i\tau \varphi_{\tau}/2) \\
-i ue^{-i\varphi}  &  0        \end{pmatrix}
\end{align}
with $\varphi_{\tau}=(d/d\tau)\varphi=2a/\tau+b/u$, that is, the monodromy 
data remain
invariant under small change of $\tau$ if and only if $u^{\tau}=(d/d\tau)u$
holds and $u(\tau)$ solves \eqref{1.1} 
\cite[Propositions 1.1, 1.2 and 2.1]{KV1}.
Let us change \eqref{3.1} into system \eqref{1.3} associated with 
\eqref{1.2}. After the transformation
$$
U=\begin{pmatrix} e^{i\varphi/2} & 0 \\  0  & e^{-i\varphi/2} \end{pmatrix}
\begin{pmatrix} \sqrt{\epsilon}\tau^{3/4} & 0 \\  0  & \tau^{-3/4}/
\sqrt{\epsilon}  \end{pmatrix}   \hat{U},\quad 
\mu =\sqrt{2/\kappa}\tau^{1/2} \hat{\mu},     
$$
put
$$
\tau^2=\kappa \xi, \quad \tau u= q/\epsilon, \quad u^{\tau}+u/\tau=
2(\epsilon \kappa)^{-1}q^{\xi},
\quad \hat{U}= \begin{pmatrix} (2/\kappa)^{1/4}  &  0  \\  0 & (2/\kappa)^{-1/4}
\end{pmatrix} V 
$$
with $\kappa$ chosen so that $\epsilon \kappa b =2.$
Then \eqref{3.1} becomes
\begin{align*}
&\frac{dV}{d\hat{\mu}}=\mathcal{V}(\hat{\mu},\xi) V,
\\
&\mathcal{V}(\hat{\mu},\xi) = -4i\xi \hat{\mu}\sigma_3 
+\begin{pmatrix}  0 & 4i \\  -\xi (2\xi q^{\xi}/q -2(1+ia) -2i\xi/q) & 0
\end{pmatrix}
\\
&\phantom{---} - \frac 1{\hat{\mu}} \Bigl(\xi \frac{q^{\xi}}q -\frac 12
 -\frac{i\xi}q \Bigr) \sigma_3 -\frac i{\hat{\mu}^2} \begin{pmatrix}
 0 & 1/q \\ q & 0 \end{pmatrix}.
\end{align*}
The further change of variables 
$$
V=\begin{pmatrix}  -i/\sqrt{q}  &  0 \\  0 & i \sqrt{q} \end{pmatrix}\Psi,
\quad    q=(x/3)^2y, \quad \xi=(x/3)^3, \quad q^{\xi}=y^x +2y/x, \quad
 (x/3)\hat{\mu}=\lambda/2
$$
with $x=t e^{i\phi},$ i.e. $t=|x|,$ $\phi=\arg x$ and $y^x=e^{-i\phi}y^t$ 
takes the system above to \eqref{1.3}:
\begin{equation*}
\frac{d\Psi}{d\lambda} =\frac t3 \mathcal{B}(\lambda,t) \Psi,\phantom{---------}
\end{equation*}
whose right-hand side is written in the form
%%%%%% (3.2) %%%%%
\begin{align}\label{3.2}
\mathcal{B}(\lambda, t)&= b_1 \sigma_1 + b_2 \sigma_2 +b_3 \sigma_3,
\\
\notag
b_1 &= -(i/2)(2e^{i\phi} y + i\Gamma_0(t,y,y^t)y^{-1}) +2ie^{i\phi} \lambda^{-2},
\\
\notag
b_2 &= (1/2) (2e^{i\phi} y - i\Gamma_0(t,y,y^t)y^{-1}),
\\
\notag
b_3 &= -ie^{i\phi}\lambda - (\Gamma_0(t,y,y^t)+ 3(1/2+ia)t^{-1})\lambda^{-1},
\\
\notag
& \Gamma_0(t,y,y^t)=\frac{y^t}y -\frac{ie^{i\phi}}y -\frac{1+3ia}t.
\end{align}
In the linear systems above, $u,$ $u^{\tau};$ $q,$ $q^{\xi}$; $y,$ $y^t$
are arbitrary complex parameters or functions, and 
$2(\epsilon\kappa)^{-1}q^{\xi}=u^{\tau}+u/\tau$, $q^{\xi}=y^x+2y/x$ and
$y^x=e^{-i\phi}y^t$ are compatible with their
derivatives.
%%%%%%%%%%%%%%%%%%
\begin{prop}\label{prop0}
System \eqref{1.3} admits isomonodromy property if and only if $y^t=(d/dt)y$
holds and $y=y(e^{i\phi}t)=y(x)$ solves equation \eqref{1.2}. 
\end{prop}
%%%%%%%%%%%%%%%%%%%%%%%%%%%%%
%%%% Remark 3.1 %%%%%%%%%%
%% \begin{rem}\label{rem3.1}
%%The pair of functions $(u,u^{\tau})$ in \eqref{3.1} is related to
%% $(y,y^t)$ in \eqref{1.3} by 
%% $$
%% \epsilon \tau u=(e^{i\phi}t/3)^2y, \quad \epsilon b\tau^2=2(e^{i\phi}t/3)^3,\quad
%% u^{\tau}+u/\tau= be^{-i\phi}(y^t+2y/t).
%% $$
%% \end{rem}
%%%%%%%% 3.2 %%%%%%%%%%
\subsection{Monodromy data}\label{ssc3.2}
For each $j\in\mathbb{Z}$ system \eqref{1.3} admits the matrix solutions  
%%%%%% (3.3) %%%%%%%%%%%
\begin{equation}\label{3.3}
\hat{Y}^{\infty}_{j}(\lambda)= (I + O(\lambda^{-1}))\lambda^{-(1/2+ia)\sigma_3}
\exp(-(i/6)e^{i\phi}t\lambda^2 \sigma_3)
\end{equation}
as $\lambda \to \infty$ through the sector %$\Sigma^{\infty}_j:$
$|\arg \lambda +\phi/2 -j\pi /2|<\pi/2,$ and
%%%%%% (3.4) %%%%%%%%%%
\begin{equation}\label{3.4}
\hat{Y}^0_j(\lambda)=(i/\sqrt{2}) (\sigma_1+\sigma_3 +O(\lambda)) \exp(-(2i/3)
e^{i\phi} t\lambda^{-1} \sigma_3)
\end{equation}
as $\lambda \to 0$ through the sector %$\Sigma^0_j:$ 
$|\arg \lambda-\phi-j\pi|<\pi.$ 
The Stokes matrices are such that
$$
\hat{Y}^{\infty}_{j+1}(\lambda)=\hat{Y}^{\infty}_j(\lambda)\hat{S}^{\infty}_j,
\quad
%\text{in $\Sigma_{j}^{\infty}\cap \Sigma_{j+1}^{\infty}$}, 
\hat{Y}^{0}_{j+1}(\lambda)=\hat{Y}^{0}_j(\lambda)\hat{S}^{0}_j,
%\text{in $\Sigma_{j}^{0}\cap \Sigma_{j+1}^{0}$}, 
$$
and the connection matrix $\hat{G}=(\hat{g}_{ij})$ is defined by
%%%%%%%% (3.5) %%%%%%
\begin{equation}\label{3.5}
\hat{Y}^{\infty}_0(\lambda)=\hat{Y}^0_0(\lambda) \hat{G}, \quad  
\hat{g}_{11}\hat{g}_{22}-\hat{g}_{12}\hat{g}_{21}=1.
\end{equation}
The Stokes matrices satisfy
%%%%%%%%%%%%%
\begin{equation*}
\hat{S}^{\infty}_{k+2}=\sigma_3 e^{-\pi(a-i/2)\sigma_3} 
\hat{S}^{\infty}_k e^{\pi(a-i/2)\sigma_3}\sigma_3, \quad 
\hat{S}^0_k=\sigma_1 \hat{S}^0_{k+1} \sigma_1,
\end{equation*}
for $k\in \mathbb{Z}$, and the monodromy manifold is given by
$$
\hat{G}\hat{S}^{\infty}_0\hat{S}^{\infty}_1 \sigma_3 e^{\pi(i/2-a)\sigma_3}
=\hat{S}^0_0\sigma_1 \hat{G} 
$$
with
$$
\hat{S}^{\infty}_0 =\begin{pmatrix} 1 & 0 \\ \hat{s}^{\infty}_0 & 1 
\end{pmatrix}, \quad
\hat{S}^{\infty}_1 =\begin{pmatrix} 1 & \hat{s}^{\infty}_1 \\ 0 & 1 
\end{pmatrix}, \quad
\hat{S}^{0}_0 =\begin{pmatrix} 1 & \hat{s}^{0}_0 \\ 0 & 1 \end{pmatrix}. 
$$
These monodromy data and their relations are obtained by the same argument as 
in \cite[Section 2]{KV1}.
%%%%%%%%%%%%%%%%%%%%%%%%%%%%%%%%%%%%
\par
Let $G=({g}_{ij})$ be the monodromy data for system \eqref{3.1} given in
\cite{KV1}, \cite{KV2}. This connection matrix is defined by
$$
{Y}^{\infty}_0(\mu) =X^0_0(\mu) G.
$$
Here ${Y}^{\infty}_0(\mu)$ and $X^0_0(\mu)$ are matrix solutions of 
system \eqref{3.1} as follows:   
$$
{Y}^{\infty}_k(\mu) =(I+O(\mu^{-1})) \mu^{-(1/2 +ia)\sigma_3} \exp
(-i\tau \mu^2\sigma_3)
$$
as $\mu\to \infty$ through the sector  
$|\arg \mu +\arg \tau^{1/2} -\pi k/2|<\pi/2$, and
\begin{align*}
X^0_k(\mu) =&(i/\sqrt{2}) \Theta^{\sigma_3}_0 (\sigma_1+\sigma_3 +O(\mu))
\exp (-i\sqrt{\tau\epsilon b} \,\mu^{-1}\sigma_3),
\\
\Theta_0 =&(\epsilon b)^{1/4}\tau^{-1/4}(-u e^{-i \varphi}/\tau)^{-1/2}
\end{align*}
as $\mu\to 0$ through the sector  
$|\arg \mu -\arg(\tau\varepsilon b)^{1/2}-\pi k |<\pi$ 
\cite[Proposition 2.2]{KV1}.
Furthermore Stokes matrices are defined by
$$
{Y}^{\infty}_{j+1}(\lambda)={Y}^{\infty}_j(\lambda){S}^{\infty}_j, 
\quad
X^{0}_{j+1}(\lambda)=X^{0}_j(\lambda){S}^{0}_j, 
$$
and the monodromy manifold for \eqref{3.1} is given by 
%%%%%%%% (3.51) %%%%%%%%
\begin{equation}\label{3.51}
G {S}^{\infty}_0 {S}^{\infty}_1 \sigma_3 
e^{\pi(i/2-a)\sigma_3}={S}^0_0\sigma_1 G 
\end{equation}
%%%%%%%%%%%%%%%%%%%%%%%
with
$$
{S}^{\infty}_0 =\begin{pmatrix} 1 & 0 \\ {s}^{\infty}_0 & 1 
\end{pmatrix}, \quad
{S}^{\infty}_1 =\begin{pmatrix} 1 & {s}^{\infty}_1 \\ 0 & 1 
\end{pmatrix}, \quad
{S}^{0}_0 =\begin{pmatrix} 1 & {s}^{0}_0 \\ 0 & 1 \end{pmatrix}. 
$$
For $k\in \mathbb{Z},$
%%%%% (3.6) %%%%%%%%%%%
\begin{equation}\label{3.6}
{S}^{\infty}_{k+2}=\sigma_3 e^{-\pi(a-i/2)\sigma_3} 
{S}^{\infty}_k e^{\pi(a-i/2)\sigma_3}\sigma_3, \quad 
{S}^0_k=\sigma_1 {S}^0_{k+1} \sigma_1.
\end{equation}
As shown in \cite[Theorems 3.1, 3.2, 3.3]{KV1} and \cite[Theorems 2.1, 2.2,
2.3]{KV2} solutions of \eqref{1.1} are parametrised by the coordinates of
the monodromy manifold ${g}_{11}
{g}_{22}$, ${g}_{12}/{g}_{22},$ ${g}_{21}/{g}_{11},$ provided that \eqref{3.1}
is an isomonodromy system governed by \eqref{1.1}. 
The following relation suggests that we are allowed to use the same monodromy 
invariants in parametrising our solutions of \eqref{1.2} as in \cite{KV1} 
and \cite{KV2} (cf. Examples \ref{exa2.1}, \ref{exa2.2}, \ref{exa2.3}).
%%%%%%%%%%%%%%%%%%%%%%%%%%%%%%%%%%%%
%%%%%%% Proposition 3.3 %%%%%%%%
%%%%%%%%%%%%%%%%%%%%%%%%%%%%%%
\begin{prop}\label{prop3.3}
Let $(Y^{\infty,*}_{0}(\lambda), \hat{Y}^0_0(\lambda))
=(\hat{Y}^{\infty}_{0}(\lambda) \Theta_{0,*}^{-\sigma_3}, \hat{Y}^0_0(\lambda))$
with $\Theta_{0,*}= \Theta_0 c_0^{1/2+ia}$ be a pair of matrix solutions of
\eqref{1.3} near $\lambda=\infty$ and $0,$ where  
$c_0= (3/2) \sqrt{\epsilon b} \, \tau^{1/2}x^{-1}.$ Then, for this pair,
the corresponding Stokes matrices and connection matrix coincide with
$S^{\infty}_0,$ $S^{\infty}_1$, $S^0_0$ and $G$ for $(Y^{\infty}_0(\mu),
X^0_0(\mu))$ of \eqref{3.1}.
\end{prop}
%%%%%%%%%%%%%%%%%%%%%%%%%%
\begin{proof}
Note that \eqref{3.1} is changed into \eqref{1.3} by the transformation
$U=\Theta_0^{\sigma_3} \Psi,$ $\mu=c_0 \lambda$ with $c_0= (3/2) \sqrt{\epsilon
b} \, \tau^{1/2}x^{-1}.$ Set $Y^{\infty,*}_{0}(\lambda)=\hat{Y}^0_0(\lambda)
G^*.$ Then 
$$
(\Theta_0^{-\sigma_3} {Y}^{\infty}_0(c_0\lambda), \Theta_0^{-\sigma_3}
X^0_0(c_0\lambda) )= (\hat{Y}^{\infty}_0(\lambda) \Theta_0^{-\sigma_3}
c_0^{-(1/2+ia)\sigma_3}, \hat{Y}^0_0(\lambda))
=(Y^{\infty,*}_{0}(\lambda), \hat{Y}^0_0(\lambda))
$$ 
solves \eqref{1.3}. Insertion of this into ${Y}^{\infty}_0(\mu)=X^0_0(\mu)
G$ yields $G={G}^*$. Let $S^{\infty,*}_0,$ $S^{\infty,*}_1$ and $S^{0,*}_0$
be the Stokes matrices for $(Y^{\infty,*}_0(\lambda),\hat{Y}^0_0(\lambda))$.
Then the equation of the monodromy manifold is 
$$
GS^{\infty,*}_0S^{\infty,*}_1\sigma_3e^{\pi(i/2-a)\sigma_3}=S^{0,*}_0
\sigma_1 G,
$$
which yields the entries of $S_0^{\infty,*},$ $S_1^{\infty,*}$ and $S_0^
{0,*}$ in terms of $g_{ij}$ coinciding with those of $S^{\infty}_0,$
$S^{\infty}_1$ and $S^0_0$ derived from \eqref{3.51} as in \cite[p. 1172]{KV1}.
This completes the proof.
\end{proof}
%%%%%%%%%%%%%%%%%%%%%%%%%%%%%%
\begin{rem}\label{rem3.2}
We have $G=\hat{G} \Theta_0^{-\sigma_3}c_0^{-(1/2+ia)\sigma_3}
=\hat{G}\Theta_{0,*}^{-\sigma_3}$, 
$ S^{\infty}_m=\Theta_{0,*}^{\sigma_3}
\hat{S}_m^{\infty} \Theta_{0,*}^{-\sigma_3}$ and $S^0_m=\hat{S}_m^0$. 
\end{rem}
%%%%%%%%%%%%%%%%%%%%%%%
Equation \eqref{3.51} of the monodromy manifold may be extended.
%%%%%%%%%%%%%%%%%%%%%%%%%%%%%%%%
%%%%%%%% Proposition 3.1 %%%%%%%%%%%%
%%%%%%%%%%%%%%%%%%%%%%%%%%%%%
\begin{prop}\label{prop3.1}
For $m=1, 2, 3, \ldots,$
\begin{align*}
& {G}{S}^{\infty}_0 {S}^{\infty}_1 \cdots {S}^{\infty}_{2m-2}{S}^{\infty}_{2m-1}
\sigma_3^m e^{m\pi (i/2-a)\sigma_3} = {S}^0_0 \cdots {S}^0_{m-1} 
\sigma_1^m {G},
\\
& {G}{S}^{\infty}_{-1} {S}^{\infty}_{-2} \cdots 
{S}^{\infty}_{-2m+1} {S}^{\infty}_{-2m}
\sigma_3^m e^{m\pi (a-i/2)\sigma_3} = {S}^0_{-1} \cdots {S}^0_{-m} 
\sigma_1^m {G}.
\end{align*}
\end{prop}
%%%%%%%%%%%%%%%%%%%%%%%%%%%%%%%%%%
\begin{proof}
Recall that $Y^{\infty}_k(\mu)=\sigma_3 Y^{\infty}_{k+2}(\mu e^{\pi i})
\sigma_3 e^{-\pi(a-i/2)\sigma_3}$ and $X^0_k(\mu)=\sigma_3 X^0_{k+1}(\mu 
e^{\pi i})\sigma_1$ \cite[(24)]{KV1}.
Then
\begin{align*}
Y_0^{\infty}(\mu)& S_0^{\infty}S_1^{\infty} \cdots S^{\infty}_{2m-2}S^{\infty}
_{2m-1} = Y^{\infty}_{2m}(\mu) =\sigma_3 Y^{\infty}_{2(m-1)}(\mu e^{-\pi i})
\sigma_3 e^{\pi(a-i/2)\sigma_3}
\\
&= \cdots =\sigma_3^m Y^{\infty}_0(\mu e^{-m\pi i}) \sigma_3^m e^{m\pi(a-i/2)
\sigma_3},
\\
Y_0^{0}(\mu)& S_0^{0} \cdots S^{0}_{m-1} = Y^{0}_{m}(\mu) 
=\sigma_3 Y^{0}_{m-1}(\mu e^{-\pi i}) \sigma_1 = \cdots =
\sigma_3^m Y^{0}_{0}(\mu e^{-m\pi i}) \sigma_1^m .
\end{align*}
Using $Y^{\infty}_0(\mu) =Y^{0}_0(\mu)G$ and $Y^{\infty}_0(\mu e^{-m\pi i})
=Y^0_0(\mu e^{-m\pi i})G,$ we have
\begin{align*}
Y^0_0(\mu)& GS^{\infty}_0 S^{\infty}_1 \cdots S^{\infty}_{2m-2}S^{\infty}_{2m-1}
=\sigma_3^m Y^0_0(\mu e^{-m\pi i}) G\sigma_3^m e^{m\pi (a-i/2)\sigma_3}
\\
& =Y^0_0(\mu) S^{0}_0  \cdots S^{0}_{m-1}\sigma_1^m G \sigma_3^m e^{m\pi
(a-i/2)\sigma_3},
\end{align*}
which implies the first relation.
\end{proof}
The formulas above are also written as follows:
%%%%%%%%%%%%%%%%%%%%%%%%%%%%%
%%%%%%%%%%%%%%%%%%%%%%%%%%%%%
%%%%%% Proposition 3.2 %%%%%%
%%%%%%%%%%%%%%%%%%%%%%%%
\begin{prop}\label{prop3.2}
For $m=1, 2, 3, \ldots,$
\begin{align*}
& {G}{S}^{\infty}_0 {S}^{\infty}_1 \cdots {S}^{\infty}_{2m-2}{S}^{\infty}_{2m-1}
= ({S}^0_0\sigma_1)^m {G} \sigma_3^m e^{m\pi (a-i/2)\sigma_3} ,
\\
& {G}{S}^{\infty}_{-1} {S}^{\infty}_{-2} \cdots 
{S}^{\infty}_{-2m+1} {S}^{\infty}_{-2m}
= (\sigma_1{S}^0_0)^m {G} \sigma_3^m e^{m\pi (i/2-a)\sigma_3} .
\end{align*}
\end{prop}
%%%%%%%%%%%%%%%%%%%%%%%%%%%%
\begin{proof}
By \eqref{3.6}
$
{S}^0_{j-1}\sigma_1^j =\sigma_1 {S}^0_{j-2} \sigma_1^{j-1}= \cdots 
=\sigma_1^{j-1}{S}^0_0 \sigma_1,
$
and hence
$$
{S}^0_0 \cdots {S}^0_{m-1} \sigma^m_1 {G} =({S}^0_0\sigma_1)^m{G}, \quad
{S}^0_{-1} \cdots {S}^0_{-m} \sigma^m_1 {G} =(\sigma_1{S}^0_0)^m{G}.
$$
Combining these with Proposition \ref{prop3.1}, we have the desired result.
\end{proof}
%%%%%%%%%%%%%%%%%%%%%%%%%%%%%%%%%%%%
%%%%%%%%%%%%%%%%%%%%%%%%%%%%%
%%%%%%%%%%%%%%%%%%%%%%%%%%%
%%%%%%% Section 4 %%%%%%%%%%%%%%%%
\section{WKB analysis}\label{sc4}
%%%%%%%%% 4.1 %%%%%%%%%%%%%%%%%%%%%%%%
\subsection{Turning points and Stokes graphs}\label{ssc4.1}
%%%%%%%%%%%%%%%%%%%%%%%%%
Let us examine the characteristic roots $\pm \mu= \pm \mu(t, \lambda)$
of $\mathcal{B}(t,\lambda)$, the turning points, i.e. the roots of $\mu,$
and the Stokes graph, which are used in calculating monodromy data for
system \eqref{1.3}. The characteristic roots are given by
%%%%%% (4.1) %%%%%%%%%%%%%
\begin{align}\label{4.1}
\mu^2=& b_1^2 +b_2^2 + b_3^2
\\
\notag
=& - e^{2i\phi} \lambda^2 + e^{2i\phi} a_{\phi}\lambda^{-2} - 4 e^{2i\phi}
\lambda^{-4} +3i e^{i\phi} (1+2ia)t^{-1}
\end{align}
with
%%%%%%% (4.2) %%%%%%%%%%%%%%%
\begin{equation}\label{4.2}
a_{\phi}=a_{\phi}(t) = e^{-2i\phi} \Bigl(\frac{y^t}y + \frac 1{2t} \Bigr)^2
+4y + \frac 1{y^2} -3ie^{-i\phi}(1+2ia)\frac 1{ty}.
\end{equation}
The Stokes graph consists of the Stokes curves and the vertices: each Stokes
curve is defined by $\re \int^{\lambda}_{\lambda_*} \mu(\lambda) d\lambda=0 $
with a turning point $\lambda_*$, and the vertices are turning points or
singular points $\lambda=0, \infty.$ Here $\mu(\lambda)$
is considered on a two-sheeted Riemann surface glued along cuts with ends of  
turning points or singular points.
\par
First suppose that $\phi=0.$ If $a_0=a_{\phi=0} =3\cdot 2^{2/3},$ then
$$
\mu(\infty,\lambda)^2|_{\phi=0} =-\lambda^2 +a_0\lambda^{-2}-4\lambda^{-4}
= -\lambda^{-4} (\lambda^2- 2^{1/3})^2(\lambda^2+ 4^{2/3}).
$$
This means that $\mu(t,\lambda)$ admits six turning points
$\lambda_0,$ $\lambda_1,$ $\lambda'_0,$ $\lambda_1',$ $\lambda_2,$ $\lambda'_2$
such that $\lambda_0$ and $\lambda_1$ coalesce at $2^{1/6},$ $\lambda'_0$
and $\lambda'_1$ at $-2^{1/6}$ as $t\to \infty,$ and that $\lambda_2$ and 
$\lambda'_2$ approach $\pm 2^{2/3}i,$ respectively. 
The Stokes graph with $\phi=0$ is used in \cite[Section 4]{KV1}.  
(Note that a solution $y(x)$ of \eqref{1.2} for $x=t>0$ corresponds to
$u(\tau)$ satisfying \eqref{1.1} for $\tau>0$ if $\epsilon b>0$.)
The limit Stokes graph with $t=\infty$ is as in Figure \ref{stokes} (c) 
and $\mu(\lambda)$ is
defined on the two-sheeted Riemann surface $\mathcal{R}_0$ glued along, say
$[\lambda_2, e^{\pi i/2}0]\cup [\lambda'_2, e^{-\pi i/2}0]$.  
\par
The limit Stokes graph for the isomonodromy system \eqref{1.3} is considered
to reflect the Boutroux equations \eqref{2.1}.
When $\phi$ increases or decreases, the limit turning points for
$\lambda_0$ and $\lambda_1$ move according to the solution 
$A_{\phi}$ of the Boutroux equations \eqref{2.1}. By Proposition \ref{prop8.16},
for $\phi$ close to $0$, the double turning point at $2^{1/6}$ is resolved into
two simple turning points such  
that $\im \lambda_0 > 0 >\im \lambda_1,$ $\re \lambda_0 < 2^{1/6} < \re 
\lambda_1$ if $\phi>0,$ and that $\im \lambda_0 <0 <\im \lambda_1,$
$\re \lambda_0 < 2^{1/6} < \re \lambda_1$ if $\phi<0$. 
As will be shown in Proposition \ref{prop8.15},
for $0<|\phi|<\pi/3$ the coalescence of turning points does not occur,
and then topological properties of the limit Stokes graph remain invariant. 
Every turning point is simple, and 
the two-sheeted Riemann surface $\mathcal{R}_{\phi}$ of $\mu(\lambda)$
is glued along the cuts $[\lambda_0, \lambda_1]$, $[\lambda'_0,\lambda'_1]$
and 
$[\lambda_2,e^{(\pi-\phi)i/2}0]\cup [\lambda'_2,e^{-(\pi+\phi)i/2}0]$.
The Stokes graph lies on the upper sheet of $\mathcal{R}_{\phi}.$ 
For $-\pi/3 <\phi< 0$ and $0<\phi<\pi/3$, the limit Stokes graphs 
are as in Figures \ref{stokes} (b) and (d), in which each limit turning point
with $t=\infty$ is also denoted by $\lambda_{\iota}$ or $\lambda'_{\iota}.$ 
In our calculation, for $0<|\phi|<\pi/3,$ we use the Stokes curve from $0$ to
$\infty$ passing through $\lambda_0$ and $\lambda_1$ appearing  
as a resolution of the double turning point. 
For a technical reason,
the cut $[\lambda_0,\lambda_1]$ on the upper sheet is made in such a way 
that the Stokes curve 
$(\lambda_0,\lambda_1)^{\sim}$ is located along the lower shore 
(respectively, the upper shore) of the cut if $0<\phi<\pi/3$ 
(respectively, $-\pi/3<\phi <0$); and the cut $[0, \lambda_2]$ in such a way
that the cut $[\lambda_0,\lambda_1]$ is located on the right-hand side of 
$[0, \lambda_2]$ (cf. Figures \ref{curve+}, \ref{curve-}, \ref{cycles2}).
\par
Let us set
$$
\mu(t,\lambda)= ie^{i\phi}\lambda^{-2} \sqrt{4-a_{\phi}\lambda^2+\lambda^6
- 3i e^{-i\phi}(1+2ia)\lambda^4 t^{-1} }.
$$
Here the branch of the square root is fixed in such a way that
\begin{align*}
- ie^{-i\phi}\lambda^2 \mu(\infty,\lambda)
=& 2 \sqrt{(1- \lambda_{0,\infty}^{-2}\lambda^2)
(1-\lambda_{1,\infty}^{-2}\lambda^2)(1- \lambda_{2,\infty}^{-2}\lambda^2)}
\\
=& 2\sqrt{1- \lambda_{0,\infty}^{-2}\lambda^2}
\sqrt{1- \lambda_{1,\infty}^{-2}\lambda^2}
\sqrt{1-\lambda_{2,\infty}^{-2}\lambda^2}
\end{align*}
with $\lambda_{j,\infty}=\lambda_j(\infty)$, $\lambda_{0,\infty}^2\lambda
_{1,\infty}^2\lambda_{2,\infty}^2=-4$, and that
$\textstyle\sqrt{1-\lambda_{j,\infty}^{-2}\lambda^2} \to 1$ 
as $\lambda \to 0$ on the upper sheet. Then $\mu(t,\lambda) \to -i e^{i\phi}
\lambda +O(1)$ as $\lambda \to \infty$ and 
$\mu(t,\lambda) \to 2ie^{i\phi}\lambda^{-2} +O(1)$
as $\lambda\to 0$ on the upper sheet.
%%%%%%%%%%%%%%%%%%%%%%%%%%%%%
%%%%%%%%%%%%%%%%%%%%%%%%%%%%%%
%%%%%%%%% Figure 2 %%%%%%%%%%%
%%%%%%%%%%%%%%%%%%%%%%%%%%%%
{\small
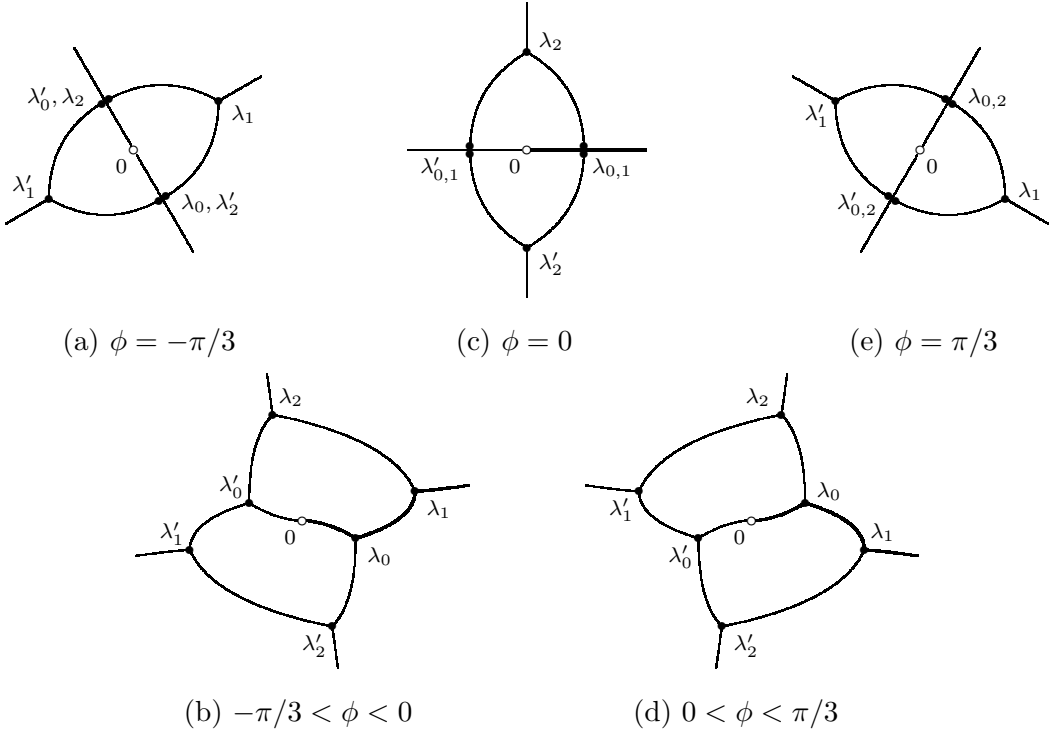
\begin{figure}[htb]
\begin{center}
\unitlength=0.78mm
%%%%%%%%%%%%%%%%%%%%%%%%%%%%%%%%%%
%%%%%%%%%%%%%%%%%%%%%%%%%%%%%%%
\begin{picture}(40,60)(-20,-30)
\put(0,0){\circle{1.5}}
\put(-4.2,8.6){\circle*{1.5}}
\put(-5.4,7.9){\circle*{1.5}}
\put(5.4,-7.9){\circle*{1.5}}
\put(4.2,-8.6){\circle*{1.5}}
\put(14.3,8.3){\circle*{1.5}}
\put(-14.3,-8.3){\circle*{1.5}}

{\tiny
\put(-18,8){\makebox{$\lambda'_{0},\lambda_{2}$}}
\put(-20.5,-7){\makebox{$\lambda'_{1}$}}
\put(8,-10){\makebox{$\lambda_{0},\lambda'_{2}$}}
\put(16.5,5){\makebox{$\lambda_1$}}
\put(-3,-4){\makebox{$0$}}
}
\qbezier (-14.3,-8.3) (-14.6,-8.5) (-21.5,-12.5)
\qbezier (21.5,12.5) (14.6,8.5) (14.3, 8.3)
\qbezier (4.8,-8.3) (7,-12) (10, -17.3)

\qbezier (0.5,-0.86) (4,-6.9) (7.5, -12.9)
\qbezier (14.3,8.3) (4.7,13.8)  (-4.8,8.3)
\qbezier (-14.3,-8.3) (-14.3,2.8)  (-4.8,8.3)
\qbezier (14.3,8.3) (14.3,-2.8)  (4.8,-8.3)
\qbezier (-14.3,-8.3) (-4.7,-13.8)  (4.8,-8.3)

\qbezier (-4.8,8.3) (-8,13.8) (-10, 17.3)
\qbezier (-0.4,0.7) (-4,6.9) (-7.5, 12.9)

\put(-12,-34){\makebox{(a) $\phi=-\pi/3$}}
\end{picture}
\qquad\qquad\quad
%%%%%%%%%%%%%%%%%%%%%%%%%%%%%%%
%%%%%%%%%%%%%%%%%%%%%%%%%%%%%%
%%%%%%%%%%%%%%%%%%%%%%%%%%%%%%%
\begin{picture}(40,60)(-20,-30)
\put(0,0){\circle{1.5}}
\put(9.6,0.7){\circle*{1.5}}
\put(9.6,-0.7){\circle*{1.5}}
\put(-9.6,0.7){\circle*{1.5}}
\put(-9.6,-0.7){\circle*{1.5}}
\put(0,16.6){\circle*{1.5}}
\put(0,-16.6){\circle*{1.5}}

{\tiny
\put(11.0,-4){\makebox{$\lambda_{0,1}$}}
\put(-18,-4){\makebox{$\lambda'_{0,1}$}}
\put(2,17.5){\makebox{$\lambda_2$}}
\put(2,-20.2){\makebox{$\lambda'_2$}}
\put(-3,-4){\makebox{$0$}}
}
\qbezier (0,-16.6) (0,-17) (0,-25)
\qbezier (0,25) (0,17) (0, 16.6)
\qbezier (-9.6,0) (-14,0) (-20, 0)

\qbezier (-1,0) (-8,0) (-15, 0)
\qbezier (0,16.6) (9.6,11.0)  (9.6,0)
\qbezier (0,-16.6) (9.6,-11)  (9.6,0)
\qbezier (0,16.6) (-9.6,11)  (-9.6,0)
\qbezier (0,-16.6) (-9.6,-11)  (-9.6,0)

\thicklines
\qbezier (9.6,0) (16,0) (20, 0)
\qbezier (0.8,0) (8,0) (15, 0)

\put(-12,-34){\makebox{(c) $\phi=0$}}
\end{picture}
\qquad\qquad\quad
%%%%%%%%%%%%%%%%%%%%%%%%%%%%%%
\begin{picture}(40,60)(-20,-30)
\put(0,0){\circle{1.5}}
\put(4.2,8.6){\circle*{1.5}}
\put(5.4,7.9){\circle*{1.5}}
\put(-5.4,-7.9){\circle*{1.5}}
\put(-4.2,-8.6){\circle*{1.5}}
\put(-14.3,8.3){\circle*{1.5}}
\put(14.3,-8.3){\circle*{1.5}}

{\tiny
\put(8,8){\makebox{$\lambda_{0,2}$}}
\put(16.5,-8){\makebox{$\lambda_{1}$}}
\put(-14,-10){\makebox{$\lambda'_{0,2}$}}
\put(-19.5,5){\makebox{$\lambda'_1$}}
\put(1,-4){\makebox{$0$}}
}
\qbezier (14.3,-8.3) (14.6,-8.5) (21.5,-12.5)
\qbezier (-21.5,12.5) (-14.6,8.5) (-14.3, 8.3)
\qbezier (-4.8,-8.3) (-7,-12) (-10, -17.3)

\qbezier (-0.5,-0.86) (-4,-6.9) (-7.5, -12.9)
\qbezier (-14.3,8.3) (-4.7,13.8)  (4.8,8.3)
\qbezier (14.3,-8.3) (14.3,2.8)  (4.8,8.3)
\qbezier (-14.3,8.3) (-14.3,-2.8)  (-4.8,-8.3)
\qbezier (14.3,-8.3) (4.7,-13.8)  (-4.8,-8.3)

\qbezier (4.8,8.3) (8,13.8) (10, 17.3)
\qbezier (0.4,0.7) (4,6.9) (7.5, 12.9)

\put(-12,-34){\makebox{(e) $\phi=\pi/3$}}
\end{picture}
%%%%%%%%%%%%%%%%%%%%%%%%%%%%%%
\vskip0.2cm
%%%%%%%%%%%%%%%%%%%%%%%%%%%%%%%%%%%%

\begin{picture}(70,60)(-35,-30)
\put(0,0){\circle{1.5}}
\put(-9,3){\circle*{1.5}}
\put(9,-3){\circle*{1.5}}
\put(-19,-5){\circle*{1.5}}
\put(19,5){\circle*{1.5}}
\put(-5,18){\circle*{1.5}}
\put(5,-18){\circle*{1.5}}

\qbezier (-9,3) (-5,0.5) (-1, 0)
\qbezier (-19,-5) (-22,-5.1) (-28,-6)
\qbezier (-5,18) (-5.3,20) (-6,25)
\qbezier (5,-18) (5.3,-20) (6,-25)

\qbezier (-5,18) (15,14) (19,5)
\qbezier (5,-18) (-15,-14) (-19,-5)

\qbezier (9,-3) (9,-14) (5,-18)
\qbezier (-9,3) (-9,14) (-5,18)

\qbezier (-9,3) (-19,0) (-19,-5)

\thicklines
\qbezier (9,-3) (5,-0.5) (0.8, 0)
\qbezier (9,-3) (19,0) (19,5)
\qbezier (19,5) (22,5.1) (28,6)

{\tiny
\put(11,-7){\makebox{$\lambda_0$}}
\put(21,1){\makebox{$\lambda_1$}}
\put(-4,20){\makebox{$\lambda_2$}}
\put(-3,-4){\makebox{$0$}}
\put(-14,5){\makebox{$\lambda'_0$}}
\put(-24,-3){\makebox{$\lambda'_1$}}
\put(0,-22){\makebox{$\lambda'_2$}}
}
\put(-20,-34){\makebox{(b) $-\pi/3 < \phi < 0$}}

\end{picture}
%%%%%%%%%%%%%%%%%%%%%%%%%%%%%%%
\,\,\,
%%%%%%%%%%%%%%%%%%%%%%%%%%%%%%%
\begin{picture}(70,60)(-35,-30)
\put(0,0){\circle{1.5}}
\put(-9,-3){\circle*{1.5}}
\put(9,3){\circle*{1.5}}
\put(-19,5){\circle*{1.5}}
\put(19,-5){\circle*{1.5}}
\put(-5,-18){\circle*{1.5}}
\put(5,18){\circle*{1.5}}

\qbezier (-9,-3) (-5,-0.5) (-1, 0)
\qbezier (-19,5) (-22,5.1) (-28,6)
\qbezier (-5,-18) (-5.3,-20) (-6,-25)
\qbezier (5,18) (5.3,20) (6,25)

\qbezier (-5,-18) (15,-14) (19,-5)
\qbezier (5,18) (-15,14) (-19,5)

\qbezier (9,3) (9,14) (5,18)
\qbezier (-9,-3) (-9,-14) (-5,-18)

\qbezier (-9,-3) (-19,0) (-19,5)

\thicklines
\qbezier (9,3) (5,0.5) (0.8, 0)
\qbezier (9,3) (19,0) (19,-5)
\qbezier (19,-5) (22,-5.1) (28,-6)

{\tiny
\put(11,4){\makebox{$\lambda_0$}}
\put(20,-3){\makebox{$\lambda_1$}}
\put(-3,-22){\makebox{$\lambda'_2$}}
\put(-3,-4){\makebox{$0$}}
\put(-14,-7){\makebox{$\lambda'_0$}}
\put(-24,1){\makebox{$\lambda'_1$}}
\put(-1,20){\makebox{$\lambda_2$}}
}
\put(-20,-34){\makebox{(d) $0 < \phi < \pi/3$}}
\end{picture}
%%%%%%%%%%%%%%%%%%%%%%%%%%%%%%%

%%%%%%%%%%%%%%%%%%%%%%%%%%%%%%%%%%%%%%%%%
%%%%%%%%%%%%%%%%%%%%%%%%%%%%%%%%%%%%%%%%%
\end{center}
\caption{Limit Stokes graphs for $|\phi|\le \pi/3$}
\label{stokes}
\end{figure}
}
%%%%%%%%%%%%%%%%%%
%%%%%%%%%%%%%%%%%%%%%%%%%%%%%%%%%%
%%%%%%%%%%%%%%%%%%

An unbounded
domain $D \subset \mathcal{R}_{\phi}$ is called a canonical domain if, for each
$\lambda \in D$, there exist contours $C_{\pm}(\lambda) \subset D$ terminating
in $\lambda$ such that
$$
\re \int^{\lambda}_{\lambda_-} \mu(\lambda) d\lambda \to -\infty \quad
\biggl(\text{respectively,} \,\,\,
\re \int^{\lambda}_{\lambda_+} \mu(\lambda) d\lambda \to +\infty \,\, \biggr)
$$
as $\lambda_- \to \infty$ along $C_-(\lambda)$ (respectively, as $\lambda_+
\to \infty$ along $C_+(\lambda)$) (see \cite{F}, \cite[p. 242]{FIKN}).
The interior of a canonical domain contains exactly one Stokes curve, and its
boundary consists of Stokes curves.
%%%%%%%%%%%%%%%%%%%%%%
%%%% 4.2 %%%%%%%%%%%%
\subsection{WKB solution}\label{ssc4.2}
%%%%%%%%%%%%%%%%%%%%%%%%%%%%%%%%%%%%%%%%%
The following WKB solution will be used in our calculus.
%%%%%%% Proposition 4.1 %%%%%%%%%%%%%%%%%%
\begin{prop}\label{prop4.1}
In the canonical domain whose interior contains a Stokes curve issueing from
the turning point $\lambda_0$ or $\lambda_1$, system \eqref{1.3} with 
$\mathcal{B}(\lambda,t)$ given by \eqref{3.2} admits an asymptotic solution 
expressed as
$$
\Psi_{\mathrm{WKB}}(\lambda) = T(I+O(t^{-\delta})) \exp\Bigl(
\int^{\lambda}_{\tilde{\lambda}_*} \Lambda(\tau) d\tau \Bigr), \quad
T = \begin{pmatrix} 1 & \frac{b_3-\mu}{b_1+ ib_2} \\
  \frac{\mu- b_3}{b_1- ib_2}  &  1    \end{pmatrix}
$$
outside suitable neighbourhoods of zeros of $b_1 \pm i b_2$ as long as
$|\lambda -\lambda_{\iota}| \gg t^{-2/3 +(2/3)\delta}$ $(\iota=0,1,2)$. Here
$\delta$ is an arbitrary number such that $0<\delta <1$, $\tilde{\lambda}_*$ 
is a base point near $\lambda_0$ or $\lambda_1$, and  
$$
\Lambda(\lambda)=\frac t3 \mu(t,\lambda) \sigma_3 -\diag T^{-1}T_{\lambda}.
$$
\end{prop}
%%%%%%%%%%%%%%%%%%%%%%%%%
\begin{proof}
This is shown by using $\mu =- i e^{i\phi} \lambda +O(1)$ near $\lambda
=\infty,$ and $= 2i e^{i\phi}\lambda^{-2}+O(1)$ near $\lambda=0$
(cf. \cite[Theorem 7.2]{FIKN}, \cite[Proposition 3.8]{Shimom}).
\end{proof} 
%%%%%%%%%%%%%%%%%%%%%%%%%%%%%%%%%%%%%
%%%%% Remark 4.1 %%%%%%%%%%%%%%%%%%
\begin{rem}\label{rem4.1}
In the proposition above
\begin{align*}
\diag T^{-1}T_{\lambda} &= \frac 1{2\mu(\mu+b_3)} (i(b_1b_2' -b_1'b_2)\sigma_3
+ (b_3 \mu' - b_3'\mu) I)
\\
&= \frac 14 \Bigl(1-\frac{b_3}{\mu} \Bigr) \frac{\partial}{\partial\lambda}
\log\frac{b_1+i b_2}{b_1 - ib_2} \sigma_3 +\frac 12 \frac{\partial}{\partial
\lambda} \log \frac{\mu}{\mu+ b_3} I,
\end{align*}
where $b_1'=(\partial/\partial\lambda)b_1$.
\end{rem}
%%%%%%%%%%%%%%%%%%%%%%%%%%%%%%%%
%%%%%% 4.3 %%%%%%%%%%%%%%%%%
\subsection{Local solution around a turning point}\label{ssc4.3}
%%%%%%%%%%%%%%%%%%%%%%%%%%
Near turning points the WKB solution above fails in expressing asymptotic 
behaviour. In the neighbourhood of $\lambda_{\iota}$ system \eqref{1.3} 
is reduced to
%%%%%% (4.3) %%%%%%
\begin{equation}\label{4.3}
\frac{dW}{d\zeta}=\begin{pmatrix} 0 & 1 \\ \zeta & 0 \end{pmatrix} W,
\end{equation}
which has the solutions ${}^T(\mathrm{Ai}(\zeta), 
\mathrm{Ai}_{\zeta}(\zeta))$, ${}^T(\mathrm{Bi}(\zeta), 
\mathrm{Bi}_{\zeta}(\zeta))$ with the Airy function $\mathrm{Ai}(\zeta)$ and
$\mathrm{Bi}(\zeta) = e^{-\pi i/6}\mathrm{Ai}(e^{-2\pi i/3}\zeta)$
\cite{AS}, \cite{HTF}.
Then we have the following solution near each simple turning point
\cite[Theorem 7.3]{FIKN}, \cite[Proposition 3.9]{Shimom}.
%%%%%%%%%%%%%%%%%%%%%%%%%%%%%%%
%%%% Proposition 4.2 %%%%%%%%%%%%
\begin{prop}\label{prop4.2}
For each simple turning point $\lambda_{\iota}$ $(\iota=0,1,2)$ write
$c_k=b_k(\lambda_{\iota}),$ $c'_k=(b_k)_{\lambda}(\lambda_{\iota})$ $(k=1,2,3)$,
and suppose that $c_k,$ $c_k'$ are bounded and $c_1 \pm i c_2\not=0.$ Let
$\hat{t}=2(2\kappa_c)^{-1/3}(c_1-ic_2)(t/3)^{1/3}$ with
$\kappa_c=c_1c_1' +c_2c'_2 +c_3c'_3.$ 
Then system \eqref{1.3} admits a matrix solution of the form 
\begin{equation*}
 \Phi_{\iota}(\lambda) =T_{\iota} (I+O(t^{-\delta'})) \begin{pmatrix}
1 & 0 \\ 0 & \hat{t}^{-1}    \end{pmatrix} W(\zeta) ,
\quad
 T_{\iota}= \begin{pmatrix}   1 &  -\frac{c_3}{c_1+ ic_2}  \\
- \frac{c_3}{c_1 -ic_2}  &  1    \end{pmatrix},
\end{equation*}
in which $\lambda-\lambda_{\iota}=(2\kappa_c)^{-1/3}(t/3)^{-2/3}(\zeta+\zeta_0)$
with $|\zeta_0| \ll t^{-1/3}$,  
as long as $|\zeta| \ll t^{(2/3-\delta')/3},$ that is, $|\lambda-\lambda_{\iota}
| \ll t^{-2/3+(2/3-\delta')/3}.$
Here $\delta'$ is an arbitrary number such that $0<\delta'< 2/3$, and
$W(\zeta)$ solves system \eqref{4.3} having canonical solutions 
$W_{\nu}(\zeta)$ $(\nu \in \mathbb{Z})$ such that
$$
W_{\nu}(\zeta)=\zeta^{-(1/4)\sigma_3}(\sigma_3+\sigma_1)(I+O(\zeta^{-3/2}))
\exp((2/3)\zeta^{3/2}\sigma_3)
$$
as $\zeta \to \infty$ through the sector $|\arg \zeta-(2\nu-1)\pi/3|<2\pi/3,$
and that $W_{\nu+1}(\zeta)=W_{\nu}(\zeta)S_{\nu}$ with
$$
S_1 =\begin{pmatrix} 1 & -i \\ 0 & 1 \end{pmatrix}, \quad 
S_2 =\begin{pmatrix} 1 & 0 \\ -i & 1 \end{pmatrix}, \quad 
S_{\nu+1}=\sigma_1 S_{\nu} \sigma_1.
$$
\end{prop}
%%%%%%%%%%%%%%%%%%%%%%%%%%%%
%%%% Remark 4.2 %%%%%%%%%%%%
\begin{rem}\label{rem4.2}
Putting $\lambda-\lambda_{\iota}=(2\kappa_c)^{-1/3}(e^{2\pi i/3})^{2j}
(t/3)^{-2/3}(\zeta+\zeta_0),$ $j\in \{0,\pm 1\},$ we have an expression of
$\Phi_{\iota}(\lambda)$ with $\hat{t}=2(2\kappa_c)^{-1/3}(e^{2\pi i/3})^{2j}
(c_1-ic_2)(t/3)^{1/3}.$
\end{rem}
%%%%%%%%%%%%%%%%%%%%%%%%%%%%%%%%%%%%%
%%%%%%%%%%%%%%% Section 5 %%%%%%%%
\section{Calculation of the connection matrix}\label{sc5}
%%%%%%%%%%%%%%%%%%%%%%%%%%%%%%
We calculate the connection matrix $\hat{G}=(\hat{g}_{ij})$ given by 
\eqref{3.5} as a solution of the direct monodromy problem 
by applying WKB analysis to system \eqref{1.3}. 
Suppose that $a_{\phi}(t)$ is given by \eqref{4.2} with a pair of
arbitrary functions $(y,y^t)=(y(t),y^t(t))$ not necessarily solving 
\eqref{1.2}, and that 
%%%%%%%%%%%%%%%%%%%%%%%%
\begin{equation}\label{5.0}
a_{\phi}(t)=A_{\phi}+ \frac{B_{\phi}(t)}{t}, \quad B_{\phi}(t) \ll 1
\end{equation}
for $t \in S_{\phi}(t'_{\infty},\kappa_1,\delta_1)$ with given  
$\kappa_1>0$, small given $\delta_1>0$ and sufficiently large $t'_{\infty}>0$. 
Here $A_{\phi}$ is a solution of the Boutroux equations \eqref{2.1}, and
$$
S_{\phi}(t'_{\infty}, \kappa_1, \delta_1)=\{t\,|\,\, \re t > t'_{\infty},
\,\, |\im t| < \kappa_1, \,\, |y(t)|+ |y^t(t)|+|y(t)|^{-1} 
< \delta_1^{-1} \}.
$$
\par
Let $0<\phi<\pi/3.$ We calculate the analytic continuation of the 
matrix solution near $\lambda=\infty$ along the Stokes curve consisting of 
$$
\mathbf{c}_{\infty}=(\infty, \lambda_1)^{\sim}, \quad
\mathbf{c}_{1}=(\lambda_1,\lambda_0)^{\sim}, \quad
\mathbf{c}_{0}=(\lambda_0,0)^{\sim}
$$
starting from $\infty$ and terminating in $0$ on the upper sheet of the 
Riemann surface $\mathcal{R}_{\phi}$ of $\mu(\infty,\lambda)$ 
as in Figure \ref{curve+}. Under supposition \eqref{5.0} these curves
$\mathbf{c}_0,$ $\mathbf{c}_1$, $\mathbf{c}_{\infty}$ lie within the distance
$O(t^{-1})$ from the limit Stokes graph. 
Recall that the curve $\mathbf{c}_1$ 
is located along the lower shore of the cut 
$[\lambda_0,\lambda_1]$.  
\par
In the WKB solution, write $\Lambda(\lambda)$ in the component-wise form
$ \Lambda(\lambda) =\Lambda_3(\lambda) + \Lambda_I(\lambda)$ with
$$
\Lambda_3(\lambda) =\frac t3 \mu(t,\lambda) \sigma_3 -\diag T^{-1}T_{\lambda}
|_{\sigma_3} \sigma_3,
\quad \Lambda_I(\lambda)= -\diag T^{-1}T_{\lambda}|_{I} I,
$$
in which $\diag T^{-1}T_{\lambda}|_{\sigma_3}\sigma_3  \in \mathbb{C}\sigma_3$,
$\diag T^{-1}T_{\lambda}|_{I}I \in \mathbb{C} I$. 
In Propositions \ref{prop4.1} and
\ref{prop4.2}, if $\delta=\delta'=2/9-\varepsilon$ with any $\varepsilon $
such that $0<\varepsilon<2/9$,
then both propositions are applicable in the annulus
$$
\mathcal{A}^{\iota}_{\varepsilon}: \quad t^{-2/3+(2/3)(2/9-\varepsilon)}
\ll |\lambda -\lambda_{\iota}| \ll t^{-2/3 +(2/3)(2/9+\varepsilon/2)}
$$
$(\iota=0,1)$. In what follows we set $\delta=2/9 -\varepsilon,$ and
write $c_k=b_k(\lambda_0)$, $d_k=b_k(\lambda_1)$ $(k=1,2,3)$.
%%%%%%%%%%%%%%%%%%%%%%%%%%%%%%%%%%%%%%%%%%%%%%%%%%%%%
%%%%%%%% Figure 3 %%%%%%%%%%%%%%%%%%%%%%%%
%%%%%%%%%%%%%%%%%%%%%%%%%%%%%%%%%%%%%
{\small
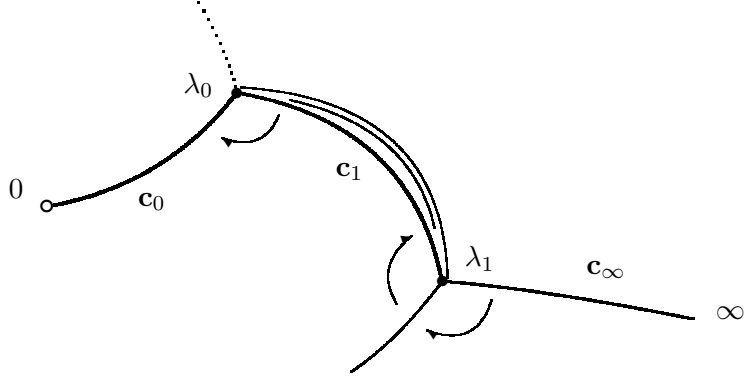
\begin{figure}[htb]
\begin{center}
\unitlength=1mm
%%%%%%%%%%%%%%%%%%%%%%%%%%%%%%%%%%
%%%%%%%%%%%%%%%%%%%%%%%%%%%%%%%%%%
%%%%%%%%%%%%%%%%%%
\begin{picture}(100,50)(0,-25)
\put(-5,1){\makebox{$0$}}
\put(18,15){\makebox{$\lambda_0$}}
\put(55,-8){\makebox{$\lambda_1$}}
\put(88,-15){\makebox{$\infty$}}
\put(12,0){\makebox{$\mathbf{c}_0$}}
\put(38,4){\makebox{$\mathbf{c}_1$}}
\put(71,-9){\makebox{$\mathbf{c}_{\infty}$}}

\thinlines
\qbezier(32,14) (48,10.5) (51,-3)
\qbezier(25.5,15.7) (52,14) (52.7,-9.5)

%   \qbezier(58.5,-9) (58.5,-5) (54,-4)
% \put(54,-4){\vector(-1,0){0}}
   \qbezier(58.5,-12.5) (56.5,-19) (50,-16.5)
 \put(50,-16.5){\vector(-3,2){0}}
   \qbezier(46,-13) (43,-8) (48,-4) %%%%%%%%%%%%%%%%%%%%%%%%%%
 \put(48,-4){\vector(3,1){0}}

  \qbezier(30.5,12) (28.5,7) (23,9)
\put(23,9){\vector(-3,2){0}}

%\qbezier[15](25,15) (23,21) (24,27)
\qbezier(40,-22) (46,-18) (52,-10)
\qbezier(40,-22.06) (46,-18.06) (52,-10.06)

\thicklines
\put(0,0){\circle{1.5}}
\put(25,15){\circle*{1.5}}
\put(52,-10){\circle*{1.5}}
\qbezier(0.8,0) (16,3) (25,15)
\qbezier(0.8,0.08) (16,3.08) (25,15.08)
\qbezier(25,15) (48,11) (52,-10)
\qbezier(24.92,15) (47.92,11) (51.92,-10)

%\qbezier[12](25,15) (23,21) (24,27)
\qbezier[12](25,15) (24,20) (20,27)
\qbezier(85,-15) (65,-11) (52,-10)

\end{picture}
\end{center}
\caption{Stokes curve for $0<\phi <\pi/3$}
\label{curve+}
\end{figure}
}
%%%%%%%%%%%%%%%%%%%%%%%%%%%%%%%%%%%
%%%%%%%%%%%%%%%%%%%%%%%%%%%%%%%%

{\bf (1)} Let $\Psi_{\infty}(\lambda)$ along $\mathbf{c}_{\infty}
=(\infty, \lambda_1)^{\sim}$ be a WKB solution by Proposition \ref{prop4.1}, and
let ${Y}^{\infty,*}_0(\lambda)=\hat{Y}^{\infty}_0(\lambda)
\Theta_{0,*}^{-\sigma_3}$ be given by \eqref{3.3} and Proposition \ref{prop3.3}.
Set
${Y}^{\infty,*}_0(\lambda)\Theta_{0,*}^{\sigma_3}=\hat{Y}^{\infty}_0(\lambda)
=\Psi_{\infty}(\lambda) \Gamma_{\infty}.$ Note that
$\mu(t,\lambda)= -i e^{i\phi} \lambda -\tfrac 32 (1+2ia)t^{-1}\lambda^{-1}
 +O(\lambda^{-3})$ along $\mathbf{c}_{\infty}$, and
$\mu -b_3 \ll \lambda^{-1}$ as $\lambda \to \infty.$ Then
%%%%%%%%%%%
\begin{align*}
\Gamma_{\infty} =& \Psi_{\infty}(\lambda)^{-1} \hat{Y}^{\infty}_0(\lambda)
 = \Psi_{\infty}(\lambda)^{-1} {Y}^{\infty,*}_0(\lambda) \Theta_{0,*}^{\sigma_3}
\\
=& \exp\Bigl(-\int^{\lambda}_{\tilde{\lambda}_1} \Lambda(\tau) d\tau\Bigr)
T^{-1} (I+O(t^{-\delta}+|\lambda|^{-1})) 
\\
&\phantom{---}\times \exp \Bigl(-\frac 16 (i e^{i\phi}
t\lambda^2 + 3(1+2ia) \log \lambda)\sigma_3 \Bigr),
\\
=& C_3(\tilde{\lambda}_1) c_I(\tilde{\lambda}_1)(I+O(t^{-\delta}))
\\
&\phantom{---}\times \exp\biggl( -\lim_{\substack {\lambda\to\infty \\ 
\lambda\in \mathbf{c}_{\infty}} }  
\Bigl(\int^{\lambda}_{\lambda_1} \Lambda_{3}(\tau)
d\tau + \frac 16 (i e^{i\phi}t\lambda^2 + 3(1+2ia) \log \lambda)\sigma_3 \Bigr)
\biggr),
\end{align*}
in which $C_3(\tilde{\lambda}_1)=\exp (\int^{\tilde{\lambda}_1}_{\lambda_1}
\Lambda_{3}(\tau)d\tau )$, $c_I(\tilde{\lambda}_1) =\exp (-\int^{\infty}
_{\tilde{\lambda}_1} \Lambda_{I}(\tau) d\tau ),$ and $\tilde{\lambda}_1 \in
\mathbf{c}_{\infty},$ $ \tilde{\lambda}_1-\lambda_1 \asymp t^{-1}.$
%%%%%%%%%%%%%%%%%%%%%%%%%%%%%%%%%%%%%%%
\par
{\bf (2)} For $\Psi_{\infty}(\lambda)$ and for $\Phi^+_1(\lambda)$ given by
Proposition \ref{prop4.2} in the annulus $\mathcal{A}^1_{\varepsilon}$
around $\lambda_1$, set $\Psi_{\infty}(\lambda)=\Phi^+_1(\lambda)\Gamma_{1+}$
along $\mathbf{c}_{\infty}.$
Suppose that the curve $(2\kappa_d)^{1/3}(\lambda-\tilde{\lambda}_1)=(t/3)^{-2/3}
(\zeta+O(t^{-1/3})),$ $\kappa_d=d_1d_1'+d_2d_2'+d_3d_3'$ with 
$\lambda \in \mathbf{c}_{\infty}$ enters the sector
$|\arg \zeta -7\pi/3 |<2\pi/3$ (the other cases are similarly treated by
Remark \ref{rem4.2}). Write $K^{-1}=2(2\kappa_d)^{-1/3}(d_1-i d_2).$ Then, by
Propositions \ref{prop4.1} and \ref{prop4.2},
\begin{align*}
\Gamma_{1+} =& \Phi^+_1(\lambda)^{-1} \Psi_{\infty}(\lambda)
\\
=& W(\zeta)^{-1} \begin{pmatrix} 1 & 0 \\  0 & (t/3)^{-1/3}K \end{pmatrix}
^{\!\!-1}
(I+O(t^{-\delta}))  \begin{pmatrix}  1 &  -\frac{d_3}{d_1+i d_2} \\
-\frac{d_3}{d_1-i d_2}  & 1  \end{pmatrix}^{\!\!-1}
\\
&\phantom{----} \times \begin{pmatrix}  1  & \frac{b_3 - \mu}{b_1 +i b_2} \\
\frac{\mu -b_3}{b_1- i b_2}  &  1  \end{pmatrix}  (I+O(t^{-\delta}))
\exp \Bigl(\int^{\lambda}_{\tilde{\lambda_1}} \Lambda(\tau)d\tau \Bigr)
\\
=& W(\zeta)^{-1} \begin{pmatrix}  1  &  \frac{d_3}{d_1+i d_2} \\
\frac{(t/3)^{1/3}\mu} {2K(d_1-i d_2)} & \frac{(t/3)^{1/3}\mu}{2K d_3}
\end{pmatrix}  (I+O(t^{-\delta})) \exp\Bigl(\int^{\lambda}_{\tilde{\lambda}_1}
\Lambda(\tau) d\tau \Bigr)
\end{align*}
for $\lambda \in \mathcal{A}^1_{\varepsilon} \cap \mathbf{c}_{\infty},$ where
$(\mu -b_3)/(b_1 \pm ib_2)=(\mu -d_3)/(d_1 \pm i d_2)+O(\eta)$, $\eta=
\lambda-\tilde{\lambda}_1$. Since
$ \mu =(2\kappa_d)^{1/2}\eta^{1/2}(1+O(\eta)) 
 =2K(d_1-i d_2)(t/3)^{-1/3}\zeta^{1/2}(1+O(\eta)),$
we have
\begin{equation*}
\Gamma_{1+} 
% & \exp\Bigl(-\frac 23 \zeta^{3/2}\sigma_3\Bigr)
% (\sigma_3+\sigma_1)^{-1}
% \zeta^{(1/4)\sigma_3}\begin{pmatrix} 1 & 0 \\ 0 & \zeta^{1/2} \end{pmatrix}
% \\
% & \phantom{---}\times
% \begin{pmatrix} 1 & -\frac{d_1-i d_2}{d_3} \\ -1 & -\frac{d_1-i d_2}{d_3}
% \end{pmatrix}(I+O(t^{-\delta}))\exp\Bigl(\int^{\lambda}_{\tilde{\lambda}_1}
% \Lambda_-(\tau)d\tau \Bigr)
% \\
= \exp\Bigl(\int^{\lambda}_{\tilde{\lambda}_1} \Lambda(\tau) 
d\tau -\frac 23 \zeta^{3/2}\sigma_3\Bigr)\zeta^{1/4}(I+O(t^{-\delta}))
\begin{pmatrix}  1 & 0 \\ 0 & -\frac{d_1-i d_2}{d_3} \end{pmatrix}.
\end{equation*}
By $\Lambda_{3}(\lambda) =((2\kappa_d)^{1/2}(t/3)\eta^{1/2}(1+O(\eta))
+O(\eta^{-1/2}))\sigma_3$ and $\Lambda_{I}(\lambda)=(-{\eta}^{-1}/4+
O(\eta^{-1/2}) )I$ (cf. Remark \ref{rem4.1}) for $\eta=\lambda
-\tilde{\lambda}_1,$ $\lambda\in \mathcal{A}^1_{\varepsilon} \cap \mathbf{c}
_{\infty},$
$$
\Gamma_{1+} = (\tilde{\zeta}_1)^{1/4}(I+O(t^{-\delta})) C_3(\tilde{\lambda}_1)
^{-1} \begin{pmatrix} 1 & 0 \\ 0 & -\frac{d_1-i d_2}{d_3} \end{pmatrix}
$$
with suitably chosen $\tilde{\zeta}_1 \asymp \tilde{\lambda}_1 -\lambda_1.$
%%%%%%%%%%%%%%%%%%%%%%%%%%%%%
\par
{\bf (3)} Let $\Phi^-_1(\lambda)$ be the solution by Proposition \ref{prop4.2} 
near $\mathbf{c}_1=(\lambda_1,\lambda_0)^{\sim}$, and 
set $\Phi^+_1(\lambda)=\Phi^-_1(\lambda)\Gamma_{1*},$ where 
$\Phi^+_1(\lambda)$ is the analytic continuation
along an arc in $\mathcal{A}^1_{\varepsilon}$ in the clockwise direction.
Then by Proposition \ref{prop4.2},
$$
\Gamma_{1*} = \Phi^-_1(\lambda)^{-1}\Phi^+_1(\lambda)=S_2S_3=\begin{pmatrix}
1 & 0 \\ -i & 1 \end{pmatrix} \begin{pmatrix} 1 & -i \\ 0 & 1 \end{pmatrix}.
$$
%%%%%%%%%%%%%%%%%%%%%%%%%
\par
{\bf (4)} For $\Phi^-_1(\lambda)$ and the WKB solution $\Psi^-_1(\lambda)$
along $\mathbf{c}_1$, set $\Phi^-_1(\lambda)=\Psi^-_1(\lambda)\Gamma_{1-}.$
Then, supposing the curve $(2\kappa_d)^{1/3}(\lambda-\tilde{\lambda}_1')
=(t/3)^{-2/3}(\zeta+O(t^{-1/3}))$ with $\lambda \in \mathbf{c}_1$ to be in the
sector $|\arg\zeta -\pi|<2\pi/3,$ we have, for $\tilde{\lambda}'_1 \in 
\mathbf{c}_1,$ $|\tilde{\lambda}'_1-\lambda_1| \asymp t^{-1},$
\begin{align*}
\Gamma_{1-} =& \Psi_1^-(\lambda)^{-1} \Phi^-_1(\lambda)
\\
=& \exp \Bigl( -\int^{\lambda}_{\tilde{\lambda}'_1} \Lambda(\tau)d\tau \Bigr)
(I+O(t^{-\delta})) \begin{pmatrix}  1  & \frac{b_3-\mu}{b_1+ ib_2}  \\
\frac{\mu-b_3}{b_1-i b_2} \end{pmatrix}^{\!\! -1}
\\
&\phantom{--}\times  \begin{pmatrix}  1 &  -\frac{d_3}{d_1+ id_2} \\
-\frac{d_3}{d_1- i d_2} &  1 \end{pmatrix} (I+O(t^{-\delta})) 
\begin{pmatrix} 1 & 0 \\ 0 & (t/3)^{-1/3}\tilde{K} \end{pmatrix} W(\zeta)
\\
=& \exp\Bigl(\frac 23\zeta^{3/2} \sigma_3 -\int^{\lambda}_{\tilde{\lambda}'_1}
\Lambda(\tau) d\tau \Bigr) \zeta^{-1/4}(I+O(t^{-\delta})) \begin{pmatrix}
1 & 0 \\ 0 & -\frac{d_3}{d_1-i d_2} \end{pmatrix},
\end{align*}
where $\tilde{K}^{-1}=2(2\kappa_d)^{-1/3} (d_1- i d_2).$
This yields 
$$
\Gamma_{1-}=(\tilde{\zeta}'_1)^{-1/4}(I+O(t^{-\delta})) C'_3(\tilde{\lambda}'_1)
\begin{pmatrix}  
1 &  0  \\  0 & -\frac{d_3}{d_1- i d_2} \end{pmatrix}
$$
with $C'_3(\tilde{\lambda}'_1)=\exp(\int^{\tilde{\lambda}'_1}_{\lambda_1}
\Lambda_3(\tau)d\tau )$ for some $\tilde{\zeta}'_1 \asymp \tilde{\lambda}'_1
-\lambda_1.$
%%%%%%%%%%%%%%%%%%%%%%%%%%%%%%%%%%%5
\par
{\bf (5)} For $\Psi^-_1(\lambda)$ and the WKB solution $\Psi^+_0(\lambda)$ along
$\mathbf{c}_1$ near $\lambda_0$, set $\Psi^-_1(\lambda)=\Psi^+_0(\lambda)
\Gamma_{01}.$ Then, for $\tilde{\lambda}_0 \in \mathbf{c}_1,$ $\tilde{\lambda}
_0 -\lambda_0 \asymp t^{-1}$,
\begin{align*}
\Gamma_{01} =& \Psi^+_0(\lambda)^{-1} \Psi^-_1(\lambda)
\\
=& \exp\Bigl(-\int^{\lambda}_{\tilde{\lambda}_0} \Lambda(\tau)d\tau \Bigr)
T^{-1}(I+O(t^{-\delta})) T \exp \Bigl(\int^{\lambda}_{\tilde{\lambda}'_1}
\Lambda(\tau)d\tau\Bigr)
\\
=& C'_3(\tilde{\lambda}'_1)^{-1} C''_3(\tilde{\lambda}_0) c_I(\tilde{\lambda}'
_1, \tilde{\lambda}_0) \exp\Bigl( -\int^{\lambda_1}_{\lambda_0} \Lambda_3(\tau)
d\tau \Bigr),
\end{align*}
where $C''_3(\tilde{\lambda}_0)=\exp (\int^{\tilde{\lambda}_0}_{\lambda_0}
\Lambda_3(\tau)d\tau),$ $c_I(\tilde{\lambda}'_1,\tilde{\lambda}_0)=\exp
(-\int^{\tilde{\lambda}'_1}_{\tilde{\lambda}_0} \Lambda_I(\tau) d\tau).$
%%%%%%%%%%%%%%%%%%%%%%%%%%%%%%%%%%%%%%%%
\par
{\bf (6)} For $\Psi^+_0(\lambda)$ and for $\Phi^+_0(\lambda)$ given by 
Proposition
\ref{prop4.2} in the annulus $\mathcal{A}^0_{\varepsilon}$ around $\lambda_0$,
set $\Psi^+_0(\lambda)=\Phi^+_0(\lambda)\Gamma_{0+}.$ Then, by the same
argument as in (2) above, we have 
$$
\Gamma_{0+} =\Phi^+_0(\lambda)^{-1}\Psi^+_0(\lambda) =(\tilde{\zeta}_0)^{1/4}
(I+O(t^{-\delta})) C''_3(\tilde{\lambda}_0)^{-1} \begin{pmatrix}
 1 & 0  \\ 0 & -\frac{c_1- i c_2}{c_3} \end{pmatrix}
$$
for some $\tilde{\zeta}_0 \asymp \tilde{\lambda}_0 -\lambda_0.$ 
%%%%%%%%%%%%%%%%%%%%%%%%%%%%%%%%%%
\par
{\bf (7)} Let $\Phi^-_0(\lambda)$ be the solution by Proposition \ref{prop4.2} 
near $\mathbf{c}_0=(\lambda_0,0)^{\sim}$, and set $\Phi^+_0(\lambda)
 =\Phi^-_0(\lambda)\Gamma_{0\,*},$ where $\Phi^+_0(\lambda)$ is the analytic
continuation along 
an arc in $\mathcal{A}^0_{\varepsilon}$ in the clockwise direction. 
Then by Proposition \ref{prop4.2},
$$
\Gamma_{0\,*}= \Phi^-_0(\lambda)^{-1} \Phi^+_0(\lambda) = S_2= \begin{pmatrix}
1 & 0 \\  -i & 1  \end{pmatrix}.
$$
%%%%%%%%%%%%%%%%%%%%%%%%%%%%%%%%%
\par
{\bf (8)} For $\Phi^-_0(\lambda)$ and the WKB solution $\Psi_0(\lambda)$ along
$\mathbf{c}_0,$ set $\Phi^-_0(\lambda)=\Psi_0(\lambda)\Gamma_{0-}.$ By the
same argument as in (4), we have
$$
\Gamma_{0-}=\Psi_0(\lambda)^{-1}\Phi^-_0(\lambda)=(\tilde{\zeta}'_0)^{-1/4}
(I+O(t^{-\delta})) \hat{C}_3(\tilde{\lambda}'_0) \begin{pmatrix}
1 & 0 \\ 0 & -\frac{c_3}{c_1-i c_2} \end{pmatrix}
$$
with $\hat{C}_3(\tilde{\lambda}'_0) =\exp(\int^{\tilde{\lambda}'_0}_{\lambda_0}
\Lambda_3(\tau) d\tau )$ for some $\tilde{\zeta}'_0 \asymp \tilde{\lambda}'_0
-\lambda_0.$
%%%%%%%%%%%%%%%%%%%%%%%%%%%%%%%
%%%%%%%%%%%%%%%%%%%%%%%%%%%%%%%%%%
\par
{\bf (9)} For $\Psi_0(\lambda)$ and $\hat{Y}^0_0(\lambda)$ given by \eqref{3.4}, set
$\Psi_0(\lambda)=\hat{Y}^0_0(\lambda)\Gamma_0.$ Then
\begin{align*}
\Gamma_0=& \hat{Y}^0_0(\lambda)^{-1}\Psi_0(\lambda)
\\
=& \exp\Bigl(\frac {2i}3 e^{i\phi} t\lambda^{-1} \sigma_3 \Bigr) 
\frac {\sqrt{2}}i (\sigma_1+\sigma_3)^{-1} (I+O(t^{-\delta}+|\lambda|))
T\exp\Bigl(\int^{\lambda}_{\tilde{\lambda}'_0}\Lambda(\tau) d\tau \Bigr).
\end{align*}
Note that $\mu(t,\lambda)=2ie^{i\phi} \lambda^{-2}+O(1)$ as $\lambda \to 0$
along $\mathbf{c}_0.$ Since
$$ 
(\sigma_1+\sigma_3)^{-1} \lim_{\lambda\to 0}T(\lambda)=\frac 12 (\sigma_1
+\sigma_3) \begin{pmatrix} 1 & -1 \\ 1 & 1 \end{pmatrix}
=\sigma_3,
$$
we have
$$
\Gamma_0 = \hat{C}_3(\tilde{\lambda}'_0)^{-1} \hat{c}_I(\tilde{\lambda}'_0)
(\sigma_3 +O(t^{-\delta})) \exp\biggl(\lim_{\substack{\lambda \to 0 \\
\lambda \in \mathbf{c}_0 }} \Bigl(\int^{\lambda}_{\lambda_0} \Lambda_3(\tau)
d\tau +\frac{2i}3 e^{i\phi}t\lambda^{-1} \sigma_3 \Bigr) \biggr)
$$
with $\hat{c}_I(\tilde{\lambda}'_0) = - \sqrt{2}\, i \exp(\int^0_{\tilde{\lambda}
'_0}\Lambda_I(\tau)d\tau).$
%%%%%%%%%%%%%%%%%%%%%%%%%%%%%%%%%
\par
Collecting the matrices above, we have the connection matrix
\begin{align*}
\hat{G}=& G \Theta_{0,*}^{\sigma_3}=
 \hat{Y}^0_0(\lambda)^{-1} {Y}^{\infty,*}_0(\lambda)\Theta_{0,*}^{\sigma_3}=
 \hat{Y}^0_0(\lambda)^{-1} \hat{Y}^{\infty}_0(\lambda)
\\
=& \Gamma_0 \Gamma_{0-}\Gamma_{0\,*} \Gamma_{0+} \Gamma_{01} \Gamma_{1-}
\Gamma_{1*} \Gamma_{1+}\Gamma_{\infty}
\\
=&\epsilon_+ i (\sigma_3+O(t^{-\delta}) ) \exp(J_0\sigma_3)
\begin{pmatrix} 1 & 0 \\ 0 & -c_0^{-1} \end{pmatrix}
\begin{pmatrix} 1 & 0 \\ -i & 1   \end{pmatrix}
\begin{pmatrix} 1 & 0 \\ 0 & -c_0 \end{pmatrix}
\\
&\phantom{--} \times 
 \exp(-J_1\sigma_3)
\begin{pmatrix} 1 & 0 \\ 0 & -d_0^{-1} \end{pmatrix}
\begin{pmatrix} 1 & 0 \\ -i & 1   \end{pmatrix}
\begin{pmatrix} 1 & -i \\ 0 & 1   \end{pmatrix}
\begin{pmatrix} 1 & 0 \\ 0 & -{d_0}  \end{pmatrix}
\exp(-J_{\infty} \sigma_3)
\\
=& \epsilon_+  (I+O(t^{-\delta}))
\\
&\times
\begin{pmatrix}  i\exp(J_0 -J_1-J_{\infty})  & -d_0\exp(J_0-J_1+J_{\infty}) \\
(c^{-1}_0 \exp(-J_1)+d^{-1}_0\exp(J_1))\exp(-J_0-J_{\infty}) &
i c^{-1}_0d_0 \exp(-J_0-J_1+J_{\infty})   \end{pmatrix}
\end{align*}
if $0<\phi<\pi/3$, where $\epsilon_+^2=1,$ 
$ c_0=(c_1-i c_2)/c_3,$ $ d_0=(d_1-i d_2)/d_3$, and
%%%%%%%%% (5.1),  (5.2) %%%%%%%%%%%%%%%%%%%%
\begin{align}\label{5.1}
& J_0\sigma_3 = \lim_{\substack{\lambda\to 0 \\ \lambda\in \mathbf{c}_0}}
\Bigl(\int^{\lambda}_{\lambda_0} \Lambda_3(\tau) d\tau +\frac{2i}3
e^{i\phi} t\lambda^{-1} \sigma_3\Bigr), \qquad J_1\sigma_3 =\int^{\lambda_1}
_{\lambda_0} \Lambda_3(\tau)d\tau \,\,\,\, \text{(along $\mathbf{c}_1$),}
\\
\label{5.2}
& J_{\infty}\sigma_3 = \lim_{\substack{\lambda\to \infty \\ 
\lambda\in \mathbf{c}_{\infty}}}
\Bigl(\int^{\lambda}_{\lambda_1} \Lambda_{3}(\tau) d\tau +\frac{1}6
(ie^{i\phi}t\lambda^2 +3(1+2ia)\log \lambda)\sigma_3 \Bigr).
\end{align}
%%%%%%%%%%%%%%%%%%%%%%%%%%%%%%%%%%%%%%%%%%
\par
In the case $-\pi/3 <\phi <0$, from the analytic continuation along the Stokes
curves as in Figure \ref{curve-}, it follows that
\begin{align*}
\hat{G}=&\epsilon_- i (\sigma_3+O(t^{-\delta}) ) \exp(J_0\sigma_3)
\begin{pmatrix} 1 & 0 \\ 0 & -c^{-1}_0 \end{pmatrix}
\begin{pmatrix} 1 & i \\ 0 & 1   \end{pmatrix}
\begin{pmatrix} 1 & 0 \\ 0 & -c_0 \end{pmatrix}
\\
&\phantom{--} \times 
 \exp(-\hat{J}_1\sigma_3)
\begin{pmatrix} 1 & 0 \\ 0 & -d_0^{-1} \end{pmatrix}
\begin{pmatrix} 1 & i \\ 0 & 1   \end{pmatrix}
\begin{pmatrix} 1 & 0 \\ i & 1   \end{pmatrix}
\begin{pmatrix} 1 & 0 \\ 0 & -d_0  \end{pmatrix}
\exp(-J_{\infty} \sigma_3)
\\
=& \epsilon_-  (I+O(t^{-\delta}))
\\
&\times
\begin{pmatrix} - ic_0d_0^{-1}\exp(J_0 +\hat{J}_1-J_{\infty})  &
(c_0 \exp(\hat{J}_1)+d_0\exp(-\hat{J}_1))\exp(J_0+J_{\infty}) \\
- d_0^{-1}\exp(-J_0+\hat{J}_1-J_{\infty})  &
- i \exp(-J_0+\hat{J}_1+J_{\infty})   \end{pmatrix}.
\end{align*}
Here $\epsilon_-^2=1,$ and
%%%%%% (5.3) %%%%%%%%%%%%%%%%%%%%%%%%
%%%%%%%%%%%%%%%%%%%%%%%%%%%%%
\begin{equation}\label{5.3}
\hat{J}_1\sigma_3 = \int^{\lambda_1}_{\lambda_0} \Lambda_3(\tau) d\tau \,\,\,\,
\text{(along $\hat{\mathbf{c}}_1$),}
\end{equation}
in which $\hat{\mathbf{c}}_1$ is a curve joining $\lambda_0$ to $\lambda_1$
located along the upper shore of the cut on the upper sheet of $\mathcal{R}
_{\phi}$.
%%%%%%%%%%%%%%%%%%%%%%%%%%%%%%%%%%%%%%%%%%%%%%%%
%%%%% Figure 4 %%%%%%%%%%%%%%%%%%%%%%%%%%
%%%%%%%%%%%%%%%%%%%%%%%%%%%%%%%%%%%%%%%%%%%%%%%%
{\small
\begin{figure}[htb]
\begin{center}
\unitlength=1mm
%%%%%%%%%%%%%%%%%%%%%%%%%%%%%%%%%%
%%%%%%%%%%%%%%%%%%%%%%%%%%%%%%%%%%
%%%%%%%%%%%%%%%%%%
\begin{picture}(100,50)(0,-25)
\put(-5,-3){\makebox{$0$}}
\put(18,-17){\makebox{$\lambda_0$}}
\put(55,4){\makebox{$\lambda_1$}}
\put(88,13){\makebox{$\infty$}}
\put(12,-1){\makebox{$\mathbf{c}_0$}}
\put(38,-5){\makebox{$\hat{\mathbf{c}}_1$}}
\put(71,8){\makebox{$\mathbf{c}_{\infty}$}}

\thinlines
\qbezier(32,-14) (48,-10.5) (51,3)
\qbezier(25.5,-15.7) (52,-14) (52.7,9.5)

%   \qbezier(58.5,9) (58.5,5) (54,4)
% \put(54,4){\vector(-1,0){0}}      %%%%%%%%%%%%%%%%%%%%%%%%%%%%%%%%%%%%%%
   \qbezier(58.5,13) (56,18) (50,16)
 \put(50,16){\vector(-3,-2){0}}      

   \qbezier(48,4) (43,7.8) (46,13)
 \put(48,4){\vector(3,-1){0}}

  \qbezier(30.5,-12) (28.5,-7) (23,-9)
\put(23,-9){\vector(-3,-2){0}}

\qbezier(40,22) (46,18) (52,10)
\qbezier(40,22.06) (46,18.06) (52,10.06)

\thicklines
\put(0,0){\circle{1.5}}
\put(25,-15){\circle*{1.5}}
\put(52,10){\circle*{1.5}}
\qbezier(0.8,0) (16,-3) (25,-15)
\qbezier(0.8,0.08) (16,-3.08) (25,-15.08)
\qbezier(25,-15) (48,-11) (52,10)
\qbezier(24.92,-15) (47.92,-11) (51.92,10)

\qbezier[12](25,-15) (24,-20) (20,-27)
\qbezier(85,15) (65,11) (52,10)

\end{picture}
\end{center}
\caption{Stokes curve for $-\pi/3<\phi <0$}
\label{curve-}
\end{figure}
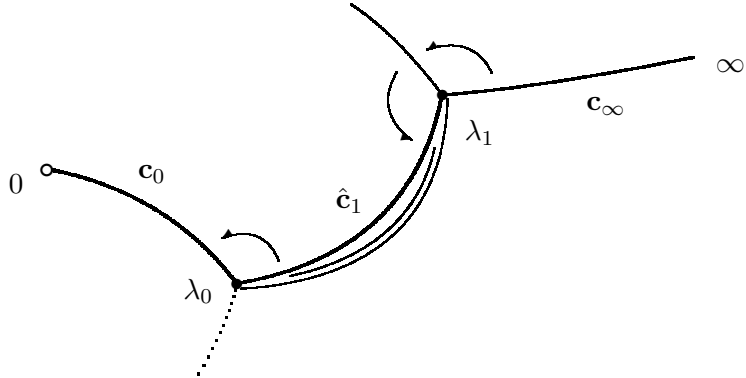
}
%%%%%%%%%%%%%%%%%%%%%%%%%%%%%%%%%%%%%%%%%%%%%%%
%%%%%%%%%%%%%%%%%%%%%%%%%%%%%%%%%%%%%%%%%%%
Thus we have the following:
%%%%%% Proposition 5.1 %%%%%%%%%%
\begin{prop}\label{prop5.1}
Let $c_0=(c_1-i c_2)/c_3,$ $d_0=(d_1-i d_2)/d_3$ with $c_k=b_k(\lambda_0),$
$d_k=b_k(\lambda_1)$ for $k=1,2,3.$ If $0<\phi <\pi/3,$ then
\begin{align*}
\hat{G}=& \epsilon_+  (I+O(t^{-\delta}))
\\
&\times
\begin{pmatrix}  i\exp(J_0 -J_1-J_{\infty})  & -d_0\exp(J_0-J_1+J_{\infty}) \\
(c^{-1}_0 \exp(-J_1)+d^{-1}_0\exp(J_1))\exp(-J_0-J_{\infty}) &
i c^{-1}_0d_0 \exp(-J_0-J_1+J_{\infty})   \end{pmatrix},
\end{align*}
and, if $-\pi/3 <\phi <0,$ then
\begin{align*}
\hat{G}=& \epsilon_-  (I+O(t^{-\delta}))
\\
&\times
\begin{pmatrix} - ic_0d_0^{-1}\exp(J_0 +\hat{J}_1-J_{\infty})  &
(c_0 \exp(\hat{J}_1)+d_0\exp(-\hat{J}_1))\exp(J_0+J_{\infty}) \\
- d_0^{-1}\exp(-J_0+\hat{J}_1-J_{\infty})  &
- i \exp(-J_0+\hat{J}_1+J_{\infty})   \end{pmatrix}.
\end{align*}
Here $\epsilon_{\pm}^2=1,$ and $J_0,$ $J_1$, $\hat{J}_1$, $J_{\infty}$ are
integrals given by \eqref{5.1}, \eqref{5.2}, and \eqref{5.3}. 
\end{prop}
%%%%%%%%%%%%%%%%%%
From the proposition above with $\hat{G}=G\Theta_{0,*}^{\sigma_3},$ 
$G=(g_{ij})$ (Remark \ref{rem3.2}), we derive key relations.
%%%%%% Corollary 5.2 %%%%%
\begin{cor}\label{cor5.2}
If $0<\phi <\pi/3$ and ${g}_{11}{g}_{12}{g}_{22}\not=0,$ then
$$
{g}_{11}{g}_{22} = -c^{-1}_{0} d_0 (1+O(t^{-\delta})) \exp(-2J_1), \quad
\frac{{g}_{12}}{{g}_{22}} = ic_{0} (1+O(t^{-\delta})) \exp(2J_0). 
$$
If $-\pi/3<\phi<0$ and ${g}_{11}{g}_{21}{g}_{22}\not=0,$ then
$$
{g}_{11}{g}_{22} = -c_{0} d_0^{-1} (1+O(t^{-\delta}))
 \exp(2\hat{J}_1), \quad
\frac{{g}_{21}}{{g}_{11}} = -ic_{0}^{-1} (1+O(t^{-\delta})) \exp(-2J_0). 
$$
\end{cor}
%%%%%%%%%%%%%%%%%%%%%%%%%%%%%%%%%%%%%%%
%%%%%%%%%%%%%%%%%%%%%%%%%%%%%%%%%%%%%%%%
%%%% Section 6 %%%%%%%%%%%%%%%%%%
%%%%%%%%%%%%%%%%%%%%%%%%%%%
\section{Asymptotic properties of monodromy data}\label{sc6}
%%%%%%%%%%%%%%%%%%%%%%%%%%%%%%%%%%%%%%%%%%%%%%%%%%%%%%%
%%%%%%%%%%%%%%% 6.1 %%%%%%%%%%%%%%%%%
\subsection{Expressions of $J_0,$ $J_1$ and $\hat{J}_1$}\label{ssc6.1}
%%%%%%%%%%%%%%%%%%%%%%%%%%%%%%%%%%%%%%%%%%%%%%
To examine asymptotic properties of $J_0,$ $J_1$ and $\hat{J}_1$, we make
the change of variables $\lambda^{-2}=z.$ Then, by \eqref{4.1} and \eqref{4.2},
$\mu(t,\lambda)$ becomes
\begin{align*}
\mu(t,\lambda)d\lambda =& \Bigl(-\frac{e^{2i\phi}}z + e^{2i\phi} a_{\phi}z
-4e^{2i\phi}z^2 +3i e^{i\phi} (1+2ia)t^{-1}\Bigr)^{\!1/2} \frac{(-z^{-3/2})}
 2 dz
\\
=& \Bigl(-\frac i2 e^{i\phi} \frac{w(z)}{z^2} -\frac 34 (1+2ia)t^{-1} \frac
1{zw(z)} + O(t^{-2}w(z)^{-3}) \Bigr) dz
\end{align*}
with $w(z)^2 = w(a_{\phi},z)^2 = 4z^3 -a_{\phi} z^2 +1,$ for $z$ such that
$w(z)\gg 1.$ The turning points
$\lambda_0,$ $\lambda_1,$ $\lambda_2$ and $0$ on $\mathcal{R}_{\phi}$ are
mapped to 
$$
z_0=\lambda_0^{-2},\quad z_1=\lambda_1^{-2}, \quad  z_2=\lambda_2^{-2}
$$
and $\infty$, respectively, on the elliptic curve $\Pi_{a_{\phi}}$ for
$w(a_{\phi},z)$ constructed by the same way as in the case of $\Pi_{A_{\phi}}$  
in Section \ref{ssc2.2}. The branch of $\mu(t,\lambda)$ is compatible with
that of $w(a_{\phi},z).$ 
Suppose that $\Pi_{a_{\phi}}$ is equipped with the
cycles $\mathbf{a}$ and $\mathbf{b}$ as in Section \ref{ssc2.2}. Then 
the inverse image of the cycle $\mathbf{a}$ is a
closed curve $\mathbf{a}_{\lambda}$ surrounding the cut $[\lambda_0, \lambda_1]$
anticlockwise (see Figure \ref{cycles2}).
\vskip0.6cm
%%%%%%%%%%%%%%%%%%%%%%%%%%%%%%%%%%%%%%%%%%%%%%%%%%
%%%%%%%%%%%%%%%%%%%%%%%%%%%%%%%%%%%%%%%%%%%%%%%%%%%
%%%%%%%%%% Figure 5 %%%%%%%%%%%%%%%%%%%%%%
%%%%%%%%%%%%%%%%%%%%%%%%%%%%%%%%%%%%%%%%%
%%%%%%%%%%%%%%%%%%%%%%%%%%%%%%%%%%%%%%%%
{\small
\begin{figure}[htb]
%%%%%%%%%%%%%%%%%%%%%%%%%%%%%%%%%%%%%%%%
%%%%%%%%%%%% \phi <0 %%%%%%%%%%%%%%%%%%%%%%%
%%%%%%%%%%%%%%%%%%%%%%%%%%%%%%%%5555
\begin{center}
\unitlength=0.60mm
%%%%%%%%%%%%%%%%%%%%%%%%%%%%%%%%%%
%%%%%%%%%%%%%%%%%%
\begin{picture}(100,60)(0,-30)
\put(-12,-5){\makebox{$0$}}
\put(14.5,-16){\makebox{$\lambda_0$}}
\put(51,17){\makebox{$\lambda_1$}}
\put(13,40){\makebox{$\lambda_2$}}
\put(73,15){\makebox{$\infty$}}
\put(7,-8){\makebox{$\mathbf{c}_0$}}
\put(39,-3){\makebox{$\hat{\mathbf{c}}_1$}}
%\put(71,-9){\makebox{$\mathbf{c}_{\infty}$}}
\put(55,-19){\makebox{$\mathbf{a}_{\lambda}$}}
\put(-1,-25){\makebox{$\mathbf{b}_{\lambda}$}}

\qbezier(32.3,-15) (34,-24) (25,-23)
\qbezier(-4,-7.1) (3.5,-18.5) (25,-23)

\qbezier(-4,-7.1) (-9,2) (-4,4.5)
\thicklines
\qbezier[35](32,-13) (25,6) (-3,5)

\thinlines
\qbezier(27,-21.5) (56.5,-25) (58.5,7)
\qbezier(24,-8) (32,6.5) (46.5,13)

\qbezier(24,-8) (19,-20) (27,-21.5)
\qbezier(46.5,13) (57,17) (58.5,7)

\qbezier(32,-14) (48,-10.5) (51,3)
\qbezier(25.5,-15.7) (52,-14) (52.7,9.5)

\qbezier(52.1,-18) (59.6,-9) (60.1,-1)
\put(60.1,-1){\vector(-1,4){0}}

\qbezier(1.5,-16.7) (9.7,-22.4) (18,-24)
\put(18,-24){\vector(1,0){0}}

\qbezier(-2,36) (7,40) (10,40)

\qbezier(-0.2,1.2) (-18.3,25.5) (10,39.3)
\qbezier(-1.0,0.7) (-18.7,25) (6.8,38.5)

\qbezier(25,-15) (22.7,-24) (20,-28)

\thicklines

\qbezier[20](32,25.4) (46,18) (52,10)
\qbezier[20](32,25.4) (21,32) (10,40)

\put(0,0){\circle{1.5}}
\put(25,-15){\circle*{1.5}}
\put(52,10){\circle*{1.5}}
\put(10,40){\circle*{1.5}}
\qbezier(0.8,0) (16,-3) (25,-15)
\qbezier(25,-15) (48,-11) (52,10)

%\qbezier[20](25,-15) (22,-31) (10,-40)
\qbezier[5](10.5,45) (10,42.5) (10,40)

%\qbezier[13](70,13.5) (64,11.1) (52,10)
\qbezier(70,13.5) (64,11.1) (52,10)

\end{picture}\quad
\unitlength=0.75mm
%%%%%%%%%%%%%%%%%%%%%%%%%%%%%%%%%%
%%%%%%%%%%%%%%%%%%%%%%%%%%%%%%%%%%
%%%% \phi <0 %%%%%%%%%
%%%%%%%%%%%%%%%%%%%%%%%
%%%%%%%%%%%%%%%%%%
\begin{picture}(80,50)(-10,-5)
 \put(-8,22){\makebox{$\infty$}}
 \put(20,12){\makebox{$z_2$}}
 \put(65,32){\makebox{$z_0$}}
 \put(53,2){\makebox{$z_1$}}
  \put(20,0){\makebox{$\Pi_{a_{\phi}}$}}
\thinlines
 \put(25,20.5){\line(-1,0){25}}
 \put(25,19.6){\line(-1,0){25}}
\put(50.5,10){\line(1,2){10}}
\put(49.5,10){\line(1,2){10}}
  \qbezier(66,20) (64,14) (61.5,11.5)
    \qbezier(27,26.3) (30,26.8) (37,25.0)
  \put(66,20){\vector(0,1){0}}
 \put(37,25){\vector(3,-2){0}}
  \put(65,10){\makebox{$\mathbf{a}$}}
  \put(27,29){\makebox{$\mathbf{b}$}}
\put(-20,-5){\makebox{$-\pi/3 <\phi <0$}}
\thicklines
 \put(25,20){\circle*{1}}
\put(50,10){\circle*{1}}
\put(60,30){\circle*{1}}
 \qbezier[20](22,19.5) (23,11.5) (44,7.5)
 \qbezier[7](48.4,7.5) (54.2,8) (52.2,13.1)
  \qbezier(37,22) (42,20.5) (44.5,19)
  \qbezier(47.6,17.5) (48.3,17.8)  (51.9,14.1)
 \qbezier(22,20.4) (23,25.5) (37,22)
\qbezier(46,8) (50,2.5) (57.7,10.5)
\qbezier(64,32) (66,26) (63.2,20)
\qbezier(46,8) (44,14) (46.8,20)
\qbezier (64,32)(60,37.5)(52.3,29.5)
%%%%%%%
\qbezier (63.2,20)(60.8,14)(57.7,10.5)
\qbezier (46.8,20)(49.2,26)(52.3,29.5)
\end{picture} 
\end{center}
%%%%%%%%%%%%%%%%%%%%%%%%%%%%%%%%%%%%%%%%%%%%%%%%%%%%%%%%%
\vskip0.8cm
%%%%%%%%%%%%%%%%%%%%%%%%%%%%%%%%%%%%%%%%

\begin{center}
\unitlength=0.67mm
%%%%%%%%%%%%%%%%%%%%%%%%%%%%%%%%%%
%%%%%%%%%%%%%  \phi >0 %%%%%%%%%%%%%%%%%%%%%
%%%%%%%%%%%%%%%%%%
\begin{picture}(100,60)(0,-20)
\put(-10,-2){\makebox{$0$}}
\put(16,14){\makebox{$\lambda_0$}}
\put(53,-19){\makebox{$\lambda_1$}}
\put(13,40){\makebox{$\lambda_2$}}
\put(73,-15){\makebox{$\infty$}}
\put(12,0){\makebox{$\mathbf{c}_0$}}
\put(38,4){\makebox{$\mathbf{c}_1$}}
%\put(71,-9){\makebox{$\mathbf{c}_{\infty}$}}
\put(62,3.5){\makebox{$\mathbf{a}_{\lambda}$}}
\put(1.5,-15){\makebox{$\mathbf{b}_{\lambda}$}}
\thicklines
\qbezier[10](32.3,15) (34,24) (25,23)
\qbezier[25](-4,7.1) (3.5,18.5) (25,23)
\thinlines
\qbezier(-4.5,6.5) (-9,-5) (2,-7)
\qbezier(32,13) (25,-6) (2,-7)
\qbezier(27,21.5) (56.5,25) (58.5,-7)
\qbezier(24,8) (32,-6.5) (46.5,-13)
\qbezier(24,8) (19,20) (27,21.5)
\qbezier(46.5,-13) (57,-17) (58.5,-7)

\qbezier(32,14) (48,10.5) (51,-3)
\qbezier(25.5,15.7) (52,14) (52.7,-9.5)

\qbezier(53.5,16) (58.7,10) (59.5,3.5)
\put(53.5,16){\vector(-3,2){0}}

\qbezier(1.5,-9.7) (9.7,-10.4) (18,-6.5)
\put(18,-6.5){\vector(1,1){0}}

\qbezier(40,-22) (46,-18) (52,-10)

\qbezier(-2,36) (7,40) (10,40)

\qbezier(-0.2,1.2) (-18.3,25.5) (10,39.3)
\qbezier(-1.0,0.7) (-18.7,25) (6.8,38.5)

\thicklines
\put(0,0){\circle{1.5}}
\put(25,15){\circle*{1.5}}
\put(52,-10){\circle*{1.5}}
\put(10,40){\circle*{1.5}}
\qbezier(0.8,0) (16,3) (25,15)
\qbezier(25,15) (48,11) (52,-10)

\qbezier[20](25,15) (22,31) (10,40)
\qbezier[5](10.5,45) (10,42.5) (10,40)

\qbezier(70,-13.5) (64,-11.1) (52,-10)

\end{picture}\quad
%%%%%%%%%%%%%%%%%%%%%%%%
\unitlength=0.75mm
%%%%% \phi >0 %%%%%%
%%%%%%%%%%%%%%%%%%
\begin{picture}(80,50)(-10,-5)
 \put(-8,22){\makebox{$\infty$}}
 \put(20,12){\makebox{$z_2$}}
 \put(65,3){\makebox{$z_0$}}
 \put(43,37){\makebox{$z_1$}}
  \put(20,0){\makebox{$\Pi_{a_{\phi}}$}}
\thinlines
 \put(25,20.5){\line(-1,0){25}}
 \put(25,19.6){\line(-1,0){25}}
\put(50.5,30){\line(1,-2){10}}
\put(49.5,30){\line(1,-2){10}}
  \qbezier(66,20) (64,26) (60.9,29.5)
    \qbezier(27,29.7) (30,32.1) (37,33.5)
  \put(60.7,29.5){\vector(-2,1){0}}
 \put(37,33.5){\vector(1,0){0}}
  \put(68,21){\makebox{$\mathbf{a}$}}
  \put(22,32){\makebox{$\mathbf{b}$}}
\put(-20,-5){\makebox{$0<\phi<\pi/3$}}
\thicklines
 \put(25,20){\circle*{1}}
\put(50,30){\circle*{1}}
\put(60,10){\circle*{1}}
 \qbezier(22,20.2) (23,28.5) (44,32.5)
 \qbezier(48.4,32.5) (54.2,32) (52.2,26.2)
  \qbezier[10](37,18) (48,21.5) (51.2,25.5)
  \qbezier[15](22,19) (23,14.5) (37,18)
\qbezier(46,32) (50,37.5) (57.7,29.5)
\qbezier(64,8) (66,14) (63.2,20)
\qbezier(46,32) (44,26) (46.8,20)
\qbezier (64,8)(60,2.5)(52.3,10.5)
%%%%%%%
\qbezier (63.2,20)(60.8,26)(57.7,29.5)
\qbezier (46.8,20)(49.2,14)(52.3,10.5)
\end{picture}
\end{center}

%%%%%%%%%%%%%%%%%%%%%%%%%%%%%%%%%%%%%%%%%%
\caption{Correspondence of the cycles under the map $z=\lambda^{-2}$}
\label{cycles2}
\end{figure}
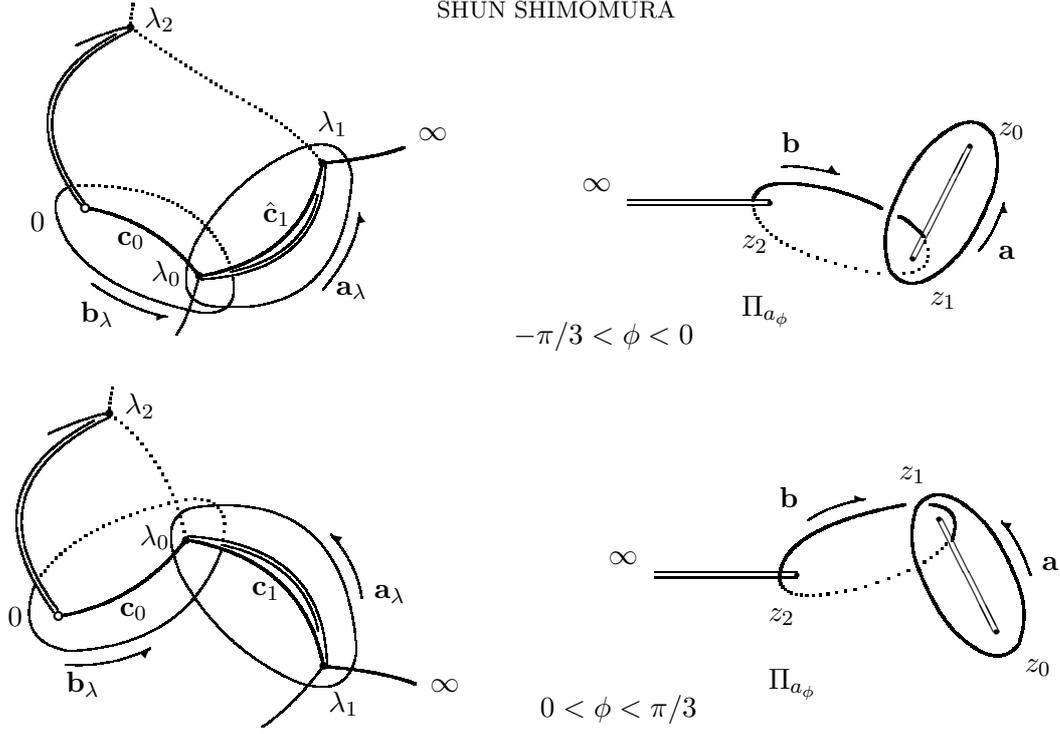
%%%%%%%%%%%%%%%%%%%%%%%%%%%%%%%%%%%%%%%%%
}
%%%%%%%%%%%%%%%%%%%%%%%%%%%%%%%%%%
%%%%%%%%%%%%%%%%%%%%%%%%%%%%%%%%%%%%
%%%%%%%%%%%%%%%%%%%%%%%%%%%%
Since
$$
\int \frac{w(z)}{z^2} dz= 2\frac{w(z)}z - a_{\phi} \int\frac{dz}{w(z)}
+ 3\int \frac {dz}{z^2w(z)},
$$
we have
\begin{align*}
\mu(t,\lambda) d\lambda =- i e^{i\phi} \frac{w(z)}{z}
 &+ \frac i2 e^{i\phi}
a_{\phi} \frac{dz}{w(z)} -\frac{3i}2 e^{i\phi}\frac{dz}{z^2w(z)} 
\\
 &-\frac 34(1+2ia)t^{-1}\frac{dz}{zw(z)} +O(t^{-2}w(z)^{-3}) dz,
\end{align*}
in which $w(z)/z=2z^{1/2} +O(z^{-1/2})$ as $z \to \infty.$ Hence 
%%%%%%%%%%%% (6.1) %%%%%%%%%%%%%%%%%%%
\begin{align}\label{6.1}
\lim_{\substack{\lambda\to 0 \\ \lambda\in\mathbf{c}_0}} & \Bigl(\int^{\lambda}
_{\lambda_0} \mu(t, \tau)d\tau +2i e^{i\phi} \lambda^{-1} \Bigr)
\\
\notag
=& -\frac i4 e^{i\phi}a_{\phi} \int_{\mathbf{b}} \frac{dz}{w(z)}
+\frac{3i}4 e^{i\phi} \int_{\mathbf{b}} \frac{dz}{z^2w(z)}  +\frac 38 (1+2ia)
t^{-1} \int_{\mathbf{b}}\frac {dz}{zw(z)} +O(t^{-2})
\\
\notag
=& \frac i4 e^{i\phi} \int_{\mathbf{b}} \frac{w(z)}{z^2}dz  +\frac 38 (1+2ia)
t^{-1} \int_{\mathbf{b}}\frac {dz}{zw(z)} +O(t^{-2}),
\end{align}
and
%%%%%%%%%%% (6.2) %%%%%%%%%%%%%%%%%%%%%%%%%
\begin{align}\label{6.2}
\int^{\lambda_1}_{\lambda_0\, (\mathbf{c}_1)} 
 &\mu(t, \tau)d\tau , \,\,\,\, 
- \int^{\lambda_1}_{\lambda_0 \, (\hat{\mathbf{c}}_1)} 
\mu(t, \tau)d\tau  
\\
\notag
=& \frac i4 e^{i\phi}a_{\phi} \int_{\mathbf{a}} \frac{dz}{w(z)}
-\frac{3i}4 e^{i\phi} \int_{\mathbf{a}} \frac{dz}{z^2w(z)}  -\frac 38 (1+2ia)
t^{-1} \int_{\mathbf{a}}\frac {dz}{zw(z)} +O(t^{-2})
\\
\notag
=& -\frac i4 e^{i\phi} \int_{\mathbf{a}} \frac{w(z)}{z^2}dz  -\frac 38 (1+2ia)
t^{-1} \int_{\mathbf{a}}\frac {dz}{zw(z)} +O(t^{-2}),
\end{align}
in which $\int^{\lambda_1}_{\lambda_0 \,(\mathbf{c})}$ denotes the integral 
along the contour $\mathbf{c}.$
By Remark \ref{rem4.1},
%%%%%%% (6.3) %%%%%%%%%%%%%%%%%%
\begin{align}\label{6.3}
\Lambda_3(t,\lambda)=&\Bigl(\frac t3 \mu(t,\lambda) 
-\diag T^{-1}T_{\lambda}|_{\sigma_3}\Bigr)\sigma_3, 
\\ 
\notag
\diag T^{-1}  & T_{\lambda}|_{\sigma_3} = \frac 14 \Bigl(1-
\frac{b_3}{\mu} \Bigr) \frac{\partial}{\partial \lambda} \log \frac{b_1+ib_2}
{b_1-ib_2}.
\end{align}
To calculate $J_0$, $J_1$ and $\hat{J}_1$, it is necessary to know 
$\diag T^{-1}T_{\lambda}|_{\sigma_3}$ in addition to \eqref{6.1} and 
\eqref{6.2}. Note that, by \eqref{3.2},
$$
b_1= 2i e^{i\phi}\lambda^{-2} -i K_+, \quad b_2=K_-, \quad b_3= -i e^{i\phi}
\lambda -K_0 \lambda^{-1}
$$
with $K_{\pm}=e^{i\phi} y \pm (i/2)y^{-1}\Gamma_0(t, y,y^t),$ 
$K_0=\Gamma_0(t,y,y^t)+3(1/2+ia)t^{-1},$ 
$\Gamma_0(t,y,y^t)=y^ty^{-1} -i e^{i\phi}y^{-1} -(1+3ia)t^{-1}.$
Setting $z_{\pm}=e^{-i\phi}(K_+\pm K_-)/2,$ i.e.
%%%%%%% (6.4) %%%%%%%%%%%%%%%%%%%%%%%%%%%%%%%%%%%%%%%%%%%%%
\begin{equation}\label{6.4}
z_+=y, \quad z_-= (i/2)e^{-i\phi}y^{-1}\Gamma_0(t,y,y^t),
\end{equation}
%%%%%%%%%%%%%%%%%%%%%
and $\lambda^{-2}=z,$ we have
%%%%%%% (6.5) %%%%%%%%%
\begin{equation}\label{6.5}
b_1-i b_2=2i e^{i\phi} (z-z_+),\quad  b_1+ ib_2 = 2i e^{i\phi}(z-z_-).
\end{equation}
By \eqref{4.1}, $\mu^2= -e^{2i\phi} \lambda^2 w(z)^2+O(t^{-1}) ,$ which
implies $\mu= i e^{i\phi}\lambda (w(z)+O(t^{-1}z))$ on the upper sheet of
$\Pi_{a_{\phi}}$, and hence
$$
\frac{b_3}{\mu} = - i e^{-i\phi}\frac {b_3}{\lambda} \Bigl(\frac 1{w(z)} +
O(t^{-1}z^{-2}) \Bigr),
$$
where $b_3/\lambda=-K_0 z-ie^{i\phi}$ satisfies $(b_3/\lambda)(z_{\pm}) =
-(\mu/\lambda)(z_{\pm}) =- i e^{i\phi}w(z_{\pm}) +O(t^{-1}),$ since
$\mu(z_{\pm})^2 =(b_1-ib_2)(b_1+i b_2)(z_{\pm}) +b_3(z_{\pm})^2 =b_3(z_{\pm})^2$
by \eqref{6.5}. These facts combined with \eqref{6.5} yield
\begin{align*}
\diag T^{-1}T_{\lambda}|_{\sigma_3} d\lambda
=& \frac 14 \Bigl(1-\frac{b_3}{\mu}\Bigr) \frac d{d\lambda} \log\frac{b_1+ib_2}
{b_1-i b_2} d\lambda
\\
=& \frac 14 \Bigl(1-\frac{b_3}{\mu}\Bigr) \frac d{dz} \log\frac{b_1+ib_2}
{b_1-i b_2} dz
\\
=& \frac 14\Bigl( 1+ i e^{-i\phi} \frac{b_3}{\lambda} \Bigl(\frac 1{w(z)}
+O(t^{-1}z^{-2}) \Bigr)\Bigr) \Bigl(\frac 1{z-z_-} -\frac 1{z-z_+} \Bigr)dz
\\
=& -\frac 14 \Bigl( \frac 1{z-z_+} -\frac 1{z-z_-} + \Bigl(\frac{w(z_+)}{z-z_+}
-\frac{w(z_-)}{z-z_-}\Bigr) \frac 1{w(z)}  +O(t^{-1}z^{-2})\Bigr) dz,
\end{align*}
which implies
$$
-\lim_{\substack{\lambda\to 0 \\ \lambda \in \mathbf{c}_0}} \int^{\lambda}
_{\lambda_0} \diag T^{-1}T_{\lambda}|_{\sigma_3} d\lambda
= -\frac 14 \log\frac{z_0-z_+}{z_0-z_-} -\frac 18 \int_{\mathbf{b}} \Bigl(
\frac{w(z_+)}{z-z_+} -\frac{w(z_-)}{z-z_-} \Bigr)\frac{dz}{w(z)} +O(t^{-1}).
$$
Here, by \eqref{6.5}, $c_0^2= (c_1-ic_2)^2/c_3^2 =-(c_1-ic_2)/(c_1+ic_2)
=-(z_0-z_+)/(z_0-z_-)$ and $\log((z_0-z_+)/(z_0-z_-))=\log(-c_0^2)=2\log(ic_0).$
Similarly,
\begin{align*}
-\int^{\lambda_1}_{\lambda_0 \,\, (\mathbf{c}_1)} \!\! \diag T^{-1}T_{\lambda}
|_{\sigma_3} d\lambda+\frac 12 \log(c_0d_0^{-1}),& \,\,\,
\int^{\lambda_1}_{\lambda_0 \,\,(\hat{\mathbf{c}}_1)}\!\! \diag T^{-1}T_{\lambda}
|_{\sigma_3} d\lambda -\frac 12 \log(c_0d_0^{-1})
\\
=&
\frac 18 \int_{\mathbf{a}} \Bigl(\frac{w(z_+)}{z-z_+} -\frac{w(z_-)}{z-z_-}
\Bigr)\frac{dz}{w(z)} +O(t^{-1}).
\end{align*}
Insertion of \eqref{6.1}, \eqref{6.2} and the relations above into \eqref{5.1},
\eqref{5.3} with \eqref{6.3} provides the expressions of $J_0$, $J_1$ and
$\hat{J}_1.$ Then by Corollary \ref{cor5.2} we have 
%%%%%%%%%%%%%%%%%%%%%%%%%%%%%%%%%%%%%%%%%%%%%%%%%
%%%%%% Proposition 6.1 %%%%%%%%%%%%%%%%%%%
\begin{prop}\label{prop6.1}
Let
$$
W(z)=\Bigl(\frac{w(z_+)}{z-z_+} -\frac{w(z_-)}{z-z_-} \Bigr)\frac 1{w(z)}.
$$
\par
$(1)$ Suppose that $g_{11}g_{22}\not=0,$ $g_{12}/g_{22} \not=0.$ For $0<\phi
<\pi/3,$
\begin{align*}
\log \frac{g_{12}}{g_{22}} &=\frac{ie^{i\phi}t}{6} \int_{\mathbf{b}} \frac{w(z)}
{z^2} dz -\frac 14 \int_{\mathbf{b}} W(z) dz +\frac 14(1+2ia)\int_{\mathbf{b}}
\frac{dz}{zw(z)} +O(t^{-\delta}),
\\
\log ({g_{11}}{g_{22}}) &=\frac{ie^{i\phi}t}{6} \int_{\mathbf{a}} \frac{w(z)}
{z^2} dz -\frac 14 \int_{\mathbf{a}} W(z) dz +\frac 14(1+2ia)\int_{\mathbf{a}}
\frac{dz}{zw(z)}+ \pi i+O(t^{-\delta}).
\end{align*}
\par
$(2)$ Suppose that $g_{11}g_{22}\not=0,$ $g_{21}/g_{11} \not=0.$ For $-\pi/3
<\phi<0,$
\begin{align*}
\log \frac{g_{21}}{g_{11}} &=-\frac{ie^{i\phi}t}{6} \int_{\mathbf{b}} \frac{w(z)}
{z^2} dz +\frac 14 \int_{\mathbf{b}} W(z) dz -\frac 14(1+2ia)\int_{\mathbf{b}}
\frac{dz}{zw(z)} +O(t^{-\delta}),
\\
\log ({g_{11}}{g_{22}}) &=\frac{ie^{i\phi}t}{6} \int_{\mathbf{a}} \frac{w(z)}
{z^2} dz -\frac 14 \int_{\mathbf{a}} W(z) dz +\frac 14(1+2ia)\int_{\mathbf{a}}
\frac{dz}{zw(z)}+ \pi i+O(t^{-\delta}).
\end{align*}
\end{prop}
%%%%%%%%%%%%%%%%%%%%%%%%%%%%%
%%%% Remark 6.1 %%%%%%
\begin{rem}\label{rem6.1}
In the proposition above
$$
\frac{ie^{i\phi}t}6 \int_{\mathbf{a},\,\, \mathbf{b}} \frac{w(z)}{z^2} dz
= -\frac{ie^{i\phi} a_{\phi} t}6 \int_{\mathbf{a}, \,\, \mathbf{b}}
\frac{dz}{w(z)} +\frac{ie^{i\phi}t}2\int_{\mathbf{a},\,\, \mathbf{b}}
\frac{dz}{z^2w(z)}.
$$
\end{rem}
%%%%%%%%%%%%%%%%%%%%%%%%%%%%%
%%%%% 6.2 %%%%%%%%%%%
%%%%%%%%%%%%%%%%%%%%%%
\subsection{Expressions by the $\vartheta$-function}\label{ssc6.2}
%%%%%%%%%%%%%
For $w(z)^2=w(a_{\phi},z)^2=4z^3-a_{\phi}z^2+1,$ the differential equation
$(dz/du)^2=w(a_{\phi},z)^2$ defines the Weierstrass $\wp$-function
$$
z=\wp(u; g_2, g_3) +\frac{a_{\phi}}{12} , \quad  g_2=\frac{a_{\phi}^2}{12},
\quad g_3= -1+\frac{a_{\phi}^3}{216}.
$$
The periods of $\wp(u; g_2,g_3)$ are
$$
\omega_{\mathbf{a}} =\int_{\mathbf{a}} \frac{dz}{w(a_{\phi},z)}, \quad
\omega_{\mathbf{b}} =\int_{\mathbf{b}} \frac{dz}{w(a_{\phi},z)}, \quad
\tau =\frac{\omega_{\mathbf{b}}}{\omega_{\mathbf{a}}}, \quad
\im \tau >0,
$$
where $\mathbf{a}$ and $\mathbf{b}$ are the cycles on the elliptic curve
$\Pi_{a_{\phi}}=\Pi_+ \cup \Pi_-$ for $w(a_{\phi},z)$ in Section \ref{ssc6.1}
(cf. Figure \ref{cycles2}). The $\vartheta$-function $\vartheta(z,\tau)
=\vartheta(z)$ is defined by
$$
\vartheta(z,\tau) =\sum_{n=-\infty}^{\infty} e^{\pi i\tau n^2 +2\pi izn}
$$
and we set
$$
\nu=\frac{1+\tau}2
$$
(cf. \cite{HC}, \cite{WW}). For $z,$ $\tilde{z} \in \Pi_{a_{\phi}}
=\Pi_+ \cup \Pi_-,$ let
$$
F(\tilde{z}, z)=\frac 1{\omega_{\mathbf{a}}} \int_{\tilde{z}}^z \frac {dz}{w(z)} 
=\frac 1{\omega_{\mathbf{a}}} \int_{\infty}^z \frac {dz}{w(z)} 
-\frac 1{\omega_{\mathbf{a}}} \int^{\tilde{z}}_{\infty} \frac {dz}{w(z)}. 
$$
For any $z_0 \in \Pi_{a_{\phi}}$ denote the projections of $z_0$ 
on the respective sheets
by $z_0^+=(z_0, w(z_0)) =(z_0, w(z_0^+))$ and $z_0^-=(z_0, -w(z_0))
=(z_0, -w(z_0^+)).$  
If $z_0 \in  \Pi_+$ (respectively, $z_0 \in \Pi_-$), then $z_0^{\pm} \in \Pi
_{\pm}$ (respectively, $z_0^{\pm} \in \Pi_{\mp}$). 
%%%%%%%%%%%%%%%%%%%%%%%%%%%%%%%%%%%%%%%%%%%
%%%%%% Proposition 6.2 %%%%%%%
\begin{prop}\label{prop6.2}
For any $z_0 \in \Pi_{a_{\phi}}$
\begin{align*}
\frac{dz}{(z-z_0)w(z)}=& \frac 1{w(z^+_0)} d\log \frac{\vartheta(F(z^+_0, z)
+\nu, \tau)}{\vartheta(F(z^-_0,z)+\nu,\tau)} -g_0(z_0) \frac{dz}{w(z)},
\\
g_0(z_0)= & \frac{w'(z^+_0)}{2w(z^+_0)} -\frac 1{\omega_{\mathbf{a}}}
\frac 1{w(z^+_0)} \Bigl(\pi i +\frac{\vartheta'}{\vartheta} (F(z^-_0, z^+_0)
+\nu, \tau) \Bigr).
\end{align*}
\end{prop}
%%%%%%%%%%%%%%%%
\begin{proof}
For $z_0= \wp(u_0)+ a_{\phi}/12 \in \Pi_{a_{\phi}}$ let $u^{\pm}_0$ be such that
$z^{\pm}_0=\wp(u^{\pm}_0)+ a_{\phi}/12.$ Then
\begin{align*}
\frac{dz}{(z-z_0)w(z)} =&\frac{du}{\wp(u)-\wp(u_0)} 
\\
=& \frac 1{w(z^+_0)} \Bigl(\zeta(u-u^+_0)-\zeta(u-u^-_0) +\zeta(u^+_0 -u^-_0)
-\frac 12 w'(z^+_0) \Bigr) du
\\
=& \frac 1{w(z^+_0)} d\log \frac{\sigma(u-u^+_0)}{\sigma(u-u^-_0)}
+ \frac 1{w(z^+_0)} \Bigl(\zeta(u^+_0-u^-_0) -\frac 12 w'(z^+_0) \Bigr)du.
\end{align*}
From
\begin{align*}
d\log \frac{\sigma(u-u^+_0)}{\sigma(u-u^-_0)} &= -\frac{2\eta_{\mathbf{a}}}
{\omega_{\mathbf{a}}}(u^+_0 -u^-_0) du + d\log \frac{\vartheta(F(z^+_0, z)
+\nu,\tau)}{\vartheta(F(z^-_0,z)+\nu,\tau)},
\\
\zeta(u^+_0 -u^-_0) &= \frac{\sigma'}{\sigma}(u^+_0-u^-_0)=\frac{2\eta_{\mathbf
{a}}}{\omega_{\mathbf{a}}} (u^+_0-u^-_0) +\frac{\pi i}{\omega_{\mathbf{a}}}
+\frac 1{\omega_{\mathbf{a}}} \frac{\vartheta'}{\vartheta}(F(z^-_0, z^+_0)
+\nu,\tau)
\end{align*}
with $F(z^{\pm}_0, z) =\omega_{\mathbf{a}}^{-1}\int^z_{z^{\pm}_0} dz/w(z),$
the desired formula follows.
\end{proof}
%%%%%%%%%%%%%%%%%%%%%%%%%%
Observe that
\begin{align*}
\log \vartheta(F(z^{\pm}_0, z) +\nu,\tau)|_{\mathbf{a}}&=0,
\\
\log \frac{\vartheta(F(z^+_0,z)+\nu,\tau)}{\vartheta(F(z^-_0,z)+\nu,\tau)}
\Bigr|_{\mathbf{b}} &= 
\log \frac{\vartheta(F(z^+_0, z_{\mathbf{b}})+\tau+\nu,\tau)\vartheta(F(z^-_0,
z_{\mathbf{b}})+\nu, \tau)   }
{\vartheta(F(z^-_0, z_{\mathbf{b}})+\tau+\nu,\tau)\vartheta(F(z^+_0,
z_{\mathbf{b}})+\nu, \tau)   }
\\
&= \log \exp(-\pi i (2(F(z^+_0, z_{\mathbf{b}})+\nu) +\tau))
\\
& \phantom{---} + \log
\exp(\pi i (2(F(z^-_0, z_{\mathbf{b}})+\nu) +\tau)) 
\\
&= 2\pi i F(z^-_0, z^+_0)
\end{align*}
for $z_{\mathbf{b}} \in \mathbf{b} \cap (\Pi_+)^{\mathrm{cl}} \cap (\Pi_-)
^{\mathrm{cl}}$, 
since $\vartheta(z\pm \tau,\tau)= e^{-\pi i (\tau\pm 2z)} \vartheta
(z,\tau),$ where $(\Pi_+)^{\mathrm{cl}}$ denotes the closure of $\Pi_+$. 
Then
\begin{align*}
\int_{\mathbf{a}} \frac{dz}{(z-z_0)w(z)} &=-g_0(z_0)\omega_{\mathbf{a}},
\\
\int_{\mathbf{b}} \frac{dz}{(z-z_0)w(z)} &=\frac{2\pi i}{w(z^+_0)}F(z^-_0,
z^+_0) +\tau \int_{\mathbf{a}} \frac{dz}{(z-z_0)w(z)}.
\end{align*}
Differentiation of both sides with respect to $z_0$ at $z_0=0$ yields
$$
\int_{\mathbf{b}} \frac{dz}{z^2 w(z)} = \frac{4\pi i}{\omega_{\mathbf{a}}}
+\tau \int_{\mathbf{a}} \frac {dz}{z^2 w(z)}.
$$
Using these formulas we have
%%%%%%%%%%%%%%%%%%%%%%%%%%%%%%%%%%%%
%%%%%% Proposition 6.3 %%%%%%%%%%%%%
\begin{prop} \label{prop6.3}
For $W(z)$ as in Proposition $\ref{prop6.1}$ and for $z_{\pm}$ by \eqref{6.4},
\begin{align*}
\int_{\mathbf{a}} W(z) dz &= -(w(z_+)g_0(z_+)-w(z_-)g_0(z_-))\omega_{\mathbf{a}}
\\
&= -\frac 12 (w'(z^+_+)-w'(z^+_-)) \omega_{\mathbf{a}}
\\
& \phantom{---} + \frac{\vartheta'}{\vartheta}(F(z^-_+, z^+_+) +\nu,\tau)
- \frac{\vartheta'}{\vartheta}(F(z^-_-, z^+_-) +\nu,\tau) ,
\\
 \Bigl(\int_{\mathbf{b}}-\tau \int_{\mathbf{a}} \Bigr) & W(z) dz =
2\pi i (F(z^-_+, z^+_+)-F(z^-_-, z^+_-)),
\end{align*}
and
\begin{align*}
&\int_{\mathbf{a}} \frac{dz}{zw(z)}=-g_0(0^+)\omega_{\mathbf{a}}, \quad
g_0(0^+)=\frac 1{\omega_{\mathbf{a}}}\Bigl(\pi i + \frac{\vartheta'}{\vartheta}
(F(0^-,0^+)+\nu, \tau) \Bigr),
\\
&\Bigl(\int_{\mathbf{b}} -\tau \int_{\mathbf{a}} \Bigr)\frac{dz}{zw(z)}
=- 2\pi i F(0^-,0^+),
\\
&\Bigl(\int_{\mathbf{b}} -\tau \int_{\mathbf{a}} \Bigr)\frac{dz}{z^2w(z)}
=\frac{4\pi i}{\omega_{\mathbf{a}}} .
\end{align*}
\end{prop}
%%%%%%%%%%%%%%%%%%%%%%%%%%%%%%%%%%%%%%%%%
%%%%%%% Remark 6.2 %%%%%%%%%%
\begin{rem}\label{rem6.2}
In the proposition above the first formula is rewritten in the form
$$
\int_{\mathbf{a}} W(z)dz = 2\Bigl(\frac{\vartheta'}{\vartheta} 
( \tfrac 12 F(z^-_+, z^+_+) +\nu, \tau) -\frac{\vartheta'}{\vartheta}
( \tfrac 12 F(z^-_-, z^+_-) +\nu,\tau) \Bigr).
$$
The right-hand side is obtained by comparing the poles of
$(\vartheta'/\vartheta) (\frac 12 F(z^- , z^+) +\nu,\tau)$
with those of $-\frac 12 w'(z^+) +(\vartheta'/\vartheta)(F(z^-,
z^+) +\nu,\tau)$
on $\Pi_{a_{\phi}}$, and showing that the difference is a constant
(see also \cite[pp.~117-119]{Kitaev-3}).
\end{rem}
%%%%%%%%%%%%%%%%%%%%%%%%%%%%%%%%%%
%%%%%%% 6.3 %%%%%%%%%%%%%%
\subsection{Expression of $B_{\phi}(t)$}\label{ssc6.3}
%%%%%%%%%%%%%%%%%%%%%%%%%%%%%%%%%%%%%%%%%%%%%%%%%%
Let us write the quantity $B_{\phi}(t)$ in terms of
\begin{align*}
\Omega_{\mathbf{a}} &=\int_{\mathbf{a}} \frac{dz}{w(A_{\phi},z)}, \quad
\Omega_{\mathbf{b}} =\int_{\mathbf{b}} \frac{dz}{w(A_{\phi},z)}, 
\\
\mathcal{J}_{\mathbf{a}} &=\int_{\mathbf{a}} \frac{w(A_{\phi},z)}{z^2} dz, \quad
\mathcal{J}_{\mathbf{b}} =\int_{\mathbf{b}} \frac{w(A_{\phi},z)}{z^2} dz
\end{align*}
with $w(A_{\phi},z)=\sqrt{4z^3-A_{\phi}z^2 +1}$ and $\mathbf{a},$ $\mathbf{b}$
on $\Pi_{A_{\phi}}= \Pi_+^* \cup \Pi_-^* =\lim_{a_{\phi}(t)
 \to A_{\phi}} \Pi_{a_{\phi}}$. By \eqref{5.0} the cycles
$\mathbf{a}$ and $\mathbf{b}$ on $\Pi_{a_{\phi}}$
may be regarded as those on $\Pi_{A_{\phi}}$, and are independent of $t$ for 
sufficiently large $t$.
%%%%%%%%%%%%%%%%
\par
Let $0<\phi<\pi/3.$ 
By Proposition \ref{prop6.3}, the integral $\int_{\mathbf{a}} W(z) dz$   
is expressed in terms of
$\vartheta_*(\pm)=(\vartheta'/\vartheta)(\tfrac 12 F(z^-_{\pm}, z^+_{\pm})+\nu,
\tau)$ (Remark \ref{rem6.2}) or $w'(z_{\pm}^+)$ and 
$(\vartheta'/\vartheta)( F(z^-_{\pm}, z^+_{\pm})+\nu,\tau)$, 
in which
$$
F(z^-_{\pm}, z^+_{\pm})=\frac 1{\omega_{\mathbf{a}}} \int^{z^+_{\pm}}
_{z^-_{\pm}} \frac{dz}{w(a_{\phi},z)}
=\frac 2{\omega_{\mathbf{a}}} \int^{z^+_{\pm}}
_{\infty} \frac{dz}{w(a_{\phi},z)}.
$$
Note that $\int_{\mathbf{a}} W(z)dz$ has no poles or zeros in  
$S_{\phi}(t'_{\infty},\kappa_1,\delta_1).$ 
Indeed, if, say $\vartheta_{*}(+)$ or 
$\vartheta_{*}(-)=\infty$ at $t=t_*$, then $z_{+}$ or $z_-=\infty$, 
and hence $t_*$ is a pole or a zero of 
$y(t)$, or a pole of $y^t(t),$ which is excluded from 
$S_{\phi}(t'_{\infty}, \kappa_1,\delta_1).$  
Consider $z_{\pm}=z_{\pm}(t)$ (cf. \eqref{6.4}) moving on the
elliptic curve $\Pi_{a_{\phi}}$ crossing $\mathbf{a}$ and $\mathbf{b}$-cycles,
and then $F(z^-_{\pm}, z^+_{\pm})=2 p_{\pm}(t) + 2q_{\pm}(t)\tau +O(1)$
with $p_{\pm}(t),$ $q_{\pm}(t) \in \mathbb{Z}$. This implies the boundedness of
$\re (\vartheta'/\vartheta)(\tfrac 12 F(z^-_{\pm}, z^+_{\pm})+\nu,\tau)$ 
or $\re (\vartheta'/\vartheta)( F(z^-_{\pm}, z^+_{\pm})+\nu,\tau)$ in
$S_{\phi}(t'_{\infty},\kappa_1,\delta_1)$, 
and hence the modulus of $\re \int_{\mathbf{a}}W(z) dz$  
is uniformly bounded in $S_{\phi}(t'_{\infty},\kappa_1,\delta_1)$.   
Note that, by \eqref{5.0}, 
\begin{align*}
\frac{1}{z^2} (w(a_{\phi},z)-w(A_{\phi},z)) &= 
\frac 1{z^2} (\sqrt{4z^3-a_{\phi}z^2+ 1} - \sqrt{4z^3-A_{\phi}z^2+1})
\\
&= -\frac{t^{-1} B_{\phi}(t)} {2w(A_{\phi}, z)} (1+O(t^{-1} B_{\phi}(t) )).
\end{align*}
By using this and Proposition \ref{prop6.3}, the second formula in Proposition 
\ref{prop6.1} (1) is written in the form 
\begin{align*}
\log(g_{11} g_{22}) & = \frac{i e^{i\phi}t}{6} \int_{\mathbf{a}} \Bigl(
\frac{w(A_{\phi},z)}{z^2} -\frac{t^{-1}B_{\phi}(t)}{2w(A_{\phi},z)} \Bigr) dz
\\
& \phantom{---} -\frac 14 \int_{\mathbf{a}} W(z) dz -\frac 14 (1+2ia)g_0(0^+)
\omega_{\mathbf{a}} +\pi i +O(t^{-\delta}),
\end{align*}
which implies
\begin{align*}
i e^{i\phi}\Bigl(  t\mathcal{J}_{\mathbf{a}} &-\frac{\Omega_{\mathbf{a}}} 2
B_{\phi}(t) \Bigr)
\\
 =&\frac 32 \int_{\mathbf{a}} W(z) dz
+\frac 32(1+2ia) g_0(0^+) \omega_{\mathbf{a}} +6\log(g_{11}g_{22}) -6\pi i
+O(t^{-\delta}).
\end{align*}
Recall that $G=\hat{G}\Theta^{-\sigma_3}_{0,*} =(g_{ij})$, $g_{ij}=g_{ij}(t)$
is a solution of the direct monodromy problem. Suppose that
%%%%%% (6.6) %%%%%%%
\begin{equation}\label{6.6}
|\log(g_{11}g_{22})| \ll 1, \quad |\log(g_{12}/g_{22})| \ll 1 
\quad\text{in $S_{\phi}(t'_{\infty},\kappa_1, \delta_1).$}
\end{equation}
By the Boutroux equations \eqref{2.1}, 
$\im e^{i\phi}\Omega_{\mathbf{a}}B_{\phi}(t)$ is 
bounded as $e^{i\phi}t
\to \infty$ through $S_{\phi}(t'_{\infty},\kappa_1,\delta_1)$. 
By using the first formula of Proposition \ref{prop6.1} (1), we have
\begin{align*}
i e^{i\phi}\Bigl(  t\mathcal{J}_{\mathbf{b}} &-\frac{\Omega_{\mathbf{b}}} 2
B_{\phi}(t) \Bigr)
\\
 =&\frac 32 \int_{\mathbf{b}} W(z) dz
+\frac 32(1+2ia)(2\pi iF(0^-,0^+)+ g_0(0^+) \omega_{\mathbf{b}}) 
+6\log\frac{g_{12}}{g_{22}}+O(t^{-\delta}),
\end{align*}
in which $\int_{\mathbf{b}} W(z)dz$ admits a similar expression in terms of the
$\vartheta$-function with $\hat{\tau}=(-\omega_{\mathbf{a}})/
\omega_{\mathbf{b}}.$ This implies the boundedness of 
$\im e^{i\phi}\Omega_{\mathbf{b}}B_{\phi}(t).$ 
Then we have $|B_{\phi}(t)|\le C_0$ for some $C_0>0$ in 
$S_{\phi}(t'_{\infty},\kappa_1,\delta_1).$
The implied constant of $B_{\phi}(t) \ll 1$ in \eqref{5.0}
may be supposed to be greater than $2C_0$, which causes no changes in the 
subsequent equations by choosing $t'_{\infty}$ larger if necessary, 
and hence the boundedness of $B_{\phi}(t)$ has been shown independently of
\eqref{5.0} under \eqref{6.6}.
The case $-\pi/3 <\phi <0$ is similarly treated under the supposition
%%%%%% (6.7) %%%%
\begin{equation}\label{6.7}
|\log(g_{11}g_{22})| \ll 1, \quad |\log(g_{21}/g_{11})| \ll 1 
\quad\text{in $S_{\phi}(t'_{\infty},\kappa_1, \delta_1).$}
\end{equation}
%%%%%%%%%%%%%%%%%%%%%%%%%%%%
%%%%% Remark 6.3 %%%%%
\begin{rem}\label{rem6.3}
The argument above works also under a weaker 
condition, say $B_{\phi}(t) \ll t^{(1-\delta)/2}$. 
The supposition $B_{\phi}(t)\ll 1$ in \eqref{5.0} guarantees that each turning
point is located within the distance $O(t^{-1})$ from its limit one, which
enables us to use the limit Stokes graph in the WKB analysis.
\end{rem}
%%%%%%%%%%%%%%%%%%%%%%%%%%%%%%%%%%%%%%%%%%%%%%%%%%%%%%%%%%%%%%%
%%%% Proposition 6.7 %%%%%%%%%%%%%%%%%
\begin{prop}\label{prop6.7}
Suppose that $0<\phi<\pi/3$ and \eqref{6.6} $($respectively, $-\pi/3<\phi<0$
and \eqref{6.7}$)$. Then,
in $S_{\phi}(t'_{\infty},\kappa_1,\delta_1),$ $B_{\phi}(t)$ is bounded, and 
\begin{align*}
i e^{i\phi}\Bigl( t\mathcal{J}_{\mathbf{a}} -\frac{\Omega_{\mathbf{a}}} 2
B_{\phi}(t) \Bigr) =&
\frac 32 \int_{\mathbf{a}} W(z) dz +\frac 32(1+2ia) g_0(0^+) \omega_{\mathbf{a}} 
\\
&\phantom{--}
+6\log(g_{11}g_{22})-6\pi i+O(t^{-\delta}),
\\
=& 3\Bigl(\frac{\vartheta'}{\vartheta}(\tfrac 12 F(z^-_+, z_+^+) +\nu , \tau)
- \frac{\vartheta'}{\vartheta}(\tfrac 12 F(z^-_-, z^+_-)+\nu , \tau) \Bigr)
\\
&\phantom{--}
 +\frac 32(1+2ia) g_0(0^+) \omega_{\mathbf{a}} 
+6\log(g_{11}g_{22})-6\pi i+O(t^{-\delta}),
\\
 g_0(0^+)=&\frac 1{\omega_{\mathbf{a}}} \Bigl(\pi i +\frac{\vartheta'}
{\vartheta} (F(0^-,0^+)+\nu, \tau) \Bigr). 
\end{align*}
\end{prop}
%%%%%%%%%%%%%%%%%%%%%%%%%%%%%%%%
%%%%%%%%%%%%%%%%%%%%%%%%%%%%%%%%%%%%%%%%%%%%%%%%%%%%%%%%%%%%%%%
The following fact guarantees the possibility of limitation with respect to
$a_{\phi}$.
%%%%%%%%%%%%%%%%%%%%%%%%%%%
%%%% Proposition 6.8 %%%%%%%%%%
\begin{prop}\label{prop6.8}
Under the same supposition as in Proposition $\ref{prop6.7}$, we have
$$
\biggl(\int^{z^{+}_+}_{z^-_+} -\int^{z^{+}_-}_{z^-_-}\biggr) 
\frac{dz}{w(a_{\phi},z)}
=  
\biggl(\int^{z^{+}_+}_{z^-_+} -\int^{z^{+}_-}_{z^-_-}\biggr) 
\frac{dz}{w(A_{\phi},z)} +O(t^{-1})
$$
uniformly in $z^{\pm}_+,$ $z^{\pm}_-$ as $t e^{i\phi} \to \infty$ through 
$S_{\phi}(t'_{\infty}, \kappa_1,\delta_1).$ 
\end{prop}
%%%%%%%%%%%%%%%%%
\begin{proof}
To show this proposition we note the lemma below, which follows from the
relations
\begin{align*}
& \int \frac{w}{z^2} dz= -\frac{180}{A_{\phi}^2} \int w dz + \Bigl(\frac{108}
{A_{\phi}^2} -A_{\phi} \Bigr) \int \frac {dz}w - \frac wz - \frac 6{A_{\phi}}
w +\frac{72}{A_{\phi}^2} zw,
\\
& \Omega_{\mathbf{a}} J_{\mathbf{b}} - \Omega_{\mathbf{b}} J_{\mathbf{a}}
= -\frac{A_{\phi}^2}{15} \pi i, \quad J_{\mathbf{a},\,\, \mathbf{b}}
=\int_{\mathbf{a},\,\, \mathbf{b}} w dz
\end{align*}
with $w=w(A_{\phi},z),$ the latter equality being obtained by the same way as 
in the proof of Legendre's relation \cite{HC}, \cite{WW}.
%%%%%%%%%%%%%%%%%%%%%
%%%% Lemma 6.9 %%%%%%
\begin{lem}\label{lem6.9}
$
\Omega_{\mathbf{a}}\mathcal{J}_{\mathbf{b}} -\Omega_{\mathbf{b}}\mathcal{J}
_{\mathbf{a}}= 12\pi i.
$
\end{lem}
From the boundedness of $B_{\phi}(t)$ it follows that $\omega_{\mathbf{a},\,
\mathbf{b}}=\Omega_{\mathbf{a},\, \mathbf{b}}+O(t^{-1}).$ By Propositions
\ref{prop6.1}, \ref{prop6.3} and Remark \ref{rem6.1}, in the case $0<\phi
<\pi/3,$ we have
\begin{align*}
& \log (g_{12}/g_{22})- \tau \log(g_{11}g_{22})
\\
=& \Bigl(\int_{\mathbf{b}}-\tau\int_{\mathbf{a}}\Bigr) \Bigl(\frac{ie^{i\phi}t}
6 \cdot\frac{w(a_{\phi},z)}{z^2} -\frac 14 W(z) +\frac {1+2ia}
{4 zw(a_{\phi},z)} \Bigr) dz -\tau \pi i +O(t^{-\delta})
\\
=&-\frac{2\pi e^{i\phi}t}{\omega_{\mathbf{a}}} -\frac{\pi i}2(F(z^-_+, z^+_+)
-F(z^-_-, z^+_-) ) +O(1)
\\
=& -\frac{2\pi e^{i\phi}t}{\Omega_{\mathbf{a}}} -\pi i\Bigl(p(t) +\frac{\Omega
_{\mathbf{b}}}{\Omega_{\mathbf{a}}} q(t) \Bigr) +O(1)=\Upsilon \ll 1,
\end{align*}
with $p(t)=p_+(t)-p_-(t)$, $q(t)=q_+(t)-q_-(t) \in \mathbb{Z}$, since 
$F(z^-_{\pm}, z^+_{\pm})=2p_{\pm}(t)+2 q_{\pm}(t)\tau ,$ $p_{\pm},$ 
$q_{\pm} \in \mathbb{Z}.$
Set $e^{i\phi}\mathcal{J}_{\mathbf{a}}t/6 +\pi q(t) =X,$ $e^{i\phi}\mathcal{J}
_{\mathbf{b}} t/6-\pi p(t) =Y,$ where $X, Y \in \mathbb{R}$ by the Boutroux
equations \eqref{2.1}. Then, by Lemma
\ref{lem6.9},
\begin{align*}
\Omega_{\mathbf{a}}\Upsilon =& -2\pi e^{i\phi}t -i(e^{i\phi}t(\Omega_{\mathbf{a}}
\mathcal{J}_{\mathbf{b}}-\Omega_{\mathbf{b}}\mathcal{J}_{\mathbf{a}})/6
+\Omega_{\mathbf{b}}X-\Omega_{\mathbf{a}}Y)
\\
=&-i(\Omega_{\mathbf{b}}X -\Omega_{\mathbf{a}}Y) \ll 1
\end{align*}
with $\im ( \Omega_{\mathbf{b}}/\Omega_{\mathbf{a}}) >0,$ which implies
$|X|,$ $|Y| \ll 1,$ and hence
$$
\pi p(t)= e^{i\phi}\mathcal{J}_{\mathbf{b}}t/6 +O(1), \quad
\pi q(t)= -e^{i\phi}\mathcal{J}_{\mathbf{a}}t/6 +O(1). 
$$
Since $w(a_{\phi},z)^{-1} -w(A_{\phi},z)^{-1}
=(z^2/2)w(A_{\phi},z)^{-3}B_{\phi}(t)t^{-1} +O(t^{-2})$, we have 
\begin{align*}
\biggl| & 
\biggl(\int^{z^+_+}_{z^-_+}-\int^{z^+_-}_{z^-_-}\biggr) 
\Bigl( \frac 1{w(a_{\phi},z)} -\frac 1{w(A_{\phi},
z)} \Bigr) dz \biggr|  \ll \biggl| 
\biggl(\int^{z^+_+}_{z^-_+}-\int^{z^+_-}_{z^-_-}\biggr) 
 \frac{z^2 B_{\phi}(t)
t^{-1}}{w(A_{\phi},z)^3} dz \biggr| +O(t^{-1})
\\
& \ll \biggl | t^{-1}
\biggl(\int^{z^+_+}_{z^-_+}-\int^{z^+_-}_{z^-_-}\biggr) 
 \frac{z^2 dz}{w(A_{\phi},z)^3}
\biggr| +O(t^{-1})
\ll t^{-1} |p(t)j_{\mathbf{a}} +q(t) j_{\mathbf{b}} |+O(t^{-1})
\\
&\ll |\mathcal{J}_{\mathbf{b}} j_{\mathbf{a}}-\mathcal{J}_{\mathbf{a}}
j_{\mathbf{b}}| +O(t^{-1})
  = 2|(\partial/\partial A_{\phi}) (\mathcal{J}_{\mathbf{b}}\Omega_{\mathbf{a}}
 - \mathcal{J}_{\mathbf{a}}\Omega_{\mathbf{b}} ) |+O(t^{-1}) \ll t^{-1},
\end{align*}
where $j_{\mathbf{a},\mathbf{b}} = \int_{\mathbf{a},\mathbf{b}} z^2
w(A_{\phi},z)^{-3} dz.$ This completes the proof of the proposition.
\end{proof}
%%%%%%%%%%%%%%%%%%%%%%%%%%%%%%%%%%%%%%%%%%%%%%%%%%
%%%%%%%%%%%% Section 7 %%%%%%%%%%%%%%%%%%%%%%%
%%%%%%%%%%%%%%%%%%%%%%%%%%%%
\section{Proofs of the main theorems}\label{sc7}
%%%%%%%%%%%%%%%%%%%%%%%%%%%%%%%%%%
\subsection{Proofs of Theorems \ref{thm2.1} and \ref{thm2.2}}\label{ssc7.1}
%%%%%%%%%%%%%%%%%%%%%%%%%%%%%%%%%%%%%%%%%%%%%%%%%%%
Suppose that $0<\phi <\pi/3$. Let $G=(g_{ij}) \in SL_2(\mathbb{C})$ be 
a given matrix with 
$g_{11}g_{12}g_{22}\not=0$ in the inverse monodromy problem. Then
\begin{align*}
& \log({g_{12}}/{g_{22}}) -\tau \log (g_{11}g_{22})
\\
=& \Bigl(\int_{\mathbf{b}}-\tau \int_{\mathbf{a}} \Bigr) \Bigl( \frac{ie^{i\phi}
t}{6}\cdot \frac {w(a_{\phi}, z)}{z^2} -\frac 14W(z)
 +\frac{1+2ia}{4zw(a_{\phi},z)}
\Bigr) dz -\tau \pi i +O(t^{-\delta})
\\
=& -\frac{2\pi e^{i\phi}t}{\omega_{\mathbf{a}}} - \frac{\pi i}2(F(z^-_+, z^+_+)
-F(z^-_-, z^+_-) ) -\frac{\pi i}2(1+2ia)F(0^-,0^+) -\tau \pi i +O(t^{-\delta})
\end{align*}
(cf. Proof of Proposition \ref{prop6.8}). By Proposition \ref{prop6.8},
replacing $a_{\phi}$ with $A_{\phi}$, we have
\begin{align*}
& \log({g_{12}}/{g_{22}}) -\tau \log (g_{11}g_{22})
= -\frac{2\pi e^{i\phi}t}{\Omega_{\mathbf{a}}}
\\
& - \frac{\pi i}2(F_{A_{\phi}}
(z^-_+, z^+_+) -F_{A_{\phi}}(z^-_-, z^+_-) ) -\frac{\pi i}2(1+2ia)F_{A_{\phi}}
(0^-,0^+) -\frac{\Omega_{\mathbf{b}}}{\Omega_{\mathbf{a}}} \pi i +O(t^{-\delta})
\end{align*}
with $F_{A_{\phi}}(\tilde{z},z)=\Omega_{\mathbf{a}}^{-1} \int^z_{\tilde{z}}
dz/w(A_{\phi},z).$ 
Note that
$$
F_{A_{\phi}}(z^-_+, z^+_+) -F_{A_{\phi}}(z^-_-,z^+_-) 
=2(F_{A_{\phi}}(\infty, z^+_+)-F_{A_{\phi}}(\infty, z^+_-)), 
\quad   F_{A_{\phi}}(0^-,0^+)=2F_{A_{\phi}}(\infty,0^+),
$$
and let $\wp(u)=\wp(u; g_2,g_3)$ with $g_2=\frac 1{12}A_{\phi}^2,$
$g_3=\frac 1{216}A_{\phi}^3-1.$
Let us set
$$
u_+=\Omega_{\mathbf{a}}F_{A_{\phi}}(\infty, z^+_+), \quad
u_-=\Omega_{\mathbf{a}}F_{A_{\phi}}(\infty, z^+_-), \quad \text{i.e. $z^+_{\pm}
=\wp(u_{\pm})+\frac{A_{\phi}}{12}$}
$$
to write
%%%%%%%%% (7.1) %%%%%%%%%%%%%%%%
\begin{align}\label{7.1}
u_+-u_-=& 2 ie^{i\phi}t+\frac{i}{\pi} 
\Bigl(\Omega_{\mathbf{a}}\log \frac{g_{12}}{g_{22}}
-\Omega_{\mathbf{b}}\log(g_{11}g_{22}) \Bigr)
\\
\notag
 & \phantom{--} -\Omega_{\mathbf{b}}
-(1+2ia)\Omega_{\mathbf{a}} F_{A_{\phi}}(\infty,0^+)+ O(t^{-\delta}).
\end{align}
By the addition theorem for the $\wp$-function
\begin{align*}
\wp(u_++u_-) =& -\wp(u_+) -\wp(u_-) + \frac 14 \Bigl(\frac{\wp'(u_+)-\wp'(u_-)}
{\wp(u_+)-\wp(u_-)} \Bigr)^{\!2}
\\
=& -z^+_+ -z^+_- +\frac{A_{\phi}}6 +\frac 14 \Bigl(\frac{w(z^+_+) -w(z^+_-)}
{z^+_+ -z^+_-} \Bigr)^{\!2}.
\end{align*}
By \eqref{6.4}, $z_+=y$ and $z_-= (i/2)e^{-i\phi}y^{-1} \Gamma_0(t,y,y^t)$ 
satisfy
\begin{align*}
& z^+_+ +z_-^+ = e^{-i\phi} K_+,
\\
&  w(z^+_{\pm})= i e^{-i\phi}(b_3/\lambda)(z^+_{\pm})
=1- ie^{-i\phi} \Gamma_0(t,y,y^t) z^+_{\pm} +O(t^{-1}),
\end{align*} 
and hence
\begin{align*}
\wp(u_++u_-) =& - e^{-i\phi}K_+ + \frac{A_{\phi}}6 +\frac 14 (ie^{-i\phi}
\Gamma_0(t,y,y^t) +O(t^{-1}) )^2
\\
=& \frac{A_{\phi}}6 -\frac 14 (4e^{-i\phi}K_+ +e^{-2i\phi}\Gamma_0(t,y,y^t)^2)
+O(t^{-1})
\\
=& -\frac{A_{\phi}}{12}+O(t^{-1}),
\end{align*}
since $4e^{-i\phi}K_+ + e^{-2i\phi}\Gamma_0(t,y,y^t)^2=a_{\phi}+O(t^{-1})$.
This implies
\begin{equation}\label{7.2}
u_++u_-= \int^{0^+}_{\infty} \frac{dz}{w(A_{\phi},z)} +O(t^{-1}) =\Omega
_{\mathbf{a}} F_{A_{\phi}}(\infty, 0^+) +O(t^{-1}).
\end{equation}
From \eqref{7.1} and \eqref{7.2} with $\Omega_{\mathbf{a}}F_{A_{\phi}}(\infty,
0^+)= \Omega_0,$ it follows that
\begin{align*}
u_+=&\int^{z^+_+} _{\infty} \frac{dz}{w(A_{\phi},z)} 
\\
=& ie^{i\phi}t +\frac i{2\pi} \Bigl(\Omega_{\mathbf{a}} \log\frac{g_{12}}{g_{22}
}-\Omega_{\mathbf{b}} \log(g_{11}g_{22}) \Bigr)
-\frac{\Omega_{\mathbf{b}}}2 -ia \Omega_0+O(t^{-\delta}),
\\
u_-=&\int^{z^+_-} _{\infty} \frac{dz}{w(A_{\phi},z)} 
\\
=& - ie^{i\phi}t -\frac i{2\pi} \Bigl(\Omega_{\mathbf{a}} \log\frac{g_{12}}{g_{22}
}-\Omega_{\mathbf{b}} \log(g_{11}g_{22}) \Bigr)
+\frac{\Omega_{\mathbf{b}}}2 +(1+ia) \Omega_0+O(t^{-\delta}),
\end{align*} 
which leads us to the asymptotic expressions of Theorem \ref{thm2.1} and 
Remark \ref{rem2.1}.
\par
{\bf Justification.} The justification of $y(x)$ as a solution of \eqref{1.2}
is made along the line in \cite[pp.~105--106, pp.~120--121]{Kitaev-3}.
Let $\mathcal{G}=(g_{12}/g_{22}, g_{11}g_{22})$ be a given point such that 
$g_{11}g_{12}g_{22}\not=0$ on the monodromy manifold for \eqref{1.3}. 
In addition
to $y(x)$ obtained above, we have the following expression of $B_{\phi}(t)$
from Proposition \ref{prop6.7}.
%%%%%%%%%%%%%%%%%%%%%%%%%%%%%%%%%%
%%%%% Proposition 7.1 %%%%%%%%%%
\begin{prop}\label{prop7.1}
In $S_{\phi}(t'_{\infty},\kappa_1,\delta_1)$,
\begin{align*}
i e^{i\phi} \Bigl(t\mathcal{J}_{\mathbf{a}} -\frac{\Omega_{\mathbf{a}}}2
B_{\phi}(t) \Bigr) = & 3\Bigl(\frac{\vartheta'}{\vartheta}
(\Omega_{\mathbf{a}}^{-1} i(x-x_0^+)+\nu,\tau_{\Omega}) 
+\frac{\vartheta'}{\vartheta}
(\Omega_{\mathbf{a}}^{-1}( i(x-x_0^+) -\Omega_0)+\nu,\tau_{\Omega}) \Bigr)
\\
& + \tfrac 32 (1+2ia)g_0(0^+)\Omega_{\mathbf{a}} +6\log(g_{11}g_{22})-6\pi i
+O(t^{-\delta})
\end{align*}
with $x=e^{i\phi} t,$ $\tau_{\Omega}=\Omega_{\mathbf{b}}/\Omega_{\mathbf{a}}.$
\end{prop}
%%%%%%%%%%%%%%%%%%%
The equation about $u_+$ and the proposition above provide the leading term
expressions 
$$
y_{\mathrm{as}}=y_{\mathrm{as}}(\mathcal{G},t)
=\wp(i(e^{i\phi}t-x_0^+); g_2(A_{\phi}), g_3(A_{\phi}) )+\frac{A_{\phi}}{12}
$$
and $(B_{\phi})_{\mathrm{as}}=(B_{\phi})_{\mathrm{as}}(\mathcal{G},t)$ without
$O(t^{-\delta})$, where $x_0^+$ depends on $(g_{12}/g_{22}, g_{11}g_{22}).$
Taking \eqref{4.2} and \eqref{5.0} into account, we set
$$
y^t_{\mathrm{as}}=-\frac{y_{\mathrm{as}}}2 t^{-1}
+ ie^{i\phi} \sqrt{4y_{\mathrm{as}}^3-A_{\phi}y_{\mathrm{as}}^2+1
-(3ie^{-i\phi}(1+2ia) +(B_{\phi})_{\mathrm{as}}y_{\mathrm{as}})y_{\mathrm{as}}
t^{-1} },
$$
where the branch of the square root is chosen in such a way that $y^t
_{\mathrm{as}}$ is compatible with $(\partial/\partial t)y_{\mathrm{as}}$.
Then $(y_{\mathrm{as}},y^t_{\mathrm{as}})=(y_{\mathrm{as}}(\mathcal{G},t),
y^t_{\mathrm{as}}(\mathcal{G},t))$ fulfils \eqref{5.0} with $B_{\phi}(t)
=(B_{\phi})_{\mathrm{as}}(\mathcal{G},t)$ in the domain
$\hat{S}(\phi,t_{\infty},\kappa_0,
\delta_2)=\{t\,|\,\, \re t>t_{\infty},\, |\im t|<\kappa_0 \}
\setminus \bigcup_{i \sigma\in Z_0} \{|t-e^{-i\phi} \sigma|<\delta_2 \}$
with $Z_0=\{ ix^+_0 + \Omega_{\mathbf{a}}\mathbb{Z} +
\Omega_{\mathbf{b}} \mathbb{Z} \} \cup
\{ ix^+_0 + \Omega_0+ \Omega_{\mathbf{a}}\mathbb{Z} +
\Omega_{\mathbf{b}} \mathbb{Z} \} \cup
\{ ix^+_0 + \xi_0 \,|\, \wp(\xi_0)= -A_{\phi} /12 \}.$
Let $\mathcal{G}_{\mathrm{as}}(t)$ be the monodromy data for system \eqref{1.3}
containing $(y_{\mathrm{as}},y^t_{\mathrm{as}})$. As a result of the WKB 
analysis for the direct monodromy problem we have
$\| \mathcal{G}_{\mathrm{as}}(t)-\mathcal{G} \| \ll t^{-\delta}$, which holds
uniformly in a neighbourhood of $\mathcal{G}$. Then the justification scheme
of Kitaev \cite{Kitaev-1} applies to our case. 
Using the maximal modulus principle
in each neighbourhood of $i \sigma=ix^+_0 +\{ \Omega_0 
+\Omega_{\mathbf{a}}\mathbb{Z}+\Omega_{\mathbf{b}}\mathbb{Z} \}\cup
\{ \xi_0 \,|\, \wp(\xi_0)=-A_{\phi}/12 \},$ we obtain Theorem \ref{thm2.1}.
Theorem \ref{thm2.2} is proved by the same argument as above.
%%%%%%%%%%%%%%%%%%%%%%%%%%%%%%%%
%%%% 7.2 %%%%%%%%%%%%%
\subsection{Proof of Theorem \ref{thm2.3}}\label{ssc7.3}
Let \eqref{1.3} with $y^t=(d/dt)y$ be an isomonodromy system.
Equation \eqref{1.2},
system \eqref{1.3} and the function $a_{\phi}$ with $y^t=(d/dt)y$ 
remain invariant under the substitution
$$
\phi=\tilde{\phi}+2m\pi/3, \quad y=e^{2m\pi i/3}\tilde{y}, \quad x=e^{2m\pi i/3}
\tilde{x}, \quad \lambda=e^{2m\pi i/3}\tilde{\lambda}, \quad a_{\phi}=e^{2m\pi i/3}
{a}_{\tilde{\phi}}.
$$
To show the theorem we use this symmetry (cf. \cite{Kitaev-2}). Let $\phi$
be such that $0<|\phi-2m\pi/3 |<\pi/3.$ Then a new system with respect to
$(\tilde{\lambda}, \tilde{y}, \tilde{x}, \tilde{\phi})$ is an isomonodromy 
system for $0<|\tilde{\phi}|<\pi/3.$ Denote by $G^{(m)}$ a connection 
matrix as the matrix monodromy data for the system governed by 
$\tilde{y}(\tilde{x})=e^{-2m\pi i/3} y(x)= e^{-2m\pi i/3} y(e^{2m\pi i/3}
\tilde{x})$. We would like to know 
the relation between $G^{(m)} $ and $G$. The matrix solutions of the new system
are
$$
\tilde{Y}^{\infty}_j(\tilde{\lambda}) \sim \tilde{\lambda}^{-(1/2 +ia)\sigma_3}
\exp(-(i/6) e^{i\tilde{\phi}} t\tilde{\lambda}^2 \sigma_3)
$$
as $\tilde{\lambda} \to \infty$ through the sector $|\arg \tilde{\lambda}
 +\tilde{\phi} /2 -j\pi/2 | <\pi/2,$ and
$$
\tilde{Y}^0_j(\tilde{\lambda}) \sim (i/\sqrt{2}) (\sigma_1+\sigma_3) \exp(-(2i/3)
e^{i\tilde{\phi}} t\tilde{\lambda}^{-1} \sigma_3)
$$
as $\tilde{\lambda} \to 0$ through the sector $|\arg \tilde{\lambda}
 -\tilde{\phi} -j\pi | <\pi.$ 
The connection matrix $G^{(m)}$ is defined by
$\tilde{Y}^{\infty,*}_0(\tilde{\lambda})
=\tilde{Y}^{\infty}_0(\tilde{\lambda})\Theta_{0,*}^{-\sigma_3}
=\tilde{Y}^0_0(\tilde{\lambda}) G^{(m)}$. 
Note that $\tilde{Y}^{\infty}_0(\tilde{\lambda})$
and $\tilde{Y}^0_0(\tilde{\lambda})$ are also expressed as
$$
\tilde{Y}^{\infty}_0(\tilde{\lambda})=\tilde{Y}^{\infty}_0(e^{-2m\pi i/3}{\lambda})
 \sim {\lambda}^{-(1/2 +ia)\sigma_3}
\exp(-(i/6) e^{i{\phi}} t{\lambda}^2 \sigma_3) e^{(2m\pi i/3)(1/2+ia) \sigma_3}
$$
in the sector $|\arg {\lambda} +{\phi}/2 -m\pi | <\pi/2,$ and that
$$
\tilde{Y}^0_0(\tilde{\lambda}) =\tilde{Y}^0_0(e^{-2m\pi i/3}{\lambda})
 \sim (i/\sqrt{2}) (\sigma_1+\sigma_3) \exp(-(2i/3)
e^{i{\phi}} t{\lambda}^{-1} \sigma_3)
$$
in the sector $|\arg {\lambda} -{\phi}| <\pi.$ Then we have $\tilde{Y}^0_0
(\tilde{\lambda})=\hat{Y}^0_0(\lambda)$ and, if $m\ge 1$,
$$
\tilde{Y}^{\infty}_0(\tilde{\lambda})=\hat{Y}^{\infty}_{2m}(\lambda) e^{(2m\pi i/3)
(1/2+ia)\sigma_3} =\hat{Y}^{\infty}_0(\lambda) \hat{S}^{\infty}_0 \hat{S}^{\infty}_1 \cdots
\hat{S}^{\infty}_{2m-2} \hat{S}^{\infty}_{2m-1} e^{(2m\pi i/3)(1/2+ia)\sigma_3}, 
$$ 
which implies, by Remark \ref{rem3.2},
$$
G^{(m)}=G S^{\infty}_0 S^{\infty}_1 \cdots
S^{\infty}_{2m-2} S^{\infty}_{2m-1} e^{(2m\pi i/3)(1/2+ia)\sigma_3}. 
$$
This combined with Proposition \ref{prop3.2} yields the expression of $G^{(m)}$
for $m\ge 1$ as in the theorem. Note that
$P_*(u,A)=\wp (u;g_2(A),g_3(A))+\frac 1{12}A$ solves 
$(P_u)^2=4P^3-AP^2 +1.$ Then
$$
P_*(u,A)=e^{-2\pi i/3} P_*(e^{2\pi i/3}u, e^{2\pi i/3}A) 
=e^{2\pi i/3} P_*(e^{-2\pi i/3}u, e^{-2\pi i/3}A).  
$$
By Theorems \ref{thm2.1} and \ref{thm2.2}, $\tilde{y}(\tilde{x})=e^{-2m\pi i/3}
y(x)$ for $0<|\phi-2m\pi /3|=|\tilde{\phi}| <\pi/3$ is represented by
$$
e^{-2m\pi i/3}y(x)=\tilde{y}(\tilde{x})
= P_*(i(\tilde{x}- x_0(G^{(m)}, \Omega_{\mathbf{a}}^{\tilde{\phi}}, 
\Omega_{\mathbf{b}}^{\tilde{\phi}}, \Omega_0^{\tilde{\phi}})) +O(x^{-\delta}); 
A_{\tilde{\phi}}).
$$
Using the relation above, we have
\begin{align*}
y(x)&= e^{2m\pi i/3} P_*(i(\tilde{x}-x_0(G^{(m)}, \Omega_{\mathbf{a}}^{\tilde
{\phi}}, \Omega_{\mathbf{b}}^{\tilde{\phi}},
\Omega_0^{\tilde{\phi}} )) +O(x^{-\delta}); A_{\tilde{\phi}})
\\
&= P_*(i({x}- e^{2m\pi i/3}x_0(G^{(m)},\Omega_{\mathbf{a}}^{\tilde{\phi}}, 
\Omega_{\mathbf{b}}^{\tilde{\phi}},
\Omega_0^{\tilde{\phi}} )) +O(x^{-\delta}); A_{{\phi}})
\\
&= P_*(i({x}- x_0(G^{(m)},\Omega^{\phi}_{\mathbf{a}}, \Omega^{\phi}_{\mathbf{b}},
\Omega^{\phi}_0 )) +O(x^{-\delta}); A_{{\phi}}),
\end{align*}
which is denoted by $P(A_{{\phi}},
x_0(G^{(m)},\Omega^{\phi}_{\mathbf{a}}, \Omega^{\phi}_{\mathbf{b}},
\Omega^{\phi}_0) ; x) $ as in the theorem. 

%%%%%%%%%%%%%%%%%%%%%%%%%%%%%%%%%%%%%%%%%%%
%%%%%%%%%%%%%%%%%%%%%%%%%%%%%%%%%%%%%%%%%5
%%%%%% Section 8 %%%%%%%%%%%%%%%%%%%%%%
%%%%%%%%%%%%%%%%%%%%%%%%%
\section{Modulus $A_{\phi}$ and the Boutroux equations}\label{sc8}
%%%%%%%%%%%%%%%%%%%%%%%%%%%%%%%%
Recall the elliptic curve $\Pi_A$ for $w(A,z)^2=4z^3-Az^2+1$ defined in
Section \ref{ssc2.2}. For a given $\phi \in \mathbb{R}$ we would like to
examine the modulus $A_{\phi}\in \mathbb{C}$ such that, for every cycle
$\mathbf{c} \subset \Pi_{A_{\phi}}$,
$$
\im e^{i\phi} \int_{\mathbf{c}} \frac{w(A_{\phi},z)}{z^2} dz=0.
$$
First, for $|\phi|\le \pi/3$, let us consider $A_{\phi}$ satisfying the
Boutroux equations
$$
(\mathrm{BE})_{\phi}: \quad\qquad \im e^{i\phi} I_{\mathbf{a}}(A_{\phi})=0, \quad   
\im e^{i\phi} I_{\mathbf{b}}(A_{\phi})=0, \phantom{------------}
$$
where $\mathbf{a},$ $\mathbf{b}$ denote the basic cycles given in Section
\ref{ssc2.2} and 
$$
I_{\mathbf{a}, \, \mathbf{b}} (A)=\int_{\mathbf{a},\,\mathbf{b}} \frac{w(A,z)}
{z^2} dz   
=\int_{\mathbf{a},\,\mathbf{b}} \frac{1}{z^2} \sqrt{4z^3-Az^2+1}\, dz.
$$
It is easy to see that $w(A,z)^2=4z^3-Az^2+1$ has double roots $z_0,$ $z_1$ if
and only if 
\begin{align*}
A&=3\cdot 2^{2/3}, \quad z_0,z_1=2^{-1/3}, \,\,\, z_2=-4^{-2/3};
\\
A&=3\cdot 2^{2/3}e^{\pm 2\pi i/3}, \quad z_0,z_1=2^{-1/3}e^{\pm 2\pi i/3}, 
\,\,\, z_2=-4^{-2/3} e^{\pm 2\pi i/3}.
\end{align*}
%%%%%% Example 8.1 %%%%%%%
\begin{exa}\label{exa8.1}
When $\phi=0$, we have $I_{\mathbf{a}}(3\cdot 2^{2/3})=0,$ 
$I_{\mathbf{b}}(3\cdot 2^{2/3})= -2^{4/3}3^{3/2}.$ Indeed
$$
I_{\mathbf{b}}(3\cdot 2^{2/3})= 2\int^{2^{-1/3}}_{-4^{-2/3}}\frac 2{z^2}
\textstyle{(i\sqrt{2^{-1/3}-z})^2\sqrt{z+4^{-2/3}}\,  dz}
\displaystyle=-4^{2/3} \int^2_{-1}\frac{(2-t)
\sqrt{t+1}} {t^2}dt, 
$$
in which the residue of the integrand at $z=0$ vanishes.
\end{exa}
%%%%%%%%
Note that $\mathbf{a}$ is a cycle enclosing the cut $[z_0,z_1]$.
In accordance with \cite[Section 7]{Kitaev-2} we begin with the following: 
%%%%%%%%%%%%%%%%%%%%%%%%%%%%%%%%%%%%%%%
%%%% Proposition 8.1 %%%%%%
\begin{prop}\label{prop8.1}
Suppose that $\im I_{\mathbf{a}}(A)=0.$ Then $A\in \mathbb{R}.$
\end{prop}
%%%%%%%%%%%%%%%%%%%%%%%%%%%
\begin{proof}
Set
$$
J_{\mathbf{a}}(A)=\int_{-\mathbf{a}}\frac 1{z^2} v(A,z) dz
$$
with $v(A,z)=\sqrt{4z^3+Az^2 -1}=-i w(A,-z)$. Since $I_{\mathbf{a}}(A)= -i
J_{\mathbf{a}}(A),$ the supposition means $J_{\mathbf{a}}(A) \in i \mathbb{R}.$
In this proof, to simplify the description, we write $v(A,z)=v_A(z)$, 
$v(\overline{A},z)=v_{\overline{A}}(z)$ and $v(A,z)\pm v(\overline{A},z)
=(v_A\pm v_{\overline{A}})(z)$. 
Then
$$
0=J_{\mathbf{a}}(A)+\overline{J_{\mathbf{a}}(A)}
=J_{\mathbf{a}}(A)+ J_{\overline{\mathbf{a}}}(\overline{A})
=J_{\mathbf{a}}(A)- J_{\mathbf{a}}(\overline{A})
=(A-\overline{A}) \int_{-\mathbf{a}}\frac{dz}{(v_A+v_{\overline{A}})(z)}.
$$
The polynomials $v_A(z)^2$ and $v_{\overline{A}}(z)^2$ have the roots $-z_0,$
$-z_1,$ $-z_2$, and $-\overline{z_0},$ $-\overline{z_1},$ $-\overline{z_2}$, 
respectively. The algebraic functions $(v_A \pm v_{\overline{A}})(z)$ may 
be considered on the two-sheeted Riemann surface glued along the cuts 
$[-z_0, {-z_1}]$,
$[-\overline{z_0}, -\overline{z_1}]$, $[-z_2, -\overline{z_2}]\cup 
[- \infty, -\re z_2]$ (cf. Figure \ref{cycle3}).
The cycle $-\mathbf{a}$ may be supposed to enclose both cuts $[-z_0,-z_1]$,
$[-\overline{z_0},-\overline{z_1}],$ and the cycles as in Figures \ref{cycle3}
(a.1) and (a.2) may be deformed into contours consisting of horizontal and
vertical lines and enclosing the cuts $[-z_2,-\overline{z_2}]\cup 
[-\infty, -\re z_2]$ clockwise as in Figures (a$^*$.1) and (a$^*$.2), 
respectively.  
Possible extension of this contour is caused by further movement of $-z_0,$ 
$-z_1$ and $-z_2$, and is given by adding horizontal and vertical lines located
in the symmetric position with respect to the real axis. 
To show $A\in \mathbb{R}$ it is sufficient to verify that, under the supposition
$A-\overline{A}\not=0,$
$$
J=\int_{-\mathbf{a}} \frac{dz}{(v_A + v_{\overline{A}})(z)} \not=0.
$$
Let us compute this integral along the contour $-\mathbf{a}$, say as in Figure
\ref{cycle3} (a$^*$.2) with vertices $\alpha \pm i\beta,$ $\gamma \pm i\beta$ 
such that $-z_2,-\overline{z_2}=\gamma \pm i\beta,$ in which 
$\alpha \le \gamma,$ $\beta\ge 0,$
and $\alpha$ may be supposed to be $\alpha < 0.$
%%%%%%%%%%%%%%%%%%%%%%%%%%%%%%%%%%%%%%
%%%%%%%%%%% Figure 6  %%%%%%%%%%%%%%%%%%%%%%%%%%%%
%%%%%%%%%%%%%%%%%%%%%%%%%%%%%%%%%%%%%%%%%%%%%%%%%
%%%%%%%%%%%%%%%%%%%%%%%%%%%%
{\small
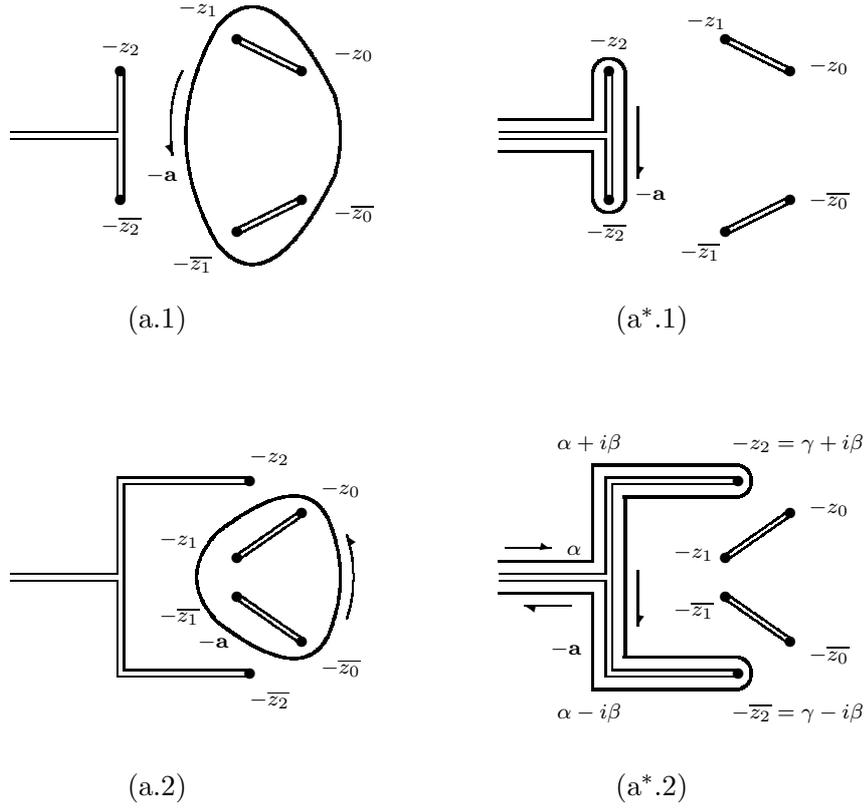
\begin{figure}[htb]
\begin{center}
\unitlength=0.85mm
%%%%%%%%%%%%%%%%%%%%%%%%%%%%%%%%%%
%%%%%%%%%%%%%%%%%%%%%%%%%%%%%%%
\begin{picture}(60,60)(-30,-36)
\put(15,10){\circle*{1.5}}
\put(15,-10){\circle*{1.5}}
\put(5,15){\circle*{1.5}}
\put(5,-15){\circle*{1.5}}
\put(-13,10){\circle*{1.5}}
\put(-13,-10){\circle*{1.5}}
\put(-5,-3){\vector(1,-3){0}}

{\tiny
\put(-4,19){\makebox{$-z_{1}$}}
\put(20,12){\makebox{$-z_{0}$}}
\put(-16,13){\makebox{$-z_{2}$}}
\put(-5,-21){\makebox{$-\overline{z_{1}}$}}
\put(20,-13){\makebox{$-\overline{z_{0}}$}}
\put(-16,-15){\makebox{$-\overline{z_{2}}$}}
\put(-9,-7){\makebox{$-\mathbf{a}$}}
}

\qbezier (-12.5,10) (-12.5,0) (-12.5,-10)
\qbezier (-13.5,10) (-13.5,5.5) (-13.5,0.5)
\qbezier (-13.5,-10) (-13.5,-5.5) (-13.5,-0.5)
\qbezier (-13.5,0.5) (-20,0.5) (-30,0.5)
\qbezier (-13.5,-0.5) (-20,-0.5) (-30,-0.5)

\qbezier (15,10.5) (10,13) (5,15.5)
\qbezier (15,9.5) (10,12) (5,14.5)
\qbezier (15,-10.5) (10,-13) (5,-15.5)
\qbezier (15,-9.5) (10,-12) (5,-14.5)

\qbezier (-5,-3) (-6,5) (-3.5,10)

\thicklines
\qbezier (20,6.5) (10.0,26.6) (2,17)
\qbezier (20,-6.5) (10.0,-26.6) (2,-17)

\qbezier (20,-6.5) (22,0) (20,6.5)

\qbezier (2,17) (-7.8,0) (2, -17)

\put(-12,-30){\makebox{(a.1) }}
\end{picture}
%%%%%%%%%%%%%%%%%%%%%%%%%%%%%%
\qquad
%%%%%%%%%%%%%%%%%%%%%%%%%%%%%%%
\begin{picture}(70,60)(-35,-36)

\put(15,10){\circle*{1.5}}
\put(15,-10){\circle*{1.5}}
\put(5,15){\circle*{1.5}}
\put(5,-15){\circle*{1.5}}
\put(-13,10){\circle*{1.5}}
\put(-13,-10){\circle*{1.5}}
\put(-8.5,-6.5){\vector(-1,-4){0}}

{\tiny
\put(-1,17){\makebox{$-z_{1}$}}
\put(18,10){\makebox{$-z_{0}$}}
\put(-16,14){\makebox{$-z_{2}$}}
\put(-2,-19){\makebox{$-\overline{z_{1}}$}}
\put(18,-11){\makebox{$-\overline{z_{0}}$}}
\put(-16.5,-16){\makebox{$-\overline{z_{2}}$}}
\put(-9,-10){\makebox{$-\mathbf{a}$}}
}

\qbezier (-12.5,10) (-12.5,0) (-12.5,-10)
\qbezier (-13.5,10) (-13.5,5.5) (-13.5,0.5)
\qbezier (-13.5,-10) (-13.5,-5.5) (-13.5,-0.5)
\qbezier (-13.5,0.5) (-20,0.5) (-30,0.5)
\qbezier (-13.5,-0.5) (-20,-0.5) (-30,-0.5)

\qbezier (15,10.5) (10,13) (5,15.5)
\qbezier (15,9.5) (10,12) (5,14.5)
\qbezier (15,-10.5) (10,-13) (5,-15.5)
\qbezier (15,-9.5) (10,-12) (5,-14.5)

\qbezier (-8.5,-6.5) (-8.5,-1) (-8.5,4.5)

\thicklines

\qbezier (-10.5, 10) (-10.5,0) (-10.5,-10)
\qbezier (-15.5,10) (-15.5,5.5) (-15.5,2.5)
\qbezier (-15.5,-10) (-15.5,-5.5) (-15.5,-2.5)
\qbezier (-15.5,2.5) (-20,2.5) (-30,2.5)
\qbezier (-15.5,-2.5) (-20,-2.5) (-30,-2.5)

\qbezier (-10.5, 10) (-11,12) (-13,12)
\qbezier (-15.5, 10) (-15,12) (-13,12)

\qbezier (-10.5, -10) (-11,-12) (-13,-12)
\qbezier (-15.5, -10) (-15,-12) (-13,-12)

\put(-12,-30){\makebox{$(\mathrm{a}^*.1)$ }}
\end{picture}

%%%%%%%%%%%%%%%%%%%%%%%%%%%%%%%%%%%%%%%%%
\end{center}
\begin{center}
\unitlength=0.85mm
%%%%%%%%%%%%%%%%%%%%%%%%%%%%%%%%%%
%%%%%%%%%%%%%%%%%%%%%%%%%%%%%%%
\begin{picture}(60,60)(-30,-30)
\put(15,10){\circle*{1.5}}
\put(15,-10){\circle*{1.5}}
\put(5,3){\circle*{1.5}}
\put(5,-3){\circle*{1.5}}
\put(7,15){\circle*{1.5}}
\put(7,-15){\circle*{1.5}}
\put(22,6.5){\vector(-2,3){0}}

{\tiny
 \put(-7,5){\makebox{$-z_{1}$}}
\put(18,13){\makebox{$-z_{0}$}}
\put(7,18){\makebox{$-z_{2}$}}
 \put(-7,-7){\makebox{$-\overline{z_{1}}$}}
\put(18,-15){\makebox{$-\overline{z_{0}}$}}
\put(7,-20){\makebox{$-\overline{z_{2}}$}}
\put(-1,-11){\makebox{$-\mathbf{a}$}}
}

\qbezier (7,15.5) (-12.5,15.5) (-13.5,15.5)
\qbezier (7,14.5) (-12.5,14.5) (-12.5,14.5)

\qbezier (7,-15.5) (-12.5,-15.5) (-13.5,-15.5)
\qbezier (7,-14.5) (-12.5,-14.5) (-12.5,-14.5)

\qbezier (-12.5,14.5) (-12.5,0) (-12.5,-14.5)
\qbezier (-13.5,15.5) (-13.5,5.5) (-13.5,0.5)
\qbezier (-13.5,-15.5) (-13.5,-5.5) (-13.5,-0.5)
\qbezier (-13.5,0.5) (-20,0.5) (-30,0.5)
\qbezier (-13.5,-0.5) (-20,-0.5) (-30,-0.5)

\qbezier (15,10.5) (10,7) (5,3.5)
\qbezier (15,9.5) (10,6) (5,2.5)
\qbezier (15,-10.5) (10,-7) (5,-3.5)
\qbezier (15,-9.5) (10,-6) (5,-2.5)

\qbezier (22,-6.5) (24,0) (22,6.5)

\thicklines
\qbezier (20,6.5) (17.0,18.6) (2,7)
\qbezier (20,-6.5) (17.0,-18.6) (2,-7)

\qbezier (20,-6.5) (22,0) (20,6.5)

\qbezier (2,7) (-4.5,0) (2, -7)

\put(-12,-34){\makebox{(a.2) }}
\end{picture}
%%%%%%%%%%%%%%%%%%%%%%%%%%%%%%
\qquad
%%%%%%%%%%%%%%%%%%%%%%%%%%%%%%%
\begin{picture}(70,60)(-35,-30)

\put(15,10){\circle*{1.5}}
\put(15,-10){\circle*{1.5}}
\put(5,3){\circle*{1.5}}
\put(5,-3){\circle*{1.5}}
\put(7,15){\circle*{1.5}}
\put(7,-15){\circle*{1.5}}
\put(-8.5,-7.5){\vector(-1,-4){0}}
\put(-26,-4.5){\vector(-4,1){0}}
\put(-22,4.5){\vector(4,-1){0}}

{\tiny
\put(-3,3){\makebox{$-z_{1}$}}
\put(18,10){\makebox{$-z_{0}$}}
 \put(6,20){\makebox{$-z_{2}=\gamma+i\beta$}}
 \put(-21,20){\makebox{$\alpha+i\beta$}}
\put(-3,-6){\makebox{$-\overline{z_{1}}$}}
\put(18,-13){\makebox{$-\overline{z_{0}}$}}
 \put(6,-22){\makebox{$-\overline{z_{2}}=\gamma-i\beta$}}
 \put(-21,-22){\makebox{$\alpha-i\beta$}}
 \put(-19.5,3.5){\makebox{$\alpha$}}
\put(-22,-12){\makebox{$-\mathbf{a}$}}
}

\qbezier (7,15.5) (-12.5,15.5) (-13.5,15.5)
\qbezier (7,14.5) (-12.5,14.5) (-12.5,14.5)

\qbezier (7,-15.5) (-12.5,-15.5) (-13.5,-15.5)
\qbezier (7,-14.5) (-12.5,-14.5) (-12.5,-14.5)

\qbezier (-12.5,14.5) (-12.5,0) (-12.5,-14.5)
\qbezier (-13.5,15.5) (-13.5,5.5) (-13.5,0.5)
\qbezier (-13.5,-15.5) (-13.5,-5.5) (-13.5,-0.5)
\qbezier (-13.5,0.5) (-20,0.5) (-30,0.5)
\qbezier (-13.5,-0.5) (-20,-0.5) (-30,-0.5)

\qbezier (15,10.5) (10,7) (5,3.5)
\qbezier (15,9.5) (10,6) (5,2.5)
\qbezier (15,-10.5) (10,-7) (5,-3.5)
\qbezier (15,-9.5) (10,-6) (5,-2.5)

\qbezier (-8.5,-7.5) (-8.5,-5) (-8.5,1)

\qbezier (-19,-4.5) (-20,-4.5) (-26,-4.5)
\qbezier (-22,4.5) (-22,4.5) (-29,4.5)

\thicklines

\qbezier (-10.5,12.5)(-10.5,0) (-10.5,-12.5)
\qbezier (-15.5,17.5) (-15.5,5.5) (-15.5,2.5)
\qbezier (-15.5,-17.5) (-15.5,-5.5) (-15.5,-2.5)
\qbezier (-15.5,2.5) (-20,2.5) (-30,2.5)
\qbezier (-15.5,-2.5) (-20,-2.5) (-30,-2.5)

\qbezier (7,17.5) (-12.5,17.5) (-15.5,17.5)
\qbezier (7,12.5) (-12.5,12.5) (-10.5,12.5)

\qbezier (7,-17.5) (-12.5,-17.5) (-15.5,-17.5)
\qbezier (7,-12.5) (-12.5,-12.5) (-10.5,-12.5)

\qbezier (7, 17.5) (9,17) (9,15)
\qbezier (9, 15) (9,13) (7,12.5)

\qbezier (7, -17.5) (9,-17) (9,-15)
\qbezier (9, -15) (9,-13) (7,-12.5)

\put(-12,-34){\makebox{$(\mathrm{a}^*.2)$ }}
\end{picture}

%%%%%%%%%%%%%%%%%%%%%%%%%%%%%%%%%%%%%%%%%
\end{center}

\caption{Modification of the cycle $-\mathbf{a}$}
\label{cycle3}
\end{figure}
}
%%%%%%%%%%%%%%%%%%
%%%%%%%%%%%%%%%%%%%%%%%%%
%%%%%%%%%%%%%%%%%%%%%%%%%%%%%%%
The integral $J$ is decomposed into three parts: $J=J_0 +J_{\mathrm{hor}}
+J_{\mathrm{ver}}$ with the real line part
$$
J_0= 2\int^{\alpha}_{-\infty} \frac{dz}{(v_A+v_{\overline{A}})(z)},
$$
the horizontal part
\begin{align*}
J_{\mathrm{hor}}&=J_{\mathrm{hor}}^+ + J_{\mathrm{hor}}^-,
\\
J_{\mathrm{hor}}^+ &= \int_{\alpha}^{\gamma}\frac{ds}{(v_A+v_{\overline{A}})
(s+i\beta)}
+ \int^{\alpha}_{\gamma}\frac{ds}{(-v_A+v_{\overline{A}})(s+i\beta)},
\\
J_{\mathrm{hor}}^- &= \int_{\alpha}^{\gamma}\frac{ds}{(-v_A+v_{\overline{A}})
(s-i\beta)}
+ \int^{\alpha}_{\gamma}\frac{ds}{(-v_A-v_{\overline{A}})(s-i\beta)},
\end{align*}
and the vertical part
\begin{align*}
J_{\mathrm{ver}}&=J_{\mathrm{ver}}^+ + J_{\mathrm{ver}}^-,
\\
J_{\mathrm{ver}}^+ &= \int^{\beta}_{0}\frac{i dt}{(v_A+v_{\overline{A}})
(\alpha+it)}
+ \int^{0}_{\beta}\frac{i dt}{(-v_A+v_{\overline{A}})(\alpha+it)},
\\
J_{\mathrm{ver}}^- &= \int_{0}^{-\beta}\frac{idt}{(-v_A+v_{\overline{A}})
(\alpha+it)}
+ \int^{0}_{-\beta}\frac{idt}{(-v_A-v_{\overline{A}})(\alpha+it)}.
\end{align*}
Then we have
$$
J_{\mathrm{hor}}=\frac 2{A-\overline{A}} \Bigl( \int^{\gamma}_{\alpha}
\frac{v_{{A}}(s+i\beta)}{(s+i\beta)^2}ds
- \int^{\gamma}_{\alpha}
\frac{v_{\overline{A}}(s-i\beta)}{(s-i\beta)^2}ds \Bigr) \in \mathbb{R}
$$
and
$$
J_{\mathrm{ver}}=\frac{2i}{A-\overline{A}} \Bigl(\int^{\beta}_0
\frac{v_A(\alpha+it)}{(\alpha+it)^2}dt
+ \int^{\beta}_0 \frac{v_{\overline{A}}(\alpha-it)}{(\alpha-it)^2} dt \Bigr)
\in \mathbb{R},
$$
and hence $J_{\mathrm{hor}}+J_{\mathrm{ver}} \in  \mathbb{R}.$
Furthermore, observing, for $-t=x-\alpha,$ $t\ge 0,$ $\alpha <0,$
\begin{align*}
v_{A}(x)&=(-4(t-\alpha)^3 +(\re A+i\im A)(t-\alpha)^2 -1)^{1/2}=i(g(t)-i h(t))
^{1/2},
\\
g(t)&= 4(t-\alpha)^3 -\re A \cdot (t-\alpha)^2+1, 
\quad h(t)=\im A\cdot (t-\alpha)^2,
\end{align*} 
we have
\begin{equation*}
\frac 12 J_0 = \int^{\alpha}_{-\infty}\frac{dx}{(v_A +v_{\overline
{A}})(x)}
= - \frac{i}{\sqrt{2}} \int^{\infty}_{0}\frac{dt}
{\sqrt{\smash[b]{g(t)+
\sqrt{\smash[b]{g(t)^2+h(t)^2}}} }} \in i \mathbb{R}\setminus \{0\},
\end{equation*}
which implies $J\not=0$ under the supposition $A-\overline{A}\not=0.$
In the case where extension by horizontal or vertical lines occurs,
the contributions from these parts to $J$ are integrals analogous
to $J_{\mathrm{hor}}$ or $J_{\mathrm{ver}}$, and $J\not=0$ are similarly
shown.
\end{proof}
%%%%%%%%%%%%%%%%%%%%%
Let us examine $I_{\mathbf{a}}(A)$ for $A\in \mathbb{R}$. It is easy to see that,
$w(A,z)^2$ has real roots $z_2<z_1<z_0$ if $A> 3\cdot 2^{2/3}.$ Then $I_{\mathbf
{a}}(A)\in i \mathbb{R}\setminus \{0\}.$ For $A=3\cdot 2^{2/3}$ we have 
$z_2<z_1 =z_0=2^{-1/3},$ and then $I_{\mathbf{a}}(3\cdot 2^{2/3})=0.$ 
\par
Suppose that $A<3\cdot 2^{2/3}.$ The roots of $w(A,z)^2$ are $\alpha\pm i\beta$
and $z_2$ with $\alpha,$ $\beta,$ $z_2 \in \mathbb{R},$ and $\mathbf{a}$ is
a cycle enclosing the cut $[\alpha-i\beta, \alpha+i\beta]$. Then $I_{\mathbf{a}}
(A) \in i\mathbb{R},$ since $\overline{I_{\mathbf{a}}(A)} =-I_{\mathbf{a}}(A),$
and the integral
\begin{equation*}
I_{\mathbf{a}}(A) = 2i \int^{\beta}_{-\beta} \frac{w(A, \alpha+it)}{(\alpha+
it)^2} dt =
 4i \int^{\beta}_0 \re \frac{w(A,\alpha+it)}{(\alpha+it)^2} dt
\end{equation*}
satisfies, for $A<3 \cdot 2^{2/3},$
$$
\frac{\partial}{\partial A} \Bigl(\frac 1i I_{\mathbf{a}}(A) \Bigr)
= 2\int^{\beta}_0 \re w(A, \alpha+it)^{-1} dt =\sqrt{2} \int^{\beta}_0
\frac{\sqrt{\smash[b]{{g}_* +\sqrt{\smash[b]{{g}_*^2+{h}_*^2}}}}}{\sqrt{{g}_*^2+{h}_*^2}} dt >0,
$$
where
\begin{align*}
g_*={g}_*(t)&=\re w(A,\alpha+it)^2 =4(\alpha^3-3\alpha t^2)-A(\alpha^2-t^2)+1,
\\
h_*={h}_*(t)&=\im w(A,\alpha+it)^2 =4(-t^3+3\alpha^2 t)-2A\alpha t.
\end{align*}
This implies $I_{\mathbf{a}}(A) \in i \mathbb{R} \setminus \{0\}$ for $A < 3
\cdot 2^{2/3}.$
\par
The fact above combined with Proposition \ref{prop8.1} leads us to the
following.
%%%%%%%%%%%%%%%%%%%%%%%%%%%%%%
%%%%% Proposition 8.2 %%%%%%
%%%%%%%%%%%%%%%%%%%%%%%%%%
\begin{prop}\label{prop8.2}
If $\phi=0$, then the Boutroux equations $(\mathrm{BE})_0$ admit a unique
solution $A_0= 3\cdot 2^{2/3}.$ 
\end{prop}
%%%%%%%%%%%%%%%%%%%%%
%%%% Corollary 8.3 %%%%
%%%%%%%%%%%%%%%%%%%%%%%%%
\begin{cor}\label{cor8.3}
For every $A\in \mathbb{C}$, $(I_{\mathbf{a}}(A), I_{\mathbf{b}}(A))\not=(0,0).$
\end{cor}
%%%%%%%%%%%%%%%%%%%
\begin{proof}
If $I_{\mathbf{a}}(A)=0,$ then $A\in  \mathbb{R}$ by Proposition \ref{prop8.1}.
By Proposition \ref{prop8.2} and Example \ref{exa8.1}, $A=3\cdot 2^{2/3}$ and 
$I_{\mathbf{b}}(A)\not=0.$
\end{proof}
%%%%%%%%%%%%%%%%%%%%
%%%%%%%%%%%%%%%%%%%%%%%%%
%%%% Proposition 8.4 %%%%%%%%%%%%
%%%%%%%%%%%%%%%%%%%%%%%%%%%
\begin{prop}\label{prop8.4}
Suppose that, for $A_{\phi}$ solving $(\mathrm{BE})_{\phi}$ with $0<|\phi|\le
\pi/3$, the elliptic curve $\Pi_{A_{\phi}}$ degenerates. Then $\phi=\pi/3$ or
$-\pi/3$ and $A_{\pm \pi/3}=3\cdot 2^{2/3} e^{\mp 2\pi i/3}.$
\end{prop}
\begin{proof}
When $\Pi_{A_{\phi}}$ degenerates, $A_{\phi}= 3\cdot 2^{2/3} e^{2k \pi i/3},$
$k=0, \pm1.$ Suppose that $A_{\phi}=3\cdot 2^{2/3} e^{2\pi i/3},$ and that the
roots of $w(A_{\phi},z)^2$ are $z_0=z_1$ and $z_2\not=z_0, z_1$. Then
$$
\mathbb{R}\ni e^{i\phi} \int^{z_2}_{z_0} \frac 1{z^2} \sqrt{4z^3-A_{\phi}z^2 
+1}\, dz= e^{i(\phi-2\pi/3)} \int^{\zeta_2}_{\zeta_0} \frac{1}{\zeta^2}
\sqrt{4\zeta^3-3\cdot 2^{2/3} \zeta^2 +1}\, d\zeta \not=0
$$
with $\zeta_{0,2}=z_{0,2} e^{-2\pi i/3}\in \{ 2^{-1/3},-4^{-2/3}\},$ 
which implies $\phi=-\pi/3.$ 
Similarly, if $A_{\phi}=3\cdot 2^{2/3} e^{-2\pi i/3},$ then $\phi=\pi/3.$
\end{proof}
%%%%%%%%%%%%%%%%%%%%%%%%%%%%%%%%%
%%%% Proposition 8.5 %%%%%%%%%%%%%
%%%%%%%%%%%%%%%%%%%%%%%%
\begin{prop}\label{prop8.5}
If $\phi=\pm \pi/3,$ then the Boutroux equations $(\mathrm{BE})_{\pm \pi/3}$ 
admit a unique solution $A_{\pm \pi/3}=3\cdot 2^{2/3}e^{\mp 2\pi i/3}.$
\end{prop}
%%%%%%%%%%%%%%%%%%%%%%%%%%%
\begin{proof}
For $\phi=\pi/3$, $(\mathrm{BE})_{\pi/3}$ are equivalent to
$$
e^{\pi i/3} \int_{\mathbf{c}} \frac 1{z^2} \sqrt{4z^3-A_{\pi/3} z^2+1} \, dz
\in \mathbb{R}
$$
for every cycle $\mathbf{c}$ on $\Pi_{A_{\pi/3}}$, which is written as
$(\mathrm{BE})_0$ with $\phi=0$
$$
e^{\pi i} \int_{\mathbf{c} e^{2\pi i/3}} \frac 1{\zeta^2} \sqrt{4\zeta^3 
-e^{2\pi i/3}A_{\pi/3} \zeta^2  +1}\, d\zeta \in\mathbb{R}  \quad (z=e^{-2\pi i
/3}\zeta).
$$
Then by Proposition \ref{prop8.2}, $e^{2\pi i/3}A_{\pi/3}=3\cdot 2^{2/3}$
is a unique solution of $(\mathrm{BE})_{\pi/3}.$
\end{proof}
\par
The function $h(A)=I_{\mathbf{a}}(A)/I_{\mathbf{b}}(A)$ 
\cite[Appendix I]{Novokshenov-2} is useful in examining $A_{\phi}$.
%%%%%%%%%%%%%%%%%%%%%%%%%%%%%%%%%%%%%%%%
%%%%%% Proposition 8.6 %%%%%%%%%%%%%%%
\begin{prop}\label{prop8.6}
Suppose that $A\in \mathbb{C}.$ 
\par
$(1)$ If $A$ solves $(\mathrm{BE})_{\phi}$ for some $\phi \in \mathbb{R}$ and
$I_{\mathbf{b}}(A) \not= 0,$ then $h(A) \in \mathbb{R}.$
\par
$(2)$ If $h(A)\in \mathbb{R} \setminus \{0\}$, then, for some $\phi \in \mathbb
{R}$, $A$ solves $(\mathrm{BE})_{\phi}$. 
\end{prop}
%%%%%%%%%%%%%%%%%%
\begin{proof}
Suppose that $h(A)=\rho \in \mathbb{R},$ and write $I_{\mathbf{a}}(A)=u+iv,$
$I_{\mathbf{b}}(A)=U+iV$ with $u$, $v$, $U,$ $V\in \mathbb{R}$. Then $u=\rho U,$
$v=\rho V,$ and hence $v/u=V/U= -\tan \phi$ for some $\phi \in [-\pi/2,\pi/2].$ 
This implies $\im e^{i\phi}I_{\mathbf{a}}(A)=\im e^{i\phi}I_{\mathbf{b}}(A)=0.$
\end{proof}
%%%%%%%%%%%%%%%%%%%%%%%%%
%%%%% Proposition 8.7 %%%%%
%%%%%%%%%%%%%%%%%%%%%%%%%%%
\begin{prop}\label{prop8.7}
The set $\{ A\in \mathbb{C} \,|\, \text{$A$ solves $(\mathrm{BE})_{\phi}$ for
some $\phi \in\mathbb{R}$} \}$ is bounded.
\end{prop}
%%%%%%%%%%%%%%%
\begin{proof}
The roots of $w(A,z)$ are $z_0,$ $z_1 \sim \pm A^{-1/2},$ $z_2\sim
A/4$ if $A$ is large. Then
\begin{align*}
& \int^{z_0}_{z_2} \frac 1{z^2}{w(A,z)} dz \sim
\int^{A^{-1/2}}_{A/4} \frac 1{z^2} \sqrt{4z^3-Az^2+1}\,dz
 \sim  iA^{1/2} \int^{4A^{-3/2}}_1 \frac 1t \sqrt{1-t} \, dt
\\
& \sim i A^{1/2}(2 + \log(2A^{-3/2})) \sim  -\frac {3i}2 A^{1/2} \log A,
\end{align*}
and
$$
 \int^{z_1}_{z_0} \frac 1{z^2}{w(A,z)} dz \sim
\int^{A^{-1/2}}_{-A^{-1/2}} \frac 1{z^2} \sqrt{4z^3-Az^2+1}\,dz 
\sim  
A^{1/2} \int^{1}_{-1} \frac 1{t^2} \sqrt{1-t^2} \, dt \sim \pi A^{1/2}.
$$
This implies $h(A) \not\in \mathbb{R}$ if $A$ is sufficiently large, which
completes the proof.
\end{proof}
%%%%%%%%%%%%%%%%%%%%%%%
\par
The following fact is used in discussing solutions of (BE)$_{\phi}$.
\par
Let $0<|\phi|<\pi/3,$ and write
$$
I_{\mathbf{a}}(A)=u(A)+i v(A), \quad I_{\mathbf{b}}(A)=U(A)+i V(A).
$$
Note that $A$ solves (BE)$_{\phi}$ if and only if
$$
\im e^{i\phi}I_{\mathbf{a}}(A)=u(A)\sin\phi +v(A)\cos \phi=0, \quad
\im e^{i\phi}I_{\mathbf{b}}(A)=U(A)\sin\phi +V(A)\cos \phi=0,
$$
that is,
%%%%%%%%%%% (8.1) %%%%%%%%%%%%%%%%%%%%%%%%%
\begin{equation}\label{8.1}
u(A)\tan \phi +v(A)=0, \quad  U(A)\tan \phi +V(A)=0.
\end{equation}
Then, by the Cauchy-Riemann equations, the Jacobian for \eqref{8.1} with
$A=x+iy$ is written as
%%%%%%%%% (8.2) %%%%%%%%
\begin{align}\label{8.2}
\det J(\phi,A) &= \det \begin{pmatrix} u_x \tan\phi +v_x & u_y \tan\phi +v_y \\
U_x \tan\phi +V_x & U_y \tan\phi +V_y  \end{pmatrix}
\\
\notag
&= (1+\tan^2\phi) (v_xV_y -v_y V_x)
\\
\notag
&= -\frac 14(1+\tan^2\phi) |\Omega_{\mathbf{a}}(A)|^2 \im \frac
{\Omega_{\mathbf{b}}(A)}{\Omega_{\mathbf{a}}(A)},
\end{align}
where $\Omega_{\mathbf{a}}(A)$ and $\Omega_{\mathbf{b}}(A)$ are periods of the
elliptic curve $w(A,z).$
For $0<|\phi|<\pi/3$, $(d/dt)$\eqref{8.1} with $t=\tan\phi$ is written as
$$
J(\phi, A)\begin{pmatrix}  x'(t) \\ y'(t) \end{pmatrix} +
\begin{pmatrix} u(A) \\ U(A)  \end{pmatrix} \equiv \mathbf{o}. 
$$
Then we have
%%%%%%%%%%%%%%%%%%%%%%%%%
%%%%%% (8.3) %%%%%%%%%%%%%%%%%%%%%
\begin{equation}\label{8.3}
\text{$(x'(t),y'(t))\not=(0,0)$ and $ (d/d\phi)A=(x'(t)+i y'(t))\cos^{-2}\phi
\not=0 $} 
\end{equation}
for $0<|\phi|<\pi/3.$
%%%%%%%%%%%%%%%%%%%%%%%%%%%%
%%%%%%%%% Proposition 8.8 %%%%%
%%%%%%%%%%%%%%%%%%%%%%%%%%%%%%%
\begin{prop}\label{prop8.8}
Suppose that, for some $\phi_0$ such that $0<|\phi_0|<\pi/3$, $A_{\phi_0}$
solves $(\mathrm{BE})_{\phi_0}$. Then there exists a trajectory $T_0:$ 
$A=\chi(\phi_0,\phi)$ for $0<|\phi| <\pi/3$ with the properties$:$
\par
$(1)$ $\chi(\phi_0,\phi_0)=A_{\phi_0};$ 
\par
$(2)$ for each $\phi,$ $A=\chi(\phi_0,\phi)$ solves $(\mathrm{BE})_{\phi};$
\par
$(3)$ $\chi(\phi_0,\phi)$ is smooth for $0<|\phi|<\pi/3.$
\end{prop}
%%%%%%%%%%%%%%%%%%%%%
\begin{proof}
Since the Jacobian \eqref{8.2} satisfies $\det J(\phi_0, A_{\phi_0}) \in
\mathbb{R}\setminus \{0\},$ there exists a local trajectory 
$A=\chi_{\mathrm{loc}}(\phi_0,\phi)$ having the properties $(1)$, $(2)$ and 
$(3)$ above for $|\phi-\phi_0|<\delta$, where $\delta$ is sufficiently small. 
Since \eqref{8.2} is in 
$\mathbb{R} \setminus \{0\}$ for $0<|\phi|<\pi/3,$ $\chi_{\mathrm{loc}}(\phi_0,
\phi)$ may be extended to the interval $0<|\phi|<\pi/3.$
\end{proof}
%%%%%%%%%%%%%%%%%%%%%%%%%%%%%%%%%%%%%%%%%%%
%%%%%% Proposition 8.9 %%%%%%%%%%%
%%%%%%%%%%%%%%%%%%%%%%%%%%%%%%
\begin{prop}\label{prop8.9}
The trajectory $T_0:$ $A=\chi(\phi_0,\phi)$ given above
may be extended to $|\phi|\le \pi/3$ such that $\chi(\phi_0,\phi)$ is 
continuous in $\phi$ and that $\chi(\phi_0,0)=A_0=3\cdot 2^{2/3},$ 
$\chi(\phi_0,\pm \pi/3)=A_{\pm \pi/3}= 3\cdot 2^{2/3}e^{\mp 2\pi i/3}.$
\end{prop}
%%%%%%%%%%%%%%%%%%%
\begin{proof}

To show that $\chi(\phi_0,\phi)\to A_0$ as $\phi \to 0$, suppose
to the contrary that there exists a sequence $\{\phi_{\nu}\}$ such that
$\phi_{\nu}\to 0$ and that $\chi(\phi_0,\phi_{\nu})$ does not converge
to $A_0.$ There exists a subsequence $\{ \phi_{\nu(n)} \}$ such that
$\chi(\phi_0,\phi_{\nu(n)}) \to A'_0$ for some $A'_0\not=A_0,$ since, 
by Proposition \ref{prop8.7}, 
the trajectory $T_0$ for $0<|\phi|<\pi/3$ is bounded. Then we have
$\im I_{\mathbf{a}}(A'_0)=\im I_{\mathbf{b}}(A'_0)=0,$ which contradicts
the uniqueness of a solution of $(\mathrm{BE})_0.$ Hence $\chi(\phi_0,\phi)$
is extended to $\phi=0$ and is continuous in a neighbourhood of $\phi=0.$
By Proposition \ref{prop8.5}, it is possible to extend $\chi(\phi_0,\phi)$ 
to $\phi=\pm\pi/3$ by the same argument. 
\end{proof}
%%%%%%%%%%%%%%%%%%%%%%%%%%%%%%%%%%%%%%
%%%%% Lemma 8.10 %%%%%%%%%%%
\begin{lem}\label{lem8.10}
$h'(A)=-6\pi i I_{\mathbf{b}}(A)^{-2}.$
\end{lem}
%%%%%%%%%%%%%%%%%%%%%%%
\begin{proof}
From $I'_{\mathbf{a},\, \mathbf{b}}(A)=-\Omega_{\mathbf{a},\, \mathbf{b}}/2$
and Lemma \ref{lem6.9}, the conclusion follows.
\end{proof}
%%%%%%%%%%%%%%%%%%%%%%%%%%
%%%% Corollary 8.11 %%%%%%
%%%%%%%%%%%%%%%%%%%%%%%%%
\begin{cor}\label{cor8.11}
If $I_{\mathbf{b}}(A)\not=0, \infty$, then $h(A)$ is conformal around $A$.
\end{cor}
\par
By Example \ref{exa8.1}, $h(A)$ is conformal at $A_0=3\cdot 2^{2/3}$ and 
$h(A_0)=0.$ By Lemma \ref{lem8.10},
$$
h(A)=h'(A_0)(A-A_0)+o(A-A_0)=-\frac{\pi i}{2^{5/3}\cdot 3^2}(A-A_0)
+o(A-A_0)
$$
around $A=A_0.$ By Proposition \ref{prop8.6}, for a
sufficiently small $\varepsilon >0,$ the inverse image of $(-\varepsilon,0) 
\cup (0,\varepsilon)$ under $h(A)$ is a trajectory $T_{0-} \cup T_{0+}$:
$A= \chi^{\pm}_0(\phi)$ solving (BE)$_{\phi}$, and is expressed as 
%%%%%%% (8.4) %%%%%%%%%%%%%%%%%%
\begin{equation}\label{8.4}
\chi^{\pm}_0(\phi) = A_0 +\gamma_0(\phi) i + o(\gamma_0(\phi)), 
\end{equation}
near $\phi=0$, where $\gamma_0(\phi) \in\mathbb{R}$ is continuous in $\phi$
and $\gamma_0(0)=0.$
%%%%%%%%%%%%%%%%%%%%%%%%%
\par
The fact above implies that there exists a local 
trajectory solving (BE)$_{\phi}$ near $\phi =0.$ From this, a trajectory
for $|\phi|\le \pi/3$ as in Proposition \ref{prop8.9} may be obtained. 
Furthermore, if two trajectories $\chi_1(\phi)$ and $\chi_2(\phi)$ solving
(BE)$_{\phi}$ satisfy $\chi_1(\phi_0)=\chi_2(\phi_0)$ for some $\phi_0$ such
that $0<|\phi_0|<\pi/3,$ then \eqref{8.2} or the conformality of $h(A)$ at
$A=A_0$ implies $\chi_1(\phi)\equiv \chi_2(\phi).$ Thus we have the following.
%%%%%%%%%%%%%%%%%%%%%%%%%
%%%%% Proposition 8.12 %%%%%%%%%%%%%%%
%%%%%%%%%%%%%%%%%%%%%%%%%%%%%%
\begin{prop}\label{prop8.12}
There exists a trajectory $A=A_{\phi}$ for $|\phi| \le \pi/3$ with the 
properties$:$
\par
$(1)$ for each $\phi$, $A_{\phi}$ is a unique solution of $(\mathrm{BE})_{\phi};$
\par
$(2)$ $A_{\phi}$ is smooth in $\phi$ for $0<|\phi|<\pi/3$ and continuous in
$\phi$ for $|\phi|\le \pi/3.$
\end{prop} 
\par
For any cycle $\mathbf{c}$, it is easy to see that
\begin{align*}
e^{i\phi}\int_{\mathbf{c}}\frac 1{z^2}w(A_{\phi},z)dz&=e^{i(\phi\mp 2\pi /3)}
\int_{e^{\mp 2\pi i/3}\mathbf{c}} \frac 1{\zeta^2} w(e^{\mp 2\pi i/3}A_{\phi},
\zeta) d\zeta,
\\
e^{i\phi}\int_{\mathbf{c}}\frac 1{z^2}w(A_{\phi},z)dz&= - e^{i(\phi+ \pi )}
\int_{\mathbf{c}} \frac 1{\zeta^2} w(A_{\phi},\zeta) d\zeta,
\end{align*}
which yields the following.
%%%%%%%%%%%%%%%%%%%%%%%%%%%%%%%%%%
%%%% Proposition 8.13 %%%%%%%%%%%%%
%%%%%%%%%%%%%%%%%%%%%%%%%%%%
\begin{prop}\label{prop8.13}
Set $A_{\phi \mp 2\pi /3}=e^{\mp 2\pi i/3} A_{\phi}$ for $|\phi| \le \pi/3.$
Then for $|\phi|\le \pi,$ $A_{\phi}$ is a unique solution of $(\mathrm{BE})
_{\phi}$. Furthermore $A_{\phi+\pi}=A_{\phi},$ $A_{-\phi}=\overline{A_{\phi}}.$ 
\end{prop}
Let us examine the properties of $A_{\phi}$ in more detail. Note that the 
trajectory $A=A_{\phi}=x+iy$ for $|\phi|<\pi /3$ satisfies 
$h(A_{\phi})\in\mathbb{R}.$ Then, by \eqref{8.3}, 
$$
\frac{d}{dt} h(A_{\phi}) =(x'(t)+iy'(t))(-6\pi i) I_{\mathbf{b}}(A_{\phi})^{-2}
\in \mathbb{R} \setminus \{0\}
$$ 
with $t=\tan\phi$ for $0<|\phi|<\pi/3.$ Setting $I_{\mathbf{b}}(A_{\phi})^{-1}
=P+i Q,$ we have
$$
-\frac 1{6\pi}\im \frac d{dt} h(A_{\phi}) =x'(t)(P^2-Q^2) -2y'(t) PQ=0.
$$
If $x'(t_0)=0$ for some $t_0 =\tan(\phi_0)\not=0, \pm \infty$, then $PQ =0,$
and hence $I_{\mathbf{b}}(A_{\phi_0}) \in i\mathbb{R} \setminus \{0\}$ or
$\mathbb{R} \setminus \{0\}.$ This is impossible for $0<|\phi|<\pi/3,$ which
implies $x'(t)\not=0$ for $0<|\phi|<\pi/3.$ Since $A_{\pm \pi/3}= A_0
e^{\mp 2\pi i/3},$ we have $x'(t)<0$ for $0<\phi <\pi/3$ and $x'(t)>0$ for
$-\pi/3 <\phi <0.$ If $y'(t_0)=0$ for some $t_0$ with $0<|\phi_0|<\pi/3$, then
$P^2-Q^2=0,$ i.e. $I_{\mathbf{b}}(A_{\phi_0})^{-1}=P(1 \pm i),$ 
implying $\phi_0=\pm \pi/4.$ Note that $-P<Q<0$ for $-\pi/4 <\phi<0$ 
and that $0<Q<P$ for $0<\phi<\pi/4.$ It follows that $y'(t)<0$ for 
$0<|\phi|<\pi/4.$
%%%%%%%%%%%%%%%%%%%%%%%%%%%%%%%%%%%%%%%
%%%%%%% Proposition 8.14 %%%%%%%%%%%%%%%%%% 
%%%%%%%%%%%%%%%%%%%%%%%%%%%
\begin{prop}\label{prop8.14}
The trajectory $A_{\phi}=x(t)+iy(t)$ with $t=\tan\phi$ has the properties$:$
\par
$(1)$ $x'(t)>0$ for $-\pi/3<\phi<0,$ and $x'(t)<0$ for $0<\phi<\pi/3;$
\par
$(2)$ $y'(t)<0$ for $0<|\phi|<\pi/4$ and $y'(\tan(\pm \pi/4))=0.$
\end{prop} 
Thus we have
%%%%%%%%%%%%%%%%%%%%%%%%%%%%%%
%%%%%% Proposition 8.15 %%%%%%%%%%%%%%%
%%%%%%%%%%%%%%%%%%%%%
\begin{prop}\label{prop8.15}
For $\phi \in \mathbb{R}$ there exists a trajectory $A=A_{\phi}$ with the
properties$:$
\par
$(1)$ for each $\phi$, $A_{\phi}$ is a unique solution of $(\mathrm{BE})_{\phi};$
\par
$(2)$ $A_{\phi \pm 2\pi /3}=e^{\pm 2\pi i/3}A_{\phi},$ $A_{\phi+\pi}=A_{\phi},$
$A_{-\phi}=\overline{A_{\phi}};$
\par
$(3)$ $A_0=3\cdot 2^{2/3},$ $A_{\pm \pi/3}=3\cdot 2^{2/3} e^{\mp 2\pi i/3};$
\par
$(4)$ $A_{\phi}$ is continuous in $\phi\in \mathbb{R}$, and smooth in $\phi \in
\mathbb{R}\setminus \{ m\pi/3 \,|\, m\in \mathbb{Z} \}.$
\end{prop}
The trajectory of $A_{\phi}$ is roughly drawn in Figure \ref{trajectory}.
%%%%%%%%%%%%%%%%%%%%%%%%%%%%%%
%%%%%%%%%%%%%%%%%%%%%%%%%%%%%%%%%%%%%%
%%%%%%% Figure %%%%%%%%%%%%%%%%%%%%
%%%%%%%%%%%%%%%%%%%%%%%%%%%%%%%%%%
%%%%%%%%%%%%%%%%%%%%%%%%%%%%
{\small
\begin{figure}[htb]
\begin{center}
\unitlength=0.85mm
%%%%%%%%%%%%%%%%%%%%%%%%%%%%%%%%%%
%%%%%%%%%%%%%%%%%%%%%%%%%%%%%%%
%%%%%%%%%%%%%%%%%%%%%%%%%%%%%%%
\begin{picture}(80,60)(-40,-30)
\put(0,0){\circle*{1.5}}
%\put(-20,0){\circle*{1}}
\put(20,0){\circle*{1.5}}
\put(-10,17.3){\circle*{1.5}}
\put(-10,-17.3){\circle*{1.5}}

\put(-40,0){\line(1,0){80}}

\qbezier (22.6,-8) (22,-14) (19.7, -18)
\put(19.7,-18){\vector(-1,-1){0}}

\thicklines
\qbezier (20,0) (20,34.6) (-10, 17.3)
\qbezier (20,0) (20,-34.6) (-10, -17.3)
\qbezier (-10,17.3) (-40,0) (-10,-17.3)

{\tiny
\put(-32,-22){\makebox{$A_{-2\pi/3}, A_{\pi/3}$}}
\put(-32,21){\makebox{$A_{-\pi/3}, A_{2\pi/3}$}}
\put(22,-4){\makebox{$A_0, A_{\pm \pi}$}}
\put(1,-4){\makebox{$0$}}
}
%\put(-20,-34){\makebox{(b) $-\pi/3 < \phi < 0$}}

\end{picture}
%%%%%%%%%%%%%%%%%%%%%%%%%%%%%%%
%%%%%%%%%%%%%%%%%%%%%%%%%%%%%%%

%%%%%%%%%%%%%%%%%%%%%%%%%%%%%%%%%%%%%%%%%
\end{center}
\caption{Trajectory of $A_{\phi}$ for $|\phi|\le \pi$}
\label{trajectory}
\end{figure}
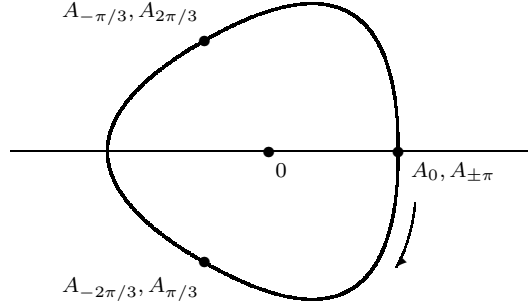
}
%%%%%%%%%%%%%%%%%%%%%%%%%%%%%%%
\par
By Proposition \ref{prop8.14}, when $|\phi|$ is sufficiently small, the location
of the turning points may be examined. Small variance
of $A_{\phi}$ around $\phi=0$ is given by $A_{\phi}=A_0+\delta_{\phi}$ with
$\delta_{\phi}$ having the properties:
(1) $\delta_{\phi}\to 0$ as $\phi\to 0;$ (2) $\re\delta_{\phi} \le 0$;
(3) $\im \delta_{\phi} \ge 0$ if $\phi\le 0$ and $\im \delta_{\phi}\le 0$ 
if $\phi \ge 0$.
Then the roots $z_0, z_1=2^{-1/3}$ and $z_2= -4^{-2/3}$ of $w(A_{0},z)^2$ 
vary in such a way that
$$
z_0= 2^{-1/3} +\varrho +O(\varrho^2), \quad 
z_1= 2^{-1/3} -\varrho +O(\varrho^2), \quad z_2= -4^{-2/3}+O(\varrho^2)
$$
with $\varrho=2^{-2/3}\cdot 3^{-1/2}\delta_{\phi}^{1/2}.$ 
Indeed, insertion of $z_0=2^{-1/3}+\varrho_+,$ $z_1=2^{-1/3}+\varrho_-,$
$z_2=-4^{-2/3}+\varrho_2$ into 
$z_0+z_1+z_2=A_{\phi}/4,$ $z_1z_2+z_2z_0+z_0z_1=0,$ $z_0z_1z_2=-1/4$
yields
$$
p+\varrho_2 =\delta_{\phi}/4, \quad p+4\varrho_2 +2^{4/3}q =O(p\varrho_2),
\quad p-2\varrho_2 +2^{1/3}q= O(|\varrho_2|(|p|+|q|))
$$
with $p=\varrho_++\varrho_-,$ $q=\varrho_+\varrho_-,$ from which the estimates
above follow.
Thus we have the following.
%%%%%%%%%%%%%%%%%%%%%%%%%%%%%%%%
%%%%% Proposition 8.16 %%%%%%%%%%%
\begin{prop}\label{prop8.16}
If $|\phi|$ is sufficiently small, the turning points $\lambda_k$ 
and $z_k=\lambda^{-2}_k$ $(k=0,1,2)$ are represented as  
\begin{align*}
&\lambda_0= 2^{1/6} -\varepsilon_{\phi} e^{i\theta_{\phi}} +O(\varepsilon_{\phi}
^2), \quad
\lambda_1= 2^{1/6} +\varepsilon_{\phi} e^{i\theta_{\phi}} +O(\varepsilon_{\phi}
^2), \quad
\lambda_2= 2^{2/3}i  +O(\varepsilon_{\phi}^2), 
\\
&z_0= 2^{-1/3} + 2^{1/2}\varepsilon_{\phi} e^{i\theta_{\phi}} +O(\varepsilon_{\phi}^2),
 \quad
z_1= 2^{-1/3} - 2^{1/2} \varepsilon_{\phi} e^{i\theta_{\phi}} +O(\varepsilon_{\phi}^2), 
\\
& z_2=- 4^{-2/3}  +O(\varepsilon_{\phi}^2). 
\end{align*}
Here $\varepsilon_{\phi}$ and $\theta_{\phi}$ fulfil 
\par
$(1)$ $\varepsilon_{\phi}>0$ and $\varepsilon_{\phi}\to 0$ as $\phi\to 0;$ 
and 
\par
$(2)$ $\theta_{\phi} \to \pi/4$ as $\phi\to 0$ with $\phi<0$, and 
$\theta_{\phi} \to -\pi/4$ as $\phi\to 0$ with $\phi>0$.  
\end{prop}  
%%%%%%%%%%%%%%%%%%%%%%%%%%%%%%%%%%%%%%%%%%%%%%
%%%%%%%%%%%%%%%%%%%%%%%%%%%%%%%%%%%%%%%%%%%%%%%

%%%% References %%%%%%%%%%%%%%%%%%%%%%%%%
%%%%%%%%%%%%%%%%%%%%%%%%%

\end{document}